\documentclass{article}

\usepackage{arxiv}
\usepackage[utf8]{inputenc} 
\usepackage[T1]{fontenc}    
\usepackage{hyperref}       
\usepackage{xcolor}
\definecolor{crimson}{RGB}{153, 0, 51}
\hypersetup{colorlinks, citecolor=red, linkcolor=blue, urlcolor=crimson}
\usepackage{amsfonts}       
\usepackage{nicefrac}       
\usepackage{microtype}      
\usepackage{graphicx}
\usepackage[square,numbers]{natbib}

\usepackage{array, multirow}
\usepackage{makecell}
\usepackage{threeparttable}
\usepackage{arydshln}

\usepackage{enumitem}
\usepackage{amsmath, amsthm}
\usepackage{algorithm}
\usepackage{algpseudocode}
\usepackage[caption=false]{subfig}

\newtheorem{theorem}{Theorem}[section]
\newtheorem{proposition}[theorem]{Proposition}
\newtheorem{lemma}[theorem]{Lemma}
\newtheorem{corollary}[theorem]{Corollary}
\theoremstyle{definition}
\newtheorem{definition}[theorem]{Definition}
\newtheorem{assumption}[theorem]{Assumption}
\theoremstyle{remark}
\newtheorem{remark}{Remark}

\def\B{\mathcal{B}}
\def\Dnorm{D^*_{\text{normalized}}}
\def\E{\mathcal{E}}
\def\G{\mathcal{G}}
\def\I{\boldsymbol{I}}
\def\tilL{\Tilde{L}}
\def\N{\mathbb{N}}
\def\O{\mathcal{O}}
\def\R{\mathbb{R}}
\def\Rnorm{R^*_{\text{normalized}}}
\def\S{\mathbb{S}}
\def\V{\mathcal{V}}
\def\vr{\mathcal{V}_\mathcal{R}}
\def\vb{\mathcal{V}_\mathcal{B}}
\def\W{\boldsymbol{W}}
\def\X{\mathcal{X}}
\def\Y{\mathcal{Y}}
\def\Z{\mathbb{Z}}

\def\c{\boldsymbol{c}}
\def\g{\boldsymbol{g}}

\def\rc{r_{\text{c}}}
\def\u{\boldsymbol{u}}
\def\v{\boldsymbol{v}}
\def\x{\boldsymbol{x}}
\def\xc{\boldsymbol{x}_{\text{c}}}
\def\y{\boldsymbol{y}}
\def\z{\boldsymbol{z}}

\def\tilalph{\Tilde{\alpha}}
\def\tilmu{\Tilde{\mu}}
\def\Phil{\boldsymbol{\Phi}^{(\ell)}}

\makeatletter
\renewcommand{\ALG@name}{Algorithmic Framework}
\makeatother

\title{On the Geometric Convergence of Byzantine-Resilient Distributed Optimization Algorithms}
\date{} 

\author{ \href{https://orcid.org/0000-0002-0769-1399}{\includegraphics[scale=0.06]{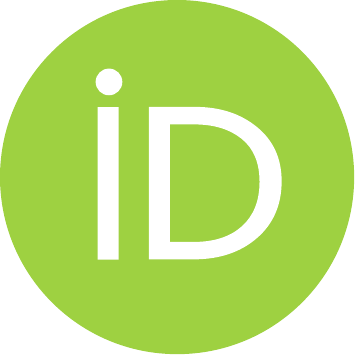}}\hspace{1mm}Kananart~Kuwaranancharoen \\
	School of Electrical and Computer Engineering\\
	Purdue University\\
	West Lafayette, IN 47907 \\
        \href{kkuwaran@alumni.purdue.edu}{\texttt{kkuwaran@alumni.purdue.edu}} \\
	\And
	\href{https://orcid.org/0000-0002-5390-2505}{\includegraphics[scale=0.06]{orcid.pdf}}\hspace{1mm}Shreyas~Sundaram \\
	School of Electrical and Computer Engineering\\
	Purdue University\\
	West Lafayette, IN 47907 \\
	\href{mailto:sundara2@purdue.edu}{\texttt{sundara2@purdue.edu}} \\
}

\renewcommand{\shorttitle}{Geometric Convergence of Resilient Distributed Algorithms}

\hypersetup{
pdftitle={On the Geometric Convergence of Byzantine-Resilient Distributed Optimization Algorithms},
pdfauthor={Kananart~Kuwaranancharoen, Shreyas~Sundaram},
pdfkeywords={Consensus Algorithm, Convex Optimization, Distributed Algorithms, Distributed Optimization, Fault Tolerant Systems, Linear Convergence, Multi-Agent Systems, Network Security},
}

\begin{document}
\maketitle

\begin{abstract}
The problem of designing distributed optimization algorithms that are resilient to Byzantine adversaries has received significant attention. For the Byzantine-resilient distributed optimization problem, the goal is to (approximately) minimize the average of the local cost functions held by the regular (non adversarial) agents in the network. In this paper, we provide a general algorithmic framework for Byzantine-resilient distributed optimization which includes some state-of-the-art algorithms as special cases. We analyze the convergence of algorithms within the framework, and derive a geometric rate of convergence of all regular agents to a ball around the optimal solution (whose size we characterize). Furthermore, we show that approximate consensus can be achieved geometrically fast under some minimal conditions. Our analysis provides insights into the relationship among the convergence region, distance between regular agents' values, step-size, and properties of the agents' functions for Byzantine-resilient distributed optimization.
\end{abstract}


\section{Introduction}  
\label{sec: intro}

Distributed optimization problems pertain to a setting where each node in a network has a local cost function, and the goal is for all agents in the network to agree on a minimizer of the average of the local cost functions. In the distributed optimization literature, there are two main paradigms: client-server and peer-to-peer. Motivated by settings where the client-server paradigm may suffer from a single point of failure or communication bottleneck, there is a growing amount of work on the peer-to-peer setting where the agents in the network are required to send and receive information only from their neighbors. A variety of algorithms have been proposed to solve such problems in peer-to-peer architectures (e.g., see \cite{nedic2009distributed, duchi2011dual, shi2014linear, nedic2017achieving, pu2020push}). The works \cite{nedic2018network, yang2019survey, xin2020general} summarize the recent advances in the field of (peer-to-peer) distributed optimization. 

These aforementioned works typically make the assumption that all agents are trustworthy and cooperative (i.e., they follow the prescribed protocol). However, it has been shown that the regular agents fail to reach an optimal solution even if a single misbehaving (or ``Byzantine") agent is present \cite{sundaram2018distributed, su2015byzantine}. Thus, designing distributed optimization algorithms that allow all the regular agents' states in the network to stay close to the minimizer of the sum of regular agents' functions regardless of the adversaries' actions has become a prevailing problem. Nevertheless, compared to the client-server setting, there have been only a few works on Byzantine-resilient algorithms in the peer-to-peer setup.

{\bf Contributions.} In this work, we consider Byzantine-resilient (peer-to-peer) distributed deterministic optimization problems. Our contributions are as follows. 
\begin{enumerate}[label=(\roman*)]
    \item We introduce an algorithmic framework called R{\scriptsize{ED}}G{\scriptsize{RAF}}, a generalization of BRIDGE in \cite{fang2022bridge}, which includes some state-of-the-art Byzantine-resilient distributed optimization algorithms as special cases.
    
    \item We propose a novel contraction property that provides a general method for proving geometric convergence of algorithms in R{\scriptsize{ED}}G{\scriptsize{RAF}}. To the best of our knowledge, this is the first work to demonstrate a geometric rate of convergence of all regular agents' states to a ball containing the true minimizer for resilient algorithms under strong convexity. We also explicitly characterize both the convergence rate and the size of the convergence region.

    \item We introduce a novel mixing dynamics property used to derive approximate consensus results for algorithms in R{\scriptsize{ED}}G{\scriptsize{RAF}}, explicitly characterizing both the convergence rate and final consensus diameter.
    
    \item Using our framework, we analyze the contraction and mixing dynamics properties of some state-of-the-art algorithms, leading to convergence and consensus results for each algorithm. Our work is the first to show that these algorithms satisfy such properties.
    
    \item We demonstrate and compare the performance of the algorithms through numerical simulations to corroborate the theoretical results for convergence and approximate consensus.
\end{enumerate}
\section{Related Work}  
\label{sec: related work}

The survey paper \cite{yang2020adversary} provides an overview of some Byzan\-tine-resilient algorithms for both the client-server and peer-to-peer paradigms. Since we are focusing on resilient algorithms for peer-to-peer settings, we discuss the following research papers attempting to solve such problems. 
Papers \cite{su2015byzantine, sundaram2018distributed, su2020byzantine} show that using the \textit{distributed gradient descent (DGD)} equipped with a \textit{trimmed mean filter} guarantees convergence to the convex hull of the local minimizers under scalar-valued objective functions. Adopting a similar algorithm, \cite{fu2021resilient} gives the same guarantee for scalar-valued problems under \textit{deception attacks}. The work \cite{zhao2019resilient} also considers the scalar version of such problems but relies on \textit{trusted agents} which cannot be compromised by adversarial attacks. 
To tackle vector-valued problems, \cite{yang2019byrdie} proposes ByRDiE, a coordinate descent method for machine learning problems leveraging the algorithm in \cite{su2015byzantine}, while \cite{fang2022bridge} presents BRIDGE, an algorithm framework for Byzantine-resilient distributed optimization problems. Even though \cite{yang2019byrdie} and \cite{fang2022bridge} show the convergence to the minimizer with high probability (for certain specific algorithms), they require that the training data are independent and identically distributed (i.i.d.) across the agents in the network. While resilient algorithms with the trimmed mean filter are widely used, e.g., \cite{sundaram2018distributed, su2015byzantine, su2020byzantine, fu2021resilient, zhao2019resilient, fang2022bridge}, the convergence analysis for multivariate functions under general assumptions is still lacking.
The work \cite{ravi2019detection} proposes decentralized robust subgradient push, a resilient algorithm based on a subgradient push method \cite{nedic2014distributed} equipped with a \textit{maliciousness score} for detecting adversaries. However, the work requires that the regular agents' functions have common statistical characteristics, and does not provide any guarantees on the proposed algorithm. 
Papers \cite{kuwaran2020byzantine_ACC, kuwaran2024scalable} introduce SDMMFD and SDFD, resilient algorithms for deterministic distributed convex optimization problems with multi-dimensional functions. These algorithms have an asymptotic convergence guarantee to a proximity of the true minimum. However, they do not provide the convergence rate for the proposed algorithms. In contrast, the work \cite{gupta2021byzantine} offers an algorithm with provable \textit{exact fault-tolerance}, but it relies on redundancy among the local functions and requires the underlying communication network to be complete.

For distributed stochastic optimization problems, \cite{peng2021byzantine} introduces a resilient algorithm based on a total variation norm penalty motivated from \cite{ben2015robust}. The recent paper \cite{guo2021byzantine} also considers stochastic problems, and proposes an algorithm utilizing a distance-based filter and objective value-based filter, but does not provide any performance guarantees. The recent paper \cite{elkordy2022basil} which also considers stochastic problems especially for machine learning, proposes a validation-based algorithm for both i.i.d. and non-i.i.d. settings. In particular, the work theoretically shows a convergence guarantee for the proposed algorithm under convex loss functions and i.i.d. data. The recent papers \cite{wu2023byzantine} and \cite{he2023byzantinerobust} propose algorithms which converge to a neighborhood of a stationary point for distributed stochastic non-convex optimization problems.

As outlined in our contributions section, this paper addresses gaps in the existing literature by demonstrating the geometric rate of convergence of all regular agents' states to a ball containing the true minimizer for a class of resilient algorithms under the strong convexity assumption. We also explicitly characterize the size of this ball. Consequently, our work provides a convergence analysis under mild assumptions for four resilient algorithms: (i) algorithms employing the trimmed mean filter (referred to as CWTM), as studied in \cite{sundaram2018distributed, su2015byzantine, su2020byzantine, fu2021resilient, zhao2019resilient, fang2022bridge}; (ii) SDMMFD, as considered in \cite{kuwaran2020byzantine_ACC, kuwaran2024scalable}; (iii) SDFD, as introduced in \cite{kuwaran2024scalable}; and (iv) a resilient algorithm based on resilient vector consensus (referred to as RVO) \cite{park2017fault, abbas2022resilient}.
For detailed descriptions of each algorithm, please refer to subsection~\ref{subsec: existing algs}. In this section, we provide a comparative summary of assumptions and theoretical results between our work and previous studies, as outlined in Table~\ref{table: algs-comparison}\footnote{We use the terms ``geometric convergence" and ``linear convergence" interchangeably.}.

The table demonstrates that prior works on Byzantine-resilient distributed optimization mostly considered the decreasing step-size regime, typically under convex local functions, leading to sublinear convergence at best. In contrast, our work explores constant step-sizes under strongly convex local functions, which enables us to achieve a linear rate of convergence for both optimality distance and consensus distance. However, this comes at the expense of obtaining an approximate consensus, rather than an exact consensus.

\begin{table}[t]
\caption{Comparison of Assumptions and Theoretical Results between Our Work and Previous Studies}
\label{table: algs-comparison}
\begin{threeparttable}
\resizebox{\textwidth}{!}{%
    \begin{tabular}{| c || c | c | c | c | c | c | c |} \hline
    \multirow{4}{*}{\bf \makecell{Assumptions \\ and Results}} 
     & \multicolumn{7}{c|}{\bf Byzantine-Resilient Distributed Optimization Algorithm} \\ \cline{2-8}
     & \makecell{CWTM \\ \cite{su2015byzantine, su2020byzantine}} & \makecell{CWTM \\ \cite{zhao2019resilient}} & \makecell{CWTM \\ \cite{sundaram2018distributed}} & \makecell{CWTM \\ \cite{fang2022bridge}} & \makecell{SDMMFD \\ \cite{kuwaran2020byzantine_ACC, kuwaran2024scalable}} & \makecell{SDFD \\ \cite{kuwaran2024scalable}} & 
    \makecell{CWTM, RVO, \\ SDMMFD, SDFD \\ (our work)} \\ \hline \hline
    Dimension\tnote{1} & single & single & single & multi & multi & multi & multi \\ \hdashline
    Convexity & convex & convex & convex & \makecell{strongly \\ convex} & convex & convex & \makecell{strongly \\ convex}  \\ \hdashline
    Gradient  & \makecell{bounded and \\ Lipschitz} & bounded & bounded & \makecell{bounded and \\ Lipschitz} & bounded & bounded & Lipschitz  \\ \hdashline
    Network & \makecell{nonempty source \\ component in \\ reduced graphs} & \makecell{connected \\ dominating \\ set (CDS)} & \makecell{robust \\ graph} & \makecell{nonempty source \\ component in \\ reduced graphs} & \makecell{robust \\ graph} & \makecell{robust \\ graph} & \makecell{robust \\ graph} \\ \hdashline
    Step-size & decreasing & decreasing & decreasing & decreasing & decreasing & decreasing & constant \\ \hdashline
    Addition & - & - & - & \makecell{i.i.d. training \\ data} & - & - & - \\ \hline
    Consensus & \makecell{exact, \\ asymptotic} & \makecell{exact, \\ asymptotic} & \makecell{exact, \\ sublinear rate} & \makecell{exact, \\ sublinear rate} & \makecell{exact, \\ asymptotic} & - & \makecell{approximate, \\ linear rate} \\ \hdashline
    Convergence & \makecell{neighborhood, \\ asymptotic} & \makecell{neighborhood, \\ asymptotic} & \makecell{neighborhood, \\ asymptotic} & \makecell{minimium, \\ sublinear} & \makecell{neighborhood, \\ asymptotic} & \makecell{neighborhood, \\ asymptotic} & \makecell{neighborhood, \\ linear} \\ \hline
    \end{tabular}}
    \begin{tablenotes}
       \item [1] \resizebox{0.95\textwidth}{!}{This pertains to the suitability of each algorithm in relation to the number of independent variables within local functions.}
     \end{tablenotes}
\end{threeparttable}
\end{table}
\section{Background and Problem Formulation} 
\label{sec: notation and formulation}

\subsection{Background}
\label{subsec: notation}

Let $\N$, $\Z$ and $\R$ denote the set of natural numbers (including zero), integers, and real numbers, respectively. Let $\Z_{+}$, $\R_{\geq 0}$ and $\R_{> 0}$ denote the set of positive integers, non-negative real numbers, and positive real numbers, respectively. For convenience, for an integer $N \in \Z_{+}$, we define $[N] := \{ 1, 2, \ldots, N \}$. The cardinality of a set is denoted by $|\cdot|$.
Given positive integers $F \in \Z_+$ and $s \geq F$, and a set of scalars $\mathcal{X} = \{ x_1, x_2, \ldots, x_s \}$, define $M_F(\mathcal{X})$ and $m_F(\mathcal{X})$ to be the $F$-th largest element and $F$-th smallest element, respectively, of the set $\mathcal{X}$.

\subsubsection{Linear Algebra}
Vectors are taken to be column vectors, unless otherwise noted. We use $x^{(\ell)}$ to represent the $\ell$-th component of a vector $\x$. The Euclidean norm on $\R^d$ is denoted by $\| \cdot \|$. We use $\boldsymbol{1}$ and $\boldsymbol{I}$ to denote the vector of all ones and the identity matrix, respectively, with appropriate dimensions. We denote by $\langle \u, \v \rangle$ the Euclidean inner product of vectors $\u$ and $\v$, i.e., $\langle \u, \v \rangle = \u^T \v$ and by $\angle (\u, \v)$ the angle between vectors $\u$ and $\v$, i.e., $\angle (\u, \v) = \arccos \big( \frac{ \langle \u, \v \rangle }{\Vert \u \Vert \Vert \v \Vert}  \big)$. 
The Euclidean ball in $\R^d$ with center at $\x_0 \in \R^d$ and radius $r \in \R_{\geq 0}$ is denoted by $\mathcal{B}(\x_0, r) := \{ \x \in \R^d : \| \x - \x_0 \| \leq r \} $. For $N \in \Z_+$, a matrix $\W \in \R^{N \times N}$ is (row-)stochastic if $\W \boldsymbol{1} = \boldsymbol{1}$ and $w_{ij} \geq 0$ for all $i, j \in [N]$. For $N \in \Z_+$, we use $\S^N$ to denote the set of all $N \times N$ (row-)stochastic matrices.

\subsubsection{Functions Properties}
For a differentiable function $f: \R^d \to \R$ and a point $\x \in \R^d$, define the vector $\nabla f (\x) \in \R^d$ to be the gradient of $f$ at point $\x$.
\begin{definition}[strongly convex function]
Given a non-negative real number $\mu \in \R_{\geq 0}$ and differentiable function $f: \R^d \to \R$, $f$ is $\mu$-strongly convex if for all $\x_1$, $\x_2 \in \R^d$,
\begin{equation}
    f ( \x_1 ) 
    \geq f ( \x_2 ) + \langle \nabla f ( \x_2 ), \x_1 - \x_2 \rangle + \frac{\mu}{2} \| \x_1 - \x_2 \|^2.
    \label{def: strongly convex 1}
\end{equation} 
\end{definition} 
Note that a differentiable function is convex if it is $0$-strongly convex.

\begin{definition}[Lipschitz gradient]
Given a non-negative real number $L \in \R_{\geq 0}$ and differentiable function $f: \R^d \to \R$, $f$ has an $L$-Lipschitz gradient if for all $\x_1$, $\x_2 \in \R^d$,
\begin{equation}
    \| \nabla f( \x_1 ) - \nabla f( \x_2 ) \| \leq L \| \x_1 - \x_2 \|.
    \label{def: lipschitz gradient 0}
\end{equation}
\end{definition} 

\subsubsection{Graph Theory}
We denote a network by a directed graph $\mathcal G = ( \V, \E)$, which consists of the set of $N$ nodes  $\V = \{ v_{1}, v_{2}, \ldots, v_{N} \}$ and the set of edges $\E \subseteq \V \times \V$. If $(v_i, v_j) \in \E$, then node $v_j$ can receive information from node $v_i$. 
The in-neighbor and out-neighbor sets are denoted by $\mathcal{N}^{\text{in}}_{i} = \{ v_j \in \mathcal V: \; (v_{j}, v_{i}) \in \E \} $ and $\mathcal{N}^{\text{out}}_{i} = \{ v_j \in \mathcal V: \; (v_{i}, v_{j}) \in \E \} $, respectively. 
A path from node $v_{i}\in \mathcal V$ to node $v_{j}\in \mathcal V$  is a sequence of nodes $v_{k_1},v_{k_2}, \dots, v_{k_l}$ such that $v_{k_1}=v_i$, $v_{k_l}=v_j$ and $(v_{k_r},v_{k_{r+1}}) \in \E$ for $r \in [l-1]$. Throughout the paper, the terms nodes and agents will be used interchangeably. Given a set of vectors $\{\x_1, \x_2, \ldots, \x_N\}$, where each $\x_i \in \R^d$, we use the following shorthand notation for all $\mathcal{S} \subseteq \V$:
$\{ \x_i \}_{\mathcal{S}} = \{ \x_i \in \R^d : v_i \in \mathcal{S} \}$.

\begin{definition}
A graph $\mathcal G = ( \V, \E)$ is said to be rooted at $v_i \in \V$ if for all $v_j \in \V \setminus \{v_i\}$, there is a path from $v_i$ to $v_j$. A graph is said to be rooted if it is rooted at some $v_{i}\in \mathcal V$.
\end{definition}

We will rely on the following definitions from \cite{leblanc2013resilient}. 

\begin{definition}[$r$-reachable set]
For a given graph $\G$ and a positive integer $r \in \Z_+$, a subset of nodes $\mathcal S  \subseteq \mathcal V$ is said to be $r$-reachable if there exists a node $v_i \in \mathcal S$  such that $|\mathcal N^{\textup{in}}_{i} \setminus \mathcal S| \geq r$.
\end{definition}

\begin{definition}[$r$-robust graphs] 
For a positive integer $r \in \Z_+$, a graph $\mathcal G$ is said to be $r$-robust if for all pairs of disjoint nonempty subsets $\mathcal{S}_1, \mathcal{S}_2 \subset \mathcal V$, at least one of $\mathcal{S}_1$ or $\mathcal{S}_2$ is $r$-reachable.
\end{definition}

The above definitions capture the idea that sets of nodes should contain individual nodes that have a sufficient number of neighbors outside that set. This will be important for the {\it local} decisions made by each node in resilient distributed algorithms, and will allow information from the rest of the network to penetrate into different sets of nodes. 

Next, following \cite{cao2008reaching}, we define the composition of two graphs and conditions on a sequence of graphs which will be useful for stating a mild condition for achieving approximate consensus guarantees later as follows.
\begin{definition}[composition]
The composition of a directed graph $\G_1 = (\V, \E_1)$ with a directed graph $\G_2 = (\V, \E_2)$ written as $\G_2 \circ \G_1$, is the directed graph $(\V, \E)$ with $(v_i, v_j) \in \E$ if there is a $v_k \in \V$ such that $(v_i, v_k) \in \E_1$ and $(v_k, v_j) \in \E_2$. 
\end{definition}

\begin{definition}[jointly rooted]
A finite sequence of directed graphs $\{ \G_k \}_{k \in [K]}$ is jointly rooted if the composition $\G_K \circ \G_{K-1} \circ \cdots \circ \G_{1}$ is rooted.
\end{definition}
\begin{definition}[repeatedly jointly rooted]
An infinite sequence of graphs \\ $\{ \G_k \}_{k \in \Z_+}$ is repeatedly jointly rooted if there is a positive integer $q \in \Z_+$ for which each finite sequence $\G_{q(k-1)+1}, \ldots, \G_{qk}$ is jointly rooted for all $k \in \Z_{+}$.
\end{definition} 

For a stochastic matrix $S \in \S^N$, let $\mathbb{G}(S)$ denote the graph $\G$ whose adjacency matrix is the transpose of the matrix obtained by replacing all of $S$'s nonzero entries with $1$'s.

\subsubsection{Adversarial Behavior}
\begin{definition}
A node $v_i \in \mathcal V$ is said to be Byzantine if during each iteration of the prescribed algorithm, it is capable of sending arbitrary values to different neighbors. It is also allowed to update its local information arbitrarily at each iteration of any prescribed algorithm.
\end{definition}

The set of Byzantine agents is denoted by $\vb \subset \mathcal V$. The set of regular agents is denoted by  $\vr = \V \setminus \vb$. The identities of the Byzantine agents are unknown to the regular agents in advance. Furthermore, we allow the Byzantine agents to know the entire topology of the network, functions equipped by the regular nodes, and the deployed algorithm. In addition, Byzantine agents are allowed to coordinate with other Byzantine agents and access the current and previous information contained by the nodes in the network (e.g., current and previous states of all nodes). Such extreme behavior is typical in the field of distributed computing \cite{lynch1996distributed} and in adversarial distributed optimization \cite{su2015byzantine, sundaram2018distributed, yin2018byzantine, yang2019byrdie, farhadkhani2022byzantine}. In exchange for allowing such extreme behavior, we will consider a limitation on the number of such adversaries in the neighborhood of each regular node, as follows. 

\begin{definition}[$F$-local model]
For a positive integer $F \in \Z_+$, we say that the set of adversaries $\vb$ is an $F$-local set if $|\mathcal{N}^{\textup{in}}_{i} \cap \vb | \leq F$ for all $v_{i} \in \vr$.
\end{definition}

Thus, the $F$-local model captures the idea that each regular node has at most $F$ Byzantine in-neighbors.

\subsection{Problem Formulation} 
\label{subsec: formulation}

Consider a group of $N$ agents $\V$ interconnected over a graph $\G = (\V, \E)$. Each agent $v_i \in \V$ has a local cost function $f_i: \R^d \rightarrow \R$. Since Byzantine nodes are allowed to send arbitrary values to their neighbors at each iteration of any algorithm, it is not possible to minimize the quantity $\frac{1}{N}\sum _{v_i \in \V} f_{i}(\x)$ that is typically sought in distributed optimization (since one is not guaranteed to infer any information about the true functions of the Byzantine agents) \cite{sundaram2018distributed, su2015byzantine}. Thus, we restrict the summation only to the regular agents' functions, i.e., the objective is to solve the following optimization problem,
\begin{equation}
    \min_{\x \in \R^d} \; f( \x ), 
    \quad \text{where} \quad
    f( \x ) 
    := \frac{1}{| \vr |} \sum_{v_i \in \vr} f_{i}(\x),
    \label{prob: regular node} 
\end{equation}
where $\vr$ represents the set of regular agents, and $f_i(\x)$ denotes the objective function associated with agent $v_i$.

A key challenge in solving the above problem is that no regular agent is aware of the identities or actions of the Byzantine agents. In particular, solving \eqref{prob: regular node} exactly is not possible under Byzantine behavior, since the identities and local functions of the Byzantine nodes are not known (the Byzantine agents can simply change their local functions and pretend to be a regular agent in the algorithms and these can never be detected). Therefore, one must settle for computing an approximate solution to \eqref{prob: regular node} (see \cite{sundaram2018distributed, su2015byzantine} for a more detailed discussion of this fundamental limitation).

Establishing the convergence (especially obtaining the rate of convergence) for resilient distributed optimization algorithms under general assumptions on the local functions (i.e., not assuming i.i.d. or redundancy) is non-trivial as evidenced by the lack of such results in the literature. We close this gap by introducing a \textit{proper} intermediate step which is showing the states contraction property (Definition~\ref{def: contraction}) before proceeding to show the convergence (Proposition~\ref{prop: convergence} and Theorem~\ref{thm: main-convergence}). Importantly, the contraction property not only captures some of state-of-the-art resilient distributed optimization algorithms in the literature (Theorem~\ref{thm: contraction}) but also facilitates the (geometric) convergence analysis.
\section{Resilient Distributed Optimization Algorithms} 
\label{sec: algorithm}

\subsection{Our Framework}

\begin{algorithm}[tb]
\caption{\underline{Re}silient \underline{D}istributed \underline{Gr}adient-Descent \underline{A}lgorithmic \underline{F}ramework (R{\scriptsize{ED}}G{\scriptsize{RAF}})}
\label{alg: algorithm framework}
\hbox{{\bfseries Input:} Network $\G$, functions $\{f_i\}_{\vr}$, parameter $F$}
\begin{algorithmic}[1]
\State \textbf{Step I:} Initialization
\Statex Each $v_i \in \vr$ sets 
$\z_i [0] \gets$ \texttt{init}($f_i$)
\For{$k = 0, 1, 2, 3, \ldots$}  
\For{$v_i \in \vr$}   
\State \textbf{Step II:} Broadcast and Receive
\Statex \qquad \quad $v_i$ broadcasts $\z_i[k]$ to $\mathcal{N}_i^\text{out}$ and receives $\z_j[k]$ from $v_j \in \mathcal{N}_i^\text{in}$. 
Let $\mathcal{Z}_i [k] = \big\{ \z_j[k] : v_j \in \mathcal{N}_i^\text{in} \cup \{ v_i \} \big\}$
\State \textbf{Step III:} Filtering Step
\Statex \qquad \quad $\Tilde{\z}_i[k] \gets$ \texttt{filt}($\mathcal{Z}_i [k], F$)  
\Comment{Note: $\Tilde{\z}_i[k] = \big[ \Tilde{\x}^T_i[k], \Tilde{\y}^T_i[k] \big]^T$}
\State \textbf{Step IV:} Gradient Update
\begin{align}
\begin{split}
    \x_i [k+1] &= \Tilde{\x}_i [k] - \alpha_k \nabla f_i (\Tilde{\x}_i [k]), \\
    \y_i [k+1] &= \Tilde{\y}_i [k] 
    \quad \text{(if exists),}
    \label{eqn: gradient step}
\end{split}    
\end{align}
\Statex \qquad \quad where $\alpha_k \in \R_{> 0}$ is the step-size
\EndFor
\State $\z_i [k+1] = \big[ \x^T_i [k+1], \y^T_i [k+1] \big]^T$ for $v_i \in \vr$
\EndFor
\end{algorithmic}
\end{algorithm}

In this subsection, we introduce a class of resilient distributed optimization algorithms represented by the \textit{\textbf{re}silient \textbf{d}istributed \textbf{gr}adient-descent \textbf{a}lgorithmic \textbf{f}ramework (R{\scriptsize{ED}}G{\scriptsize{RAF}})} shown in Algorithmic~Framework~\ref{alg: algorithm framework}. 
At each time-step $k \in \N$, each regular agent\footnote{Byzantine agents do not necessarily need to follow the above algorithm, and can update their states, however they wish.} $v_i \in \vr$ maintains and updates a state vector $\x_i[k] \in \R^d$, which is its estimate of the solution to problem~\eqref{prob: regular node}, and optionally an auxiliary vector $\y_i[k] \in \R^{d'}$ where the dimension $d' \in \N$ depends on the specific algorithm. In our algorithmic framework, we let $\z_i[k] = \big[ \x^T_i[k], \y^T_i[k] \big]^T \in \R^{d + d'}$ and, similarly, $\Tilde{\z}_i[k] = \big[ \Tilde{\x}^T_i[k], \Tilde{\y}^T_i[k] \big]^T \in \R^{d + d'}$. In fact, R{\scriptsize{ED}}G{\scriptsize{RAF}} is a generalization of BRIDGE proposed in \cite{fang2022bridge} in the sense that our framework allows the state vector $\z_i [k]$ to include the auxiliary vector $\y_i [k]$.
In Algorithmic~Framework~\ref{alg: algorithm framework}, the operation \texttt{init}($f_i$) initializes $\z_i[0] = \big[ \x^T_i[0], \y^T_i[0] \big]^T$, and the operation \texttt{filt}($\mathcal{Z}_i [k], F$) performs a filtering procedure (to remove potentially adversarial states received from neighbors) and returns a vector $\Tilde{\z}_i[k]$. These functions will vary across algorithms, and will be discussed for specific algorithms later.

\subsection{Definition of Some Standard Operations for Resilient Distributed Optimization}
\label{subsec: methods}

To show that our framework (R{\scriptsize{ED}}G{\scriptsize{RAF}}) captures several existing resilient distributed optimization algorithms as special cases, we first define some operations that are used by existing algorithms. Throughout,
let $\V_i[k] \subseteq \mathcal{N}^{\text{in}}_i \cup \{ v_i \}$, $\X_i[k] = \{ \x_j [k] \}_{\mathcal{N}_i^{\text{in}} \cup \{ v_i \}}$ and $\Y_i[k] = \{ \y_j [k] \}_{\mathcal{N}_i^{\text{in}} \cup \{ v_i \}}$.

\begin{itemize}
    \item $\Tilde{\V}_i[k] \gets$ \texttt{dist\_filt}($\V_i[k]$, $\mathcal{Z}_i[k]$, $F$): \\
    Regular agent $v_i \in \vr$ removes $F$ states that are far away from $\y_i[k]$.
    More specifically, an agent $v_j \in \V_i[k]$ is in $\Tilde{\V}_i[k]$ if and only if 
    \begin{equation*}
        \| \x_j [k] - \y_i [k] \| 
        \leq \max \{ q_M, \| \x_i [k] - \y_i [k] \| \},
    \end{equation*}
    where $q_M = M_F \big( \{ \| \x_s [k] - \y_i [k] \| \}_{v_s \in \V_i[k]} \big)$.
    
    \item $\Tilde{\V}_i[k] \gets$ \texttt{full\_mm\_filt}($\V_i[k]$, $\X_i[k]$, $F$): \\
    Regular agent $v_i \in \vr$ removes states that have extreme values in any of their components.
    For a given $k \in \mathbb{N}$ and $\ell \in [d]$, let 
    $q_m^{(\ell)} = m_F \big( \{ x^{(\ell)}_s [k] \}_{\V_i[k]} \big)$ and
    $q_M^{(\ell)} = M_F \big( \{ x^{(\ell)}_s [k] \}_{\V_i[k]} \big)$.
    An agent $v_j \in \V_i[k]$ is in $\Tilde{\V}_i[k]$ if and only if for all $\ell \in [d]$,
    \begin{equation*}
        \min \{ q_m^{(\ell)}, x^{(\ell)}_i [k] \} 
        \leq x^{(\ell)}_j [k] \leq
        \max \{ q_M^{(\ell)}, x^{(\ell)}_i [k] \}.
    \end{equation*}
    
    \item $\{ \Tilde{\V}^{(\ell)}_i[k] \}_{\ell \in [d]} \gets$ \texttt{cw\_mm\_filt}($\V_i[k]$, $\X_i[k]$, $F$): \\
    For each dimension $\ell \in [d]$, regular agent $v_i \in \vr$ removes the $F$ highest and $F$ lowest values of the states of agents in $\V_i[k]$ along that dimension.
    More specifically, for a given $k \in \mathbb{N}$ and $\ell \in [d]$, let 
    $q_m^{(\ell)} = m_F \big( \{ x^{(\ell)}_s [k] \}_{\V_i[k]} \big)$ and
    $q_M^{(\ell)} = M_F \big( \{ x^{(\ell)}_s [k] \}_{\V_i[k]} \big)$.
    An agent $v_j \in \V_i[k]$ is in $\Tilde{\V}^{(\ell)}_i[k]$ if and only if 
    \begin{equation*}
        \min \{ q_m^{(\ell)}, x^{(\ell)}_i [k] \} 
        \leq x^{(\ell)}_j [k] \leq
        \max \{ q_M^{(\ell)}, x^{(\ell)}_i [k] \}.
    \end{equation*}
    
    \item $\Tilde{\x}_i[k] \gets$ \texttt{full\_average}($\V_i[k]$, $\X_i[k]$): \\
    Regular agent $v_i \in \vr$ computes
    \begin{equation}
        \Tilde{\x}_i [k] 
        = \sum_{v_j \in \V_i[k] } w_{ij} [k] \; \x_j [k],  
        \label{eqn: weight average}
    \end{equation}
    where $\sum_{v_j \in \V_i[k]} w_{ij} [k] = 1$ and $w_{ij} [k]  \in \R_{>0}$ for all $v_j \in \V_i[k]$. 
    
    \item $\Tilde{\x}_i[k] \gets$ \texttt{cw\_average}($\{ \V^{(\ell)}_i[k] \}_{\ell \in [d]}$, $\X_i[k]$): \\
    For each dimension $\ell \in [d]$, regular agent $v_i \in \vr$ computes
    \begin{equation}
        \Tilde{x}^{(\ell)}_i [k] 
        = \sum_{v_j \in \V^{(\ell)}_i[k] } 
        w^{(\ell)}_{ij} [k] \; x^{(\ell)}_j [k],  
        \label{eqn: CW weight average}
    \end{equation}
    where $w^{(\ell)}_{ij} [k] \in \R_{>0}$ for all $v_j \in \V^{(\ell)}_i[k]$ and $\sum_{v_j \in \V^{(\ell)}_i[k]} w^{(\ell)}_{ij} [k] = 1$. 
    
    \item $\Tilde{\x}_i[k] \gets$ \texttt{safe\_point}($\V_i[k]$, $\X_i[k]$, $F$): \\
    Regular agent $v_i \in \vr$ returns a state $\Tilde{\x}_i [k]$ which can be written as
    \begin{equation}
        \Tilde{\x}_i [k] 
        = \sum_{v_j \in \V_i[k] \cap \vr} 
        w_{ij} [k] \; \x_j [k],  
        \label{eqn: convex comb}
    \end{equation}
    where $w_{ij} [k] \in \R_{>0}$ for all $v_j \in \V_i[k] \cap \vr$ and $\sum_{v_j \in \V_i[k] \cap \vr} w_{ij} [k] = 1$. The works \cite{park2017fault, abbas2022resilient} discuss methods used to compute $\Tilde{\x}_i [k]$.
\end{itemize}

\subsection{Mapping Existing Algorithms into R{\scriptsize{ED}}G{\scriptsize{RAF}}}
\label{subsec: existing algs}

Using the operations defined above, we now discuss some algorithms in the literature that fall into our algorithmic framework.

\textit{Simultaneous distance-minmax filtering dynamics (SDMMFD)} \cite{kuwaran2020byzantine_ACC, kuwaran2024scalable} and \textit{simultaneous distance filtering dynamics (SDFD)} \cite{kuwaran2024scalable}: these two algorithms are captured in our framework by defining $\z_i[k] = \big[ \x^T_i[k], \y^T_i[k] \big]^T$ where $\y_i[k] \in \R^d$. 
In the initialization step $\z_i [0] \gets$ \texttt{init}($f_i$) (\textbf{line~1}) of both algorithms, each regular agent $v_i \in \vr$ computes an approximate minimizer $\hat{\x}_i^* \in \R^d$ of its local function $f_i$ (using any appropriate optimization algorithm)
and then sets $\x_i[0] \in \R^d$ arbitrarily and $\y_i[0] = \hat{\x}_i^*$.
In the filtering step $\Tilde{\z}_i[k] \gets$ \texttt{filt}($\mathcal{Z}_i [k], F$) (\textbf{line~5}), SDMMFD executes the following sequence of methods:
\begin{enumerate}
    \item $\V^{\text{dist}}_i[k] \gets \texttt{dist\_filt}(\mathcal{N}^{\text{in}}_i \cup \{ v_i \}, \mathcal{Z}_i[k], F)$,
    \item $\V^{\text{x,mm}}_i[k] \gets \texttt{full\_mm\_filt}(\V^{\text{dist}}_i[k], \X_i[k], F)$,
    \item $\Tilde{\x}_i[k] \gets \texttt{full\_average}(\V^{\text{x,mm}}_i[k], \X_i[k])$,
    \item $\{ \Tilde{\V}^{(\ell)}_i[k] \}_{\ell \in [d]} \gets \texttt{cw\_mm\_filt}(\mathcal{N}^{\text{in}}_i \cup \{ v_i \}, \Y_i[k], F)$,
    \item $\Tilde{\y}_i[k] \gets \texttt{cw\_average}(\{ \Tilde{\V}^{(\ell)}_i[k] \}_{\ell \in [d]}, \Y_i[k])$.
\end{enumerate}
The first three steps compute the intermediate main state $\Tilde{\x}_i[k]$ while the last two steps compute the intermediate auxiliary vector $\Tilde{\y}_i[k]$.
On the other hand, SDFD executes the same sequence of methods except that step (ii) is removed and $\V^{\text{x,mm}}_i[k]$ in step (iii) is replaced by $\V^{\text{dist}}_i[k]$.
Then, for both algorithms, we set $\Tilde{\z}_i [k] = [ \Tilde{\x}^T_i [k], \Tilde{\y}^T_i [k] ]^T$.

\textit{Coordinate-wise trimmed mean (CWTM)} \cite{sundaram2018distributed, su2015byzantine, su2020byzantine, fu2021resilient, zhao2019resilient, fang2022bridge} and \textit{resilient vector optimization (RVO)} based on \textit{resilient vector consensus} \cite{park2017fault, abbas2022resilient}: these algorithms are captured by setting $\z_i[k] = \x_i[k]$  (i.e., $\y_i [k] = \emptyset$).
In the initialization step $\z_i [0] \gets$ \texttt{init}($f_i$) (\textbf{line~1}) of both algorithms, the regular agents $v_i \in \vr$ arbitrarily initialize $\x_i[0] \in \R^d$. 
In the filtering step $\Tilde{\z}_i[k] \gets$ \texttt{filt}($\mathcal{Z}_i [k], F$) (\textbf{line~5}), CWTM executes the following sequence of methods:
\begin{enumerate}
    \item $\{ \Tilde{\V}^{(\ell)}_i[k] \}_{\ell \in [d]} \gets
    \texttt{cw\_mm\_filt}(\mathcal{N}_i^{\text{in}} \cup \{ v_i \}, \X_i[k], F)$,
    \item $\Tilde{\x}_i[k] \gets \texttt{cw\_average}(\{ \V^{(\ell)}_i[k] \}_{\ell \in [d]}, \X_i[k])$,
\end{enumerate}
whereas RVO executes 
\begin{enumerate}
    \item $\Tilde{\x}_i[k] \gets \texttt{safe\_point}(\mathcal{N}_i^{\text{in}} \cup \{ v_i \}, \X_i[k], F)$.
\end{enumerate}

In fact, the algorithms proposed in \cite{ben2015robust, peng2021byzantine, guo2021byzantine, elkordy2022basil} also fall into our framework. However, in this work, we focus on analyzing the four algorithms above since they share some common property (stated formally in Theorem~\ref{thm: contraction}), and we will provide a discussion on the algorithms in the works \cite{ben2015robust, peng2021byzantine, guo2021byzantine, elkordy2022basil} in Remark~\ref{rem: another-algorithms}.
\section{Assumptions and Main Results} 
\label{sec: assumption result}

We now turn to stating assumptions and definitions in subsection~\ref{subsec: asm} which will be used to prove convergence properties in subsection~\ref{subsec: convergence} and consensus properties in subsection~\ref{subsec: consensus}. Finally, in subsection~\ref{subsec: algorithms}, we analyze certain properties of each algorithm mentioned in the previous section. 

\subsection{Assumptions and Definitions}  
\label{subsec: asm}

\begin{assumption} \label{asm: convex}
For all $v_i \in \V$, given positive numbers $\mu_i \in \R_{> 0}$ and $L_i \in \R_{> 0}$, the functions $f_i$ are $\mu_i$-strongly convex and differentiable. Furthermore, the functions $f_i$ have $L_i$-Lipschitz continuous gradients.
\end{assumption}

The strongly convex and Lipschitz continuous gradient assumptions given \; above are common in the distributed convex optimization literature \cite{yuan2016convergence, nedic2017achieving, yang2019survey, hendrikx2019accelerated, kovalev2021linearly}. We define $\tilL := \max_{v_i \in \vr} L_i$ and $\tilmu := \min_{v_i \in \vr} \mu_i$.
Since $\{ f_i \}_{v_i \in \vr}$ are strongly convex functions, let $\x_i^* \in \R^d$ be the minimizer of the function $f_i$, i.e., $f_i (\x_i^*) = \min_{\x \in \R^d} f_i (\x)$. Moreover, let $\c^* \in \R^d$ and $r^* \in \R_{\geq 0}$ be such that $\x_i^* \in \B(\c^*, r^*)$ for all $v_i \in \vr$. Both quantities $\c^*$ and $r^*$ exist due to the existence of the set $\{ \x_i^* \}_{\vr} \subset \R^d$.
Let $\x^* \in \R^d$ be the minimizer of the function $f (\x)$, i.e., the solution of problem~\eqref{prob: regular node}. In other words, $f (\x^*) = \min_{\x \in \R^d} f (\x)$. For convenience, we also denote
$f^* := f(\x^*)$ and $\g_i [k] := \nabla f_i (\Tilde{\x}_i [k])$.

\begin{assumption} \label{asm: robust}
Given a positive integer $F \in \mathbb{Z}_+$, the Byzantine agents form an $F$-local set.
\end{assumption}

\begin{assumption} \label{asm: weight matrices}
There exists a positive number $\omega \in \R_{> 0}$ such that for all $k \in \mathbb{N}$ and $\ell \in [d]$, the non-zero weights $w_{ij} [k]$ in \texttt{full\_average} and \texttt{safe\_point}, and $w_{ij}^{(\ell)} [k]$ in \texttt{cw\_average} (all defined in subsection~\ref{subsec: methods}) are lower bounded by $\omega$.
\end{assumption}

Now, we introduce certain properties related to algorithms within our algorithmic framework (Algorithmic~Framework~\ref{alg: algorithm framework}). These definitions are crucial for proving the convergence results in subsection~\ref{subsec: convergence} and the consensus results in subsection~\ref{subsec: consensus}. Specifically, the following definition captures a certain behavior of the aggregation and filtering mechanism in \textbf{Step III} of the framework.

\begin{definition} \label{def: contraction}
For a vector $\xc \in \R^d$, constant $\gamma \in \R_{\geq 0}$, and sequence $\{ c[k] \}_{k \in \N}$ $\subset \R$, a resilient distributed optimization algorithm $A$ in R{\scriptsize{ED}}G{\scriptsize{RAF}} is said to satisfy the $( \xc, \gamma, \{ c[k] \} )$-\textit{states contraction} property if it holds that $\lim_{k \to \infty} c[k] = 0$ and for all $k \in \mathbb{N}$ and $v_i \in \vr$,
\begin{equation}
    \| \Tilde{\x}_i [k] - \xc \| \leq \sqrt{\gamma} \max_{v_j \in \vr} \| \x_j [k] - \xc \| + c[k].  
    \label{eqn: contraction def} 
\end{equation} 
\end{definition} 
In the above definition, the vector $\xc$ is called the \textit{contraction center} and the constant $\gamma$ is called the \textit{contraction factor}. In general, we want the contraction factor $\gamma$ to be small so that the intermediate state $\Tilde{\x}_i [k]$ remains close to the center $\xc$. The sequence $\{ c[k] \}_{k \in \N}$ captures a perturbation to the contraction term in each time-step. For a given vector $\xc \in \R^d$, we define
\begin{equation}
    \rc := \max_{v_i \in \vr} \| \xc - \x_i^* \|.
    \label{def: x_c distance}
\end{equation}

Next, we introduce a property, concerning algorithms in Algorithmic~Framework~\ref{alg: algorithm framework}, that builds on the states contraction property by further specifying a range for some of the parameters.

\begin{definition} \label{def: reduction}
Suppose Assumption~\ref{asm: convex} holds, and a resilient distributed optimization algorithm $A$ in R{\scriptsize{ED}}G{\scriptsize{RAF}} satisfies the $( \xc, \gamma, \{ c[k] \} )$-\textit{states contraction} property (for some $\xc \in \R^d$, $\gamma \in \R_{\geq 0}$, and $\{ c[k] \}_{k \in \N} \subset \R$). Algorithm $A$ is said to satisfy the reduction property of
\begin{itemize}
    \item Type-I if $\gamma \in [0, 1)$ and $\alpha_k = \alpha \in \big( 0, \frac{1}{\tilL} \big]$;
    \item Type-II if $\gamma \in \Big[ 1, \frac{1}{1 - \frac{\tilmu}{\tilL}} \Big)$ and $\alpha_k = \alpha \in \Big( \frac{1}{\tilmu} \big( 1 - \frac{1}{\gamma} \big), \frac{1}{\tilL} \Big]$.
\end{itemize}
\end{definition} 
This reduction property, in particular, will be used to specify the conditions under which an algorithm converges.

Let $\vr = \{ v_{i_1}, v_{i_2}, \ldots, v_{i_{| \vr |}} \}$ denote the set of all regular agents. For a set of vectors $\{ \u_i \}_{i \in \V} \subset \R^d$ and $\ell \in [d]$, we denote $\u^{(\ell)} = \big[ u^{(\ell)}_{i_1}, u^{(\ell)}_{i_2}, \ldots, u^{(\ell)}_{i_{| \vr |}} \big]^T \in \R^{| \vr |}$, the vector containing the $\ell$-th dimension of each vector $\u_i$ corresponding to the regular agents' indices.
The following definition characterizes the dynamics of all the regular agents in the network which will be a crucial ingredient in showing the approximate consensus result in subsection~\ref{subsec: consensus}.

\begin{definition} \label{def: mixing}
For a set of sequences of matrices $\{ \W^{(\ell)}[k] \}_{k \in \mathbb{N}, \; \ell \in [d]} \subset \mathbb{S}^{| \vr |}$ and constant $G \in \R_{\geq 0}$, a resilient distributed optimization algorithm $A$ in R{\scriptsize{ED}}G{\scriptsize{RAF}} is said to possess $( \{ \W^{(\ell)} [k] \}, G )$-\textit{mixing dynamics} if the state dynamics can be written as
\begin{equation}
    \x^{(\ell)} [k+1] = \W^{(\ell)} [k] \x^{(\ell)} [k] - \alpha_k \g^{(\ell)} [k]
    \label{def: perturbed mixing}
\end{equation}
for all $k \in \mathbb{N}$ and $\ell \in [d]$, the sequences of graphs $\{ \mathbb{G}(\W^{(\ell)} [k]) \}_{k \in \mathbb{N}}$ are repeatedly jointly rooted for all $\ell \in [d]$, and $\limsup_k \| \g_i [k] \|_\infty \leq G$ for all $v_i \in \vr$.
\end{definition}  
The matrix $\W^{(\ell)} [k]$ is called a \textit{mixing matrix} which directly affects the ability of the nodes to reach consensus \cite{cao2008reaching} while the constant $G$ quantifies an upper bound on the perturbation (i.e., the scaled gradient $\alpha_k \g^{(\ell)} [k]$) to the consensus process.

Later in subsection~\ref{subsec: algorithms}, particularly in Theorem~\ref{thm: contraction}, we will formally discuss the $(\xc, \gamma, \{ c[k] \})$-state contraction and the $(\{ \W^{(\ell)} [k] \}, G)$-mixing dynamics properties, along with the parameter values for each algorithm considered in subsection~\ref{subsec: existing algs}.

\subsection{The Region To Which The States Converge}
\label{subsec: convergence}

In this subsection, we derive a convergence result for some particular algorithms in R{\scriptsize{ED}}G{\scriptsize{RAF}} (Theorem~\ref{thm: main-convergence}). We start by establishing convergence to a neighborhood around the center point $\xc$ (Proposition~\ref{prop: convergence}). Following this, we present the main convergence theorem (Theorem~\ref{thm: main-convergence}), leveraging the fact that the minimizer $\x^*$, the solution to problem~\eqref{prob: regular node}, resides within this neighborhood (Lemma~\ref{appx-lem: true minimizer} in Appendix~\ref{subsec: true minimizer proof}).

For convenience, if Assumption~\ref{asm: convex} holds and the step-size $\alpha_k = \alpha$ for all $k \in \N$, we define
\begin{equation}
    \beta := \sqrt{1 - \alpha \tilmu}.
    \label{def: beta}
\end{equation}
We now present an intermediate result, demonstrating that the states of all the regular agents will converge to a ball for all algorithms in R{\scriptsize{ED}}G{\scriptsize{RAF}} that satisfy the reduction property (Definition~\ref{def: reduction}). The proof is provided in Appendix~\ref{subsec: convergence proof}.

\begin{proposition}  \label{prop: convergence}
Suppose Assumption~\ref{asm: convex} holds. If an algorithm $A$ satisfies the reduction property of Type-I or Type-II, then for all $v_i \in \vr$, it holds that
\begin{equation}
    \limsup_k \| \x_i [k] - \xc \| \leq \frac{\rc \sqrt{\alpha \tilL}}{1 - \beta \sqrt{\gamma}} := R^*,
    \label{eqn: conv thm limit}
\end{equation}
where $\xc$, $\rc$ and $\beta$ are defined in Definition~\ref{def: contraction}, \eqref{def: x_c distance}, and \eqref{def: beta}, respectively.
Furthermore, if $c[k] = \O (\xi^k)$ and $\xi \in (0, 1) \setminus \{ \beta \sqrt{\gamma} \}$, then 
\begin{equation}
    \| \x_i [k] - \xc \| \leq R^* + \O \big( ( \max \{ \beta \sqrt{\gamma}, \; \xi \} )^k \big)
    \quad \text{for all} \quad v_i \in \vr.
    \label{eqn: convergence rate}
\end{equation}
\end{proposition}

We refer to $R^*$ in \eqref{eqn: conv thm limit} as the \textit{convergence radius}, and we denote $\B ( \xc, R^* )$ as the \textit{convergence region}. Additionally, the term $\beta \sqrt{\gamma}$ in \eqref{eqn: convergence rate} is referred to as the \textit{convergence rate}\footnote{In this context, a smaller convergence rate indicates faster convergence.}. In particular, the convergence region is the ball which has the center at $\xc$ and the radius $R^*$ depending on the functions' parameters $\mu_i$ and $L_i$, the contraction factor $\gamma$, the constant step-size $\alpha$, and the constant capturing the position of the contraction center $\rc$ (defined in \eqref{def: x_c distance}). We emphasize that the convergence region does not depend on the contraction perturbation sequence $\{ c[k] \}_{k \in \N}$ as long as the sequence converges to zero. 

To analyze the expression of $R^*$ and the convergence rate $\beta \sqrt{\gamma}$, we simplify the expression as follows. Let $\kappa = \frac{\tilL}{\tilmu}$, $\tilalph = \alpha \tilmu$, and $\text{dom}_{R^*} = \Big\{ ( \gamma, \tilalph ) \in \R^2: \gamma \in [0, \infty) \; \text{and} \; \tilalph \in \big( \max \big\{ 0, 1 - \frac{1}{\gamma} \big\}, 1 \big] \Big\}$.
Then for $R^*$, we normalize the expression by $r_c \sqrt{\kappa}$. As a result, we have the convergence rate, $rate: \text{dom}_{R^*} \to [0, 1)$, and (normalized) convergence radius, $\Rnorm: \text{dom}_{R^*} \to \R_{\geq 0}$, as
\begin{equation}
    rate ( \gamma, \tilalph ) = \sqrt{\gamma} \sqrt{1 - \tilalph}
    \quad \text{and} \quad
    \Rnorm ( \gamma, \tilalph ) = \frac{\sqrt{\tilalph}}{1 - \sqrt{\gamma} \sqrt{1 - \tilalph}},
    \label{def: norm-conv}
\end{equation}
respectively. Note that the variable $\kappa$ can be interpreted as an upper bound on the \textit{condition numbers} of $\{ f_i \}_{\vr}$ \cite{gutman2021condition}. The plots regarding the convergence rate ($rate$) and normalized convergence radius ($\Rnorm$) with respect to the scaled constant step-size ($\tilalph$) for some values of $\gamma$ are given in Figures~\ref{fig: conv-rate} and \ref{fig: conv-radius}, respectively.

\begin{figure}
\centering
\subfloat[Convergence Rate]{\includegraphics[width=.45\textwidth]{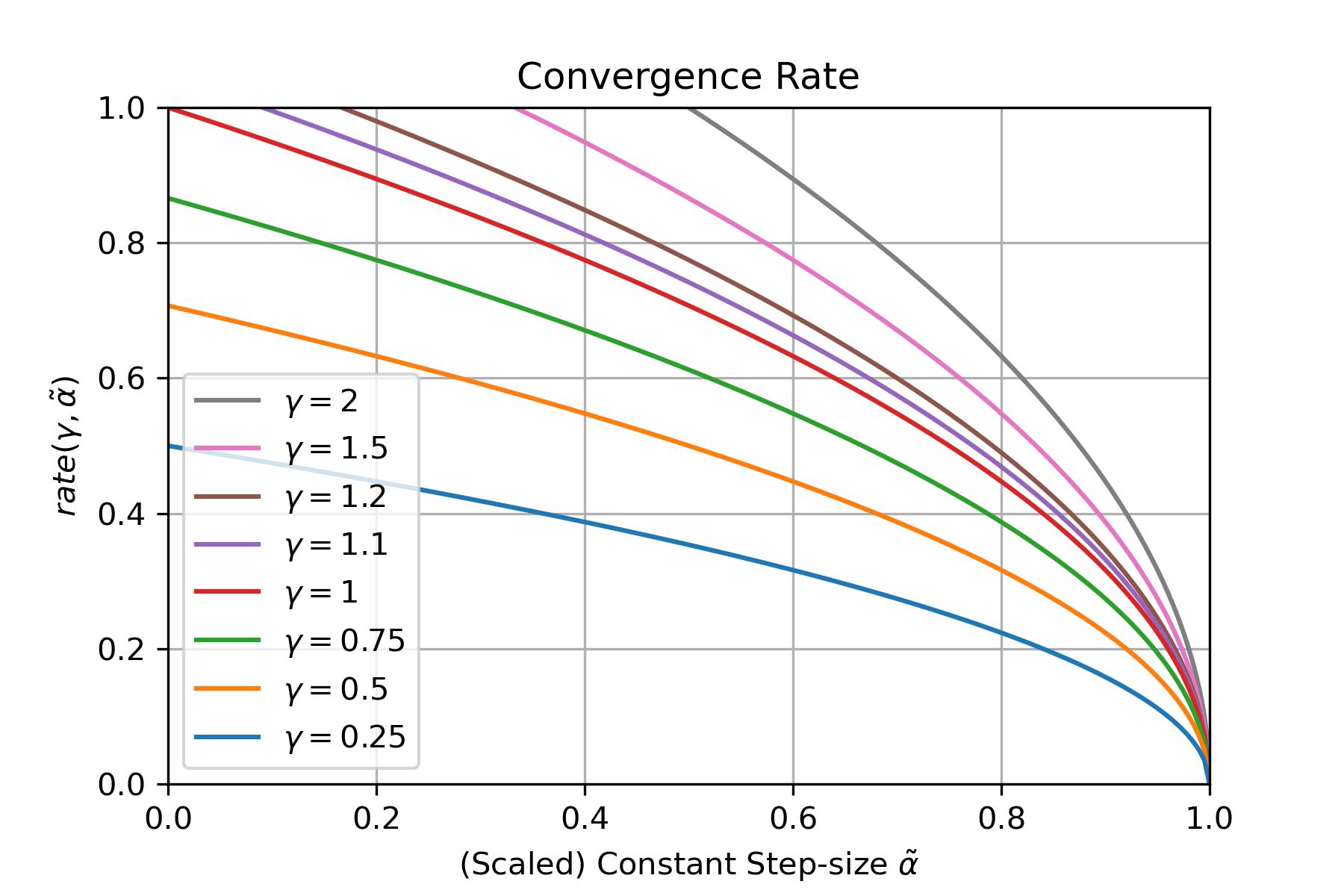} \label{fig: conv-rate}} \;\;
\subfloat[Normalized Convergence Radius]{\includegraphics[width=.45\textwidth]{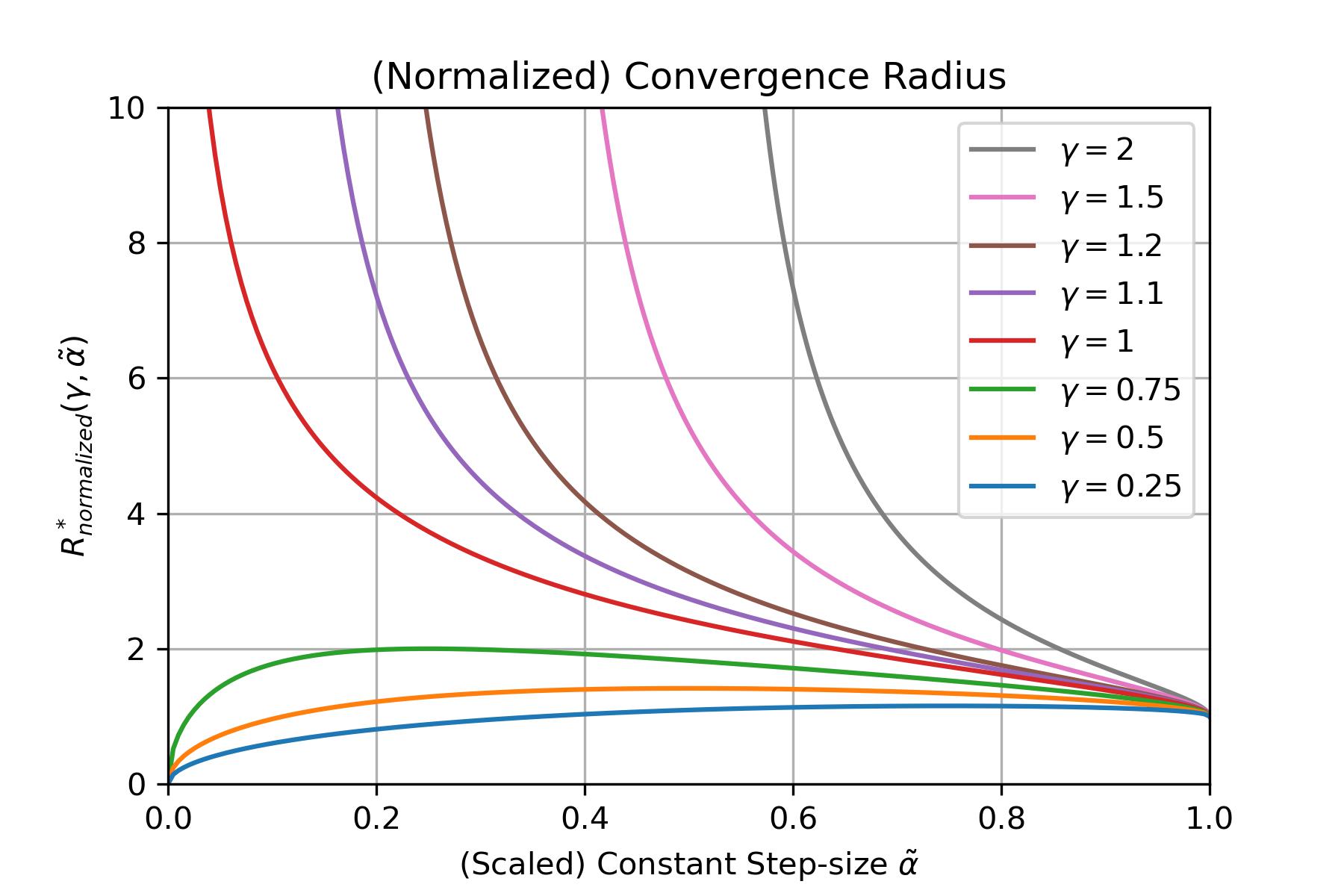} \label{fig: conv-radius}}
\caption{The convergence rate and the normalized convergence radius for different values of the contraction factor $\gamma$ and for legitimate values of the scaled constant step-size $\tilalph$.}
\end{figure}

From \eqref{eqn: convergence rate}, we can conclude that the states of the regular agents converge \textbf{geometrically} to the convergence region $\B ( \xc, R^* )$. Furthermore, it is evident from the inequality \eqref{eqn: convergence rate} and Figure~\ref{fig: conv-rate} that as the constant step-size $\alpha$ increases, the convergence rate decreases (i.e., the states of regular agents converge faster).

Considering the expression of the (normalized) convergence radius $\Rnorm$ in \eqref{def: norm-conv}, it should be noted that $R^*$ is strictly increasing with respect to $\gamma$. In addition, applying Lemma~\ref{lem: R* decreasing} and noting that $\Rnorm$ is a continuous function with respect to $\tilalph$, we can conclude as follows.
\begin{itemize}
    \item For $\gamma \in [0, 1)$, a small constant step-size $\alpha$ would yield a small convergence radius $R^*$ (since $R^* \big|_{\tilalph=0} = 0$ and $R^* \big|_{\tilalph=1} = \rc \sqrt{\kappa}$). Furthermore, we have $R^* \leq R^* \big|_{\tilalph = 1 - \gamma} = \frac{\rc \sqrt{\kappa}}{\sqrt{1 - \gamma}}$ for all valid values of $\alpha$.
    \item For $\gamma \in [1, \infty)$, the optimal convergence radius is obtained by choosing $\alpha = \frac{1}{\tilL}$ (due to the condition on $\alpha$ in Proposition~\ref{prop: convergence}) and the corresponding radius is $R^* \big|_{\tilalph = \frac{1}{\kappa}} = \frac{\rc}{1 - \sqrt{\gamma} \sqrt{1 - \frac{1}{\kappa}}}$.
\end{itemize}

Additionally, the explicit characterization of $R^*$ in \eqref{eqn: conv thm limit} also allows us to analyze its behavior with respect to the (scaled) constant step-size $\tilalph$ when $\tilalph$ is closed to the respective lower bound.
\begin{itemize}
    \item When $\gamma \in [0, 1 )$ and $\tilalph$ approaches $0^+$, we have that $R^* \approx \frac{r_c \sqrt{\kappa}}{1 - \sqrt{\gamma}} \cdot \tilalph^{\frac{1}{2}}$.
    \item When $\gamma = 1$ and $\tilalph$ approaches $0^+$, we have that $R^* \approx 2 r_c \sqrt{\kappa} \cdot \tilalph^{-\frac{1}{2}}$.
    \item When $\gamma \in (1, \infty)$ and $\tilalph$ approaches $\big( 1 - \frac{1}{\gamma} \big)^+$, we have that $R^* \approx \frac{2 r_c \sqrt{\kappa}}{\gamma} \cdot \sqrt{1 - \frac{1}{\gamma}} \cdot \hat{\alpha}^{-1}$, where $\hat{\alpha} = \tilalph - \big( 1 - \frac{1}{\gamma} \big)$.
\end{itemize}

Next, recall that $\x^* \in \R^d$ is the minimizer of the function $\frac{1}{| \vr |} \sum_{v_i \in \vr} f_i (\x)$, which is our objective function (problem~\eqref{prob: regular node}). Lemma~\ref{appx-lem: true minimizer}, which is formally stated and proved in Appendix~\ref{subsec: true minimizer proof}, informs us that the true minimizer $\x^*$ is within the convergence region $\B (\xc, R^*)$, provided a certain condition on $\gamma$ and $\alpha$ holds. This condition, in fact, aligns with the reduction property of Type-II (see Definition~\ref{def: reduction}). Consequently, the geometric convergence of regular agents to a neighborhood of the true minimizer $\x^*$, as shown in Theorem~\ref{thm: main-convergence}, follows directly from applying Lemma~\ref{appx-lem: true minimizer} to Proposition~\ref{prop: convergence}. However, it is important to note that determining the true minimizer $\x^*$ exactly is impossible in the presence of Byzantine agents.

\begin{theorem}[Convergence]  \label{thm: main-convergence}
Suppose Assumption~\ref{asm: convex} holds and an algorithm $A$ satisfies the reduction property of Type-II. If the perturbation sequence $c[k] = \O (\xi^k)$, where $\xi \in (0, 1) \setminus \{ \beta \sqrt{\gamma} \}$, then it holds that
\begin{equation*}
    \| \x_i [k] - \x^* \| \leq 2 R^* + \O \big( (\max \{ \beta \sqrt{\gamma}, \; \xi \} )^k \big)
    \quad \text{for all} \quad v_i \in \vr,
\end{equation*}
where the convergence radius $R^*$ is defined in \eqref{eqn: conv thm limit}.
\end{theorem}

Before comparing our obtained convergence rate with those in the literature, let us first assume that the $f_i$ are $\mu$-strongly convex and have an $L$-Lipschitz continuous gradient for all $v_i \in \vr$, i.e., $\tilmu = \mu$ and $\tilL = L$. Since our work is the first to achieve a linear convergence rate for Byzantine-resilient algorithms under these assumptions, we compare our results with non-resilient algorithms in the literature that consider the same assumptions.

To begin, recall that the convergence rate derived from our approach (see Theorem~\ref{thm: main-convergence}) is $\sqrt{\gamma} \sqrt{1 - \alpha \mu}$. 
For DGD with a constant step-size $\alpha$ \cite{yuan2016convergence}, it has been shown that the convergence rate is $\sqrt{1 - \frac{\alpha}{2} \cdot \frac{\mu L}{\mu + L}}$. Since it is the case that $\mu \leq L$, the best rate for DGD guaranteed by this work is $\sqrt{1 - \frac{1}{2} \alpha \mu}$ (which is the case when $\mu \ll L$).
For DIGing \cite{nedic2017achieving}, representing an algorithm from distributed optimization algorithms with the gradient tracking technique, it has been shown that the convergence rate is $\sqrt{1 - \frac{2}{3} \alpha \mu}$. Since a smaller rate yields faster convergence, our rate of convergence is superior to traditional DGD and DIGing (even though they all have the same order of $1 - \mathcal{O} (\alpha \mu)$). These convergence rates are summarized in Table~\ref{tab: convergence_rates}.

\begin{table}[t]
\caption{Convergence Rates of Algorithms Under Strong Convexity and Lipschitz Gradient Assumptions}
\centering \footnotesize
\begin{tabular}{|c|c|c|c|}
\hline
 & DGD \cite{yuan2016convergence} & DIGing \cite{nedic2017achieving} & \makecell{Our Work \\ (Theorem~\ref{thm: main-convergence})}   \\ 
\hline
\makecell{Convergence Rate \\ (lower is better)} & $\sqrt{1 - \frac{1}{2} \alpha \mu}$ & $\sqrt{1 - \frac{2}{3} \alpha \mu}$ & $\sqrt{1 - \alpha \mu}$ \\ 
\hline
\end{tabular}
\label{tab: convergence_rates}
\end{table}

\begin{remark}
Crucially, even in the absence of Byzantine agents, it is important to highlight that the states of regular agents within our framework converge to a neighborhood of $\x^*$, as formally affirmed in Theorem~\ref{thm: main-convergence}. Notably, this convergence to a neighborhood stands as a fundamental trait of Byzantine distributed optimization problems, regardless of the specific algorithms employed \cite{su2015byzantine, sundaram2018distributed}. This is in contrast to algorithms in traditional distributed optimization settings, where convergence to a neighborhood is an inherent property of the algorithm itself and can be avoided by changing the algorithm. Consequently, comparing the convergence radius between Byzantine and non-Byzantine settings may not be suitable.
\end{remark}

Our result from Theorem~\ref{thm: main-convergence} offers a different approach to convergence proofs than those typically found in the literature, which are often designed for specific algorithms. By focusing on proving the states contraction property (Definition~\ref{def: contraction}), rather than the details of the functions involved, one can save a considerable amount of time and effort. However, it is worth noting that this approach only provides a sufficient condition for convergence. There may be resilient algorithms that do not satisfy the property but still converge geometrically. In fact, finding general necessary conditions for convergence in resilient distributed optimization remains an open question in the literature.

\begin{remark}
The work \cite{wu2023byzantine} introduces a contraction property which seems to be similar to Definition~\ref{def: contraction}. However, there is a subtle difference in that their contraction center is time-varying (since it is a function of neighbors' states) while it is a constant (but depends on algorithms) in our case. However, it is unclear whether their notion of contraction allows for the proof of geometric convergence, as demonstrated in Proposition~\ref{prop: convergence} and Theorem~\ref{thm: main-convergence}.
\end{remark}

Having established convergence of all regular agent's values to a ball that contains the true minimizer, we now turn our attention to characterizing the distance between the regular agents' values within that ball.
Given the reduction property in Definition~\ref{def: reduction} (either Type-I or Type-II), we can use \eqref{eqn: conv thm limit} to derive a bound on the distance between the values held by different nodes: $\limsup_k \| \x_i [k] - \x_j [k] \| \leq 2 R^*$ for all $v_i, v_j \in \vr$. However, this bound is not particularly useful since the right-hand side quantity can be large. In the next subsection, we will demonstrate that the mixing dynamics (Definition~\ref{def: mixing}) and a constant step-size are sufficient to obtain a more meaningful bound on the approximate consensus.

\subsection{Convergence to Approximate Consensus of States}
\label{subsec: consensus}

The following proposition characterizes the approximate consensus among the regular agents in the network under the mixing dynamics (Definition~\ref{def: mixing}) and a constant step-size (proved in Appendix~\ref{subsec: proof of consensus proposition}).

\begin{proposition}  \label{prop: consensus}
If an algorithm $A$ in R{\scriptsize{ED}}G{\scriptsize{RAF}} satisfies the \; $( \{ \W^{(\ell)} [k] \}, G )$-mixing dynamics property (for some $\{ \W^{(\ell)}[k] \}_{k \in \mathbb{N}, \; \ell \in [d]} \subset \mathbb{S}^{| \vr |}$ and $G \in \R_{\geq 0}$) and $\alpha_k = \alpha$ for all $k \in \mathbb{N}$, then there exist $\rho \in \R_{\geq 0}$ and $\lambda \in (0, 1)$ such that 
\begin{equation}
    \limsup_k \| \x_i [k] - \x_j [k] \| \leq \frac{\alpha \rho G \sqrt{d}}{1 - \lambda}
    \quad \text{for all} \quad v_i, v_j \in \vr.
    \label{eqn: lim approx consensus}
\end{equation}
\end{proposition}

From the consensus theorem above, we note that $\max_{v_i, v_j \in \vr} \| \x_i [k] - \x_j [k] \| = \O (\alpha \sqrt{d})$ if $G$ does not depend on the constant step-size $\alpha$ and the dimension $d$.

\begin{remark}
According to \cite{cao2008reaching}, the quantity $\lambda \in (0, 1)$ depends only on the network topology (for each time-step) induced by the sequence of graphs $\{ \mathbb{G} (\W^{(\ell)} [k]) \}$ while the quantity $\rho \in \R_{\geq 0}$ depends on the number of regular agents $| \vr |$ and the quantity $\lambda$.
\end{remark}

In fact, the states contraction property (Definition~\ref{def: contraction}) implies a bound on the gradient $\| \g_i [k] \|_\infty$ (the formal statement is provided in Appendix~\ref{subsec: grad bound}) which is one of the requirements of the mixing dynamics. Thus, we can achieve a similar approximate consensus result as Proposition~\ref{prop: consensus} given that an algorithm satisfies the reduction property (Definition~\ref{def: reduction}) and the associated sequence of graphs for each dimension is repeatedly jointly rooted as shown in the following theorem whose proof is provided in Appendix~\ref{subsec: consensus theorem}.

\begin{theorem}[Consensus] \label{thm: consensus}
Suppose Assumption~\ref{asm: convex} holds and an algorithm $A$ satisfies the reduction property of Type-I or Type-II. If the dynamics of the regular states can be written as \eqref{def: perturbed mixing}, where $\{ \mathbb{G}(\W^{(\ell)} [k]) \}_{k \in \mathbb{N}}$ is repeatedly jointly rooted for all $\ell \in [d]$, then there exist $\rho \in \R_{\geq 0}$ and $\lambda \in (0, 1)$ such that
\begin{equation}
    \limsup_k \| \x_i [k] - \x_j [k] \| 
    \leq \frac{\alpha \rho \rc \tilL \sqrt{d}}{1 - \lambda} \Bigg( 1 + \frac{\sqrt{\alpha \gamma \tilL}}{1 - \beta \sqrt{\gamma}} \bigg) := D^*
    \;\; \text{for all} \;\; v_i, v_j \in \vr,
    \label{eqn: cor lim approx consensus}
\end{equation}
where $\rc$ and $\beta$ are defined in \eqref{def: x_c distance} and \eqref{def: beta}, respectively.
Furthermore, if $c[k] = \O(\xi^k)$ where $\xi \in (0, 1) \setminus \{ \beta \sqrt{\gamma} \}$, then there exist $\rho \in \R_{\geq 0}$ and $\lambda \in (0, 1)$ such that
\begin{equation}
    \| \x_i [k] - \x_j [k] \| \leq D^* + \O \big( ( \max \{ \beta \sqrt{\gamma}, \xi, \lambda \} )^k \big)
    \quad \text{for all} \quad v_i, v_j \in \vr.
    \label{eqn: cor approx consensus}
\end{equation}
\end{theorem}

We refer to $D^*$ in \eqref{eqn: cor lim approx consensus} as the \textit{approximate consensus diameter}. From \eqref{eqn: cor approx consensus}, we can conclude that the distance between any two regular agents' states converge \textbf{geometrically} to the approximate consensus diameter $D^*$. Furthermore, as suggested by \eqref{eqn: cor approx consensus}, for the case that $\beta \sqrt{\gamma} > \max \{ \xi, \lambda \}$, the distance converges faster as the constant step-size $\alpha$ increases.

To analyze the expression of $D^*$, we simplify the expression as follows. Let $\text{dom}_{D^*} = \Big\{ (\kappa, \gamma, \tilalph) \in \R^3 : 
\kappa \in [1, \infty), \gamma \in [0, 1) \; \text{and} \; \tilalph \in \big( 0, \frac{1}{\kappa} \big], \; \text{or} \;
\kappa \in [1, \infty), \gamma \in \big[ 1, \frac{\kappa}{\kappa - 1} \big) \; \text{and} \;$ $\tilalph \in \Big( 1 - \frac{1}{\gamma}, \frac{1}{\kappa} \Big] \Big\}$.
Using changes of variables $\kappa = \frac{\tilL}{\tilmu}$ and $\tilalph = \alpha \tilmu$ (as in subsection~\ref{subsec: convergence}) and then normalizing the expression by $\frac{\rho \rc \sqrt{d}}{1 - \lambda}$, we have the (normalized) approximate consensus diameter $\Dnorm: \text{dom}_{D^*} \to \R_{\geq 0}$ defined as
\begin{equation}
    \Dnorm (\kappa, \gamma, \tilalph)
    = \kappa \tilalph \bigg( 1 + \frac{\sqrt{\kappa \gamma \tilalph}}{1 - \sqrt{\gamma} \cdot \sqrt{1 - \tilalph}} \bigg)
    = \kappa \tilalph (1 + \sqrt{\kappa \gamma} \cdot \Rnorm),
    \label{def: norm-cons}
\end{equation}
where $\Rnorm$ is the (normalized) convergence radius defined in \eqref{def: norm-conv}.
It can be noted that $\Dnorm$ is strictly increasing with both $\kappa$ and $\gamma$. However, $\Dnorm$ is neither an increasing nor decreasing function with respect to $\tilalph$. The plots between the (normalized) approximate consensus diameter $\Dnorm$ and the (scaled) constant step-size $\tilalph$ for some values of $\kappa$ and $\gamma$ are given in Figures~\ref{fig: sigma small} to \ref{fig: sigma large}.
The plots suggest that for $\gamma \leq 1$, small constant step-sizes $\alpha$ provide small approximate consensus diameters $D^*$ while large constant step-sizes $\alpha$ may be preferable in the case that $\gamma > 1$. 

\begin{figure}
\centering
\subfloat[$\kappa = 1.5$.]{\includegraphics[width=.32\textwidth]{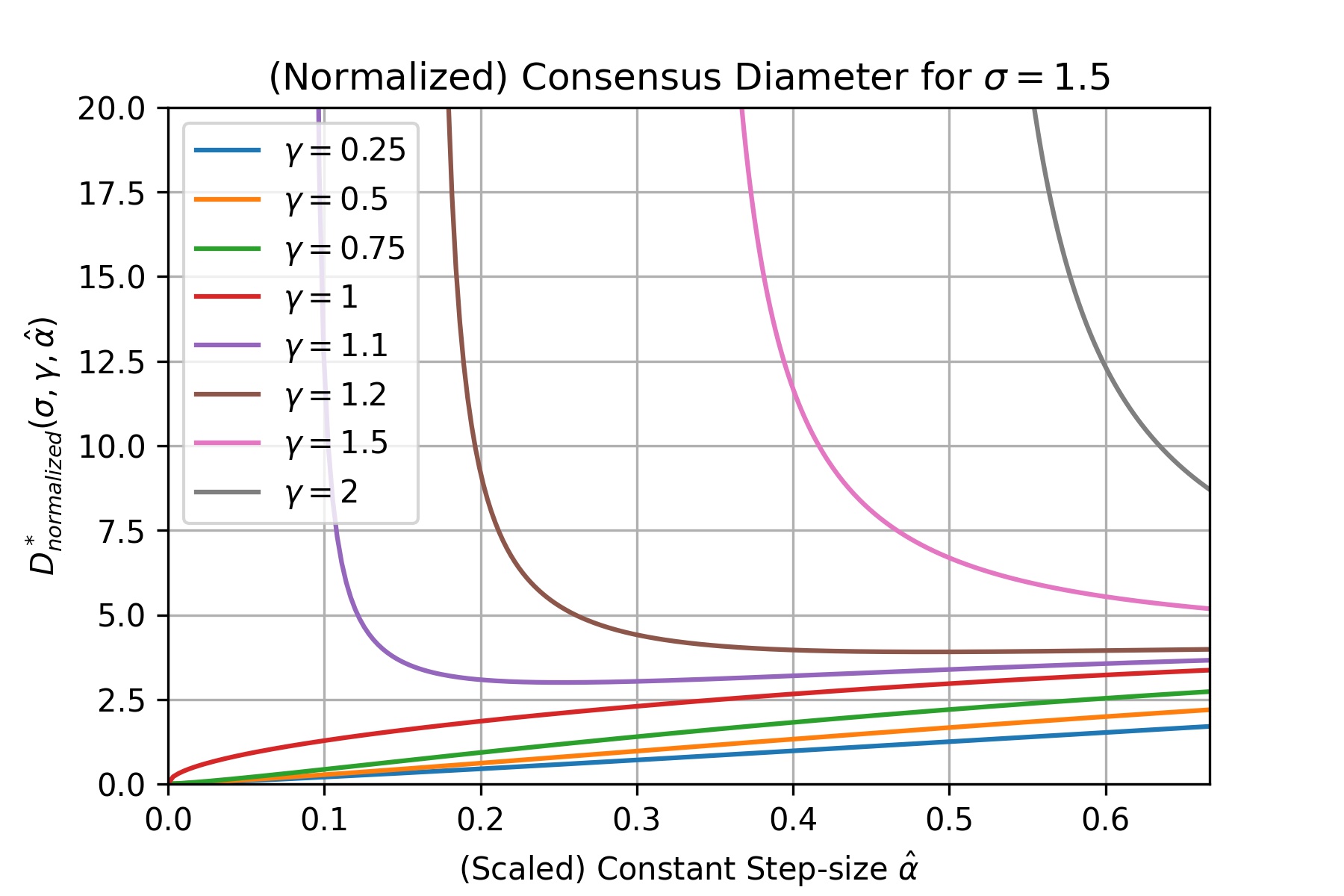} \label{fig: sigma small}}
\subfloat[$\kappa = 2$.]{\includegraphics[width=.32\textwidth]{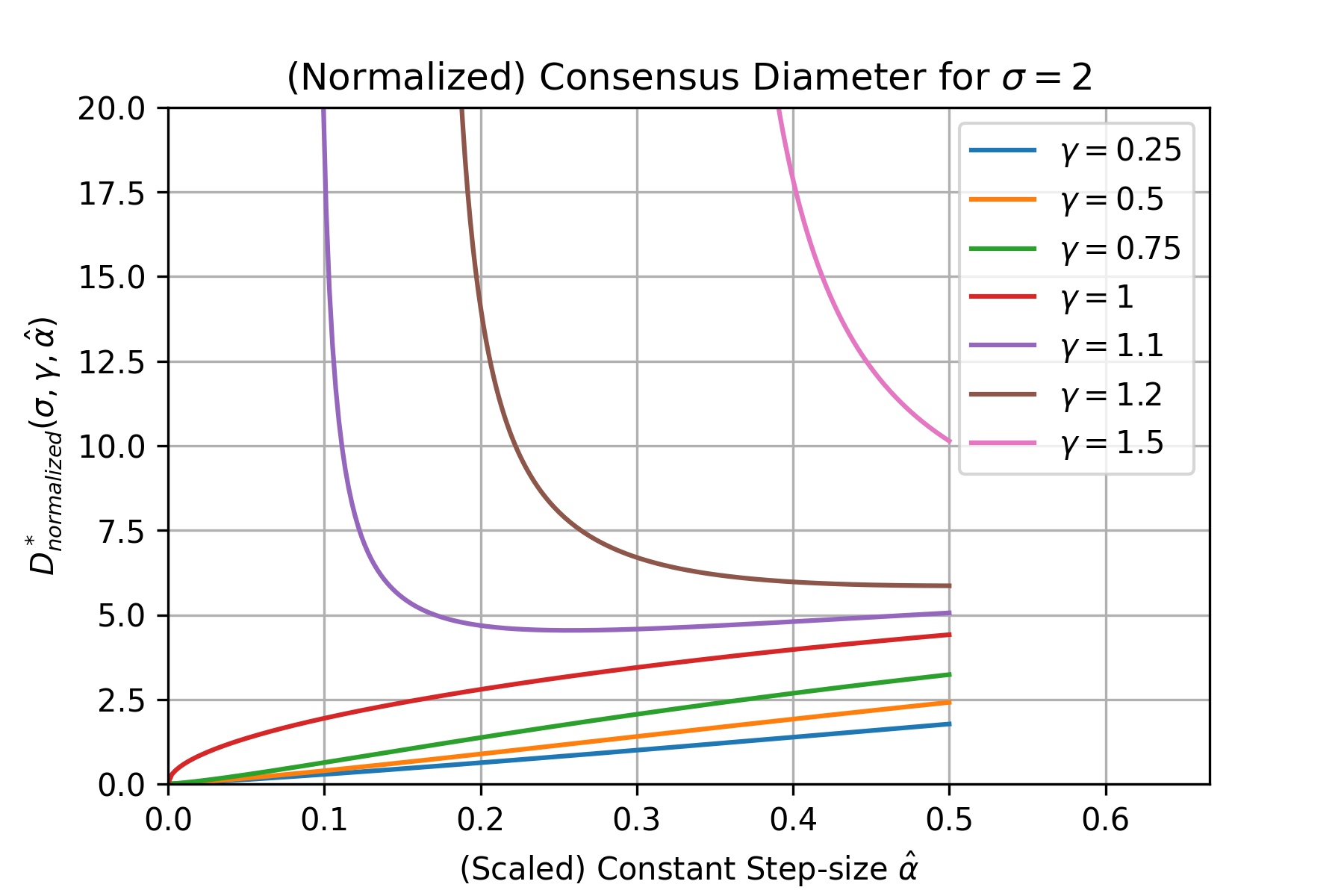} \label{fig: sigma medium}}
\subfloat[$\kappa = 3$.]{\includegraphics[width=.32\textwidth]{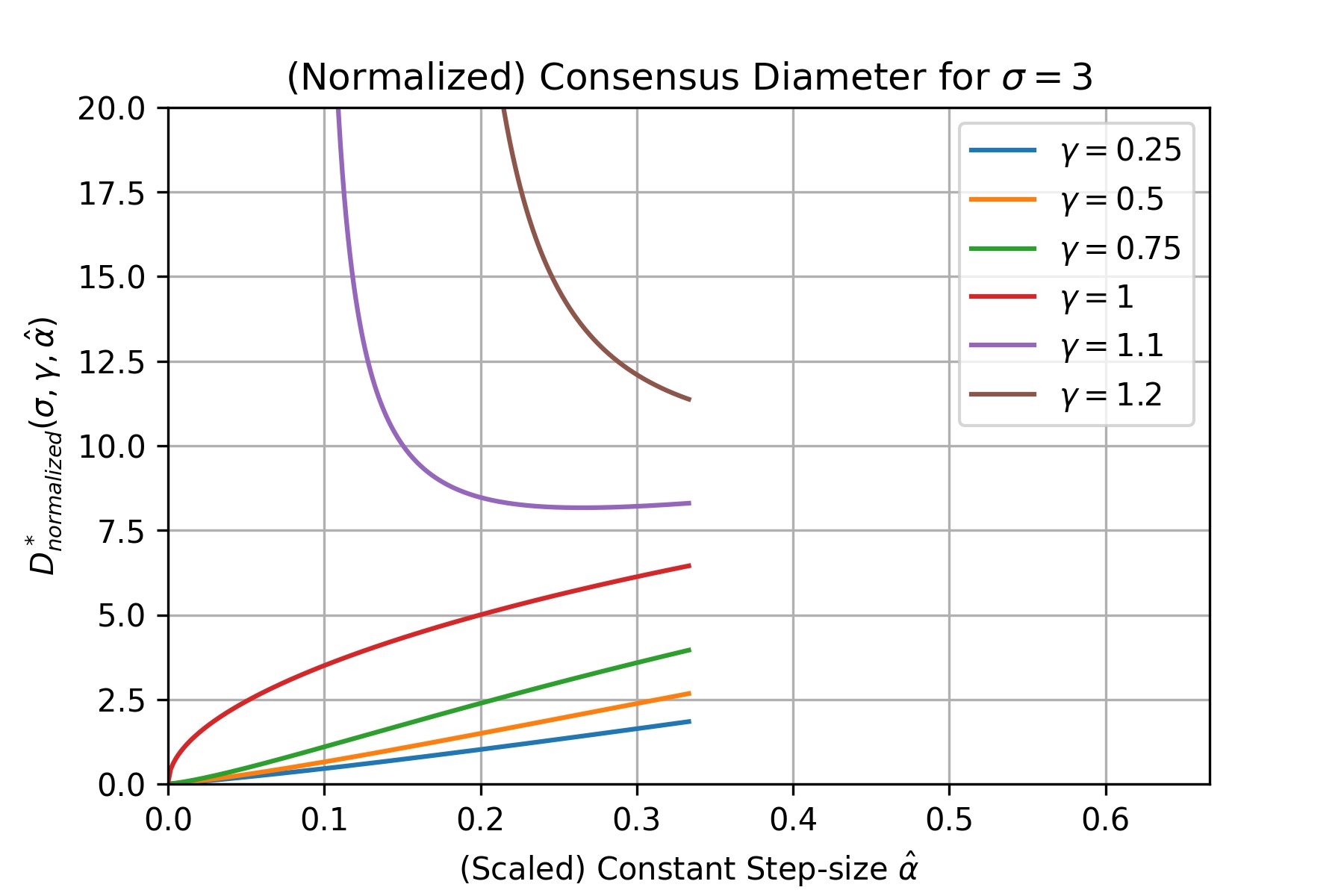} \label{fig: sigma large}}
\caption{The (normalized) approximate consensus diameter $\Dnorm$ for different values of the contraction factor $\gamma$ and for legitimate values of the (scaled) constant step-size $\tilalph$.}
\label{fig: analysis_cons-diam}
\end{figure}

Additionally, by applying the approximation of $R^*$ from subsection~\ref{subsec: convergence} to \eqref{def: norm-cons}, we obtain insights regarding the dependence of the approximate consensus diameter $D^*$ on the (scaled) constant step-size $\tilalph$ as shown below.
\begin{itemize}
    \item When $\gamma \in [0, 1 )$ and $\tilalph$ approaches $0^+$, we have that $D^* \approx \frac{\rho \rc \kappa \sqrt{d}}{1 - \lambda} \cdot \tilalph$.
    \item When $\gamma = 1$ and $\tilalph$ approaches $0^+$, we have that $D^* \approx \frac{2 \rho \rc \kappa^{\frac{3}{2}} \sqrt{d}}{1 - \lambda} \cdot \tilalph^\frac{1}{2}$.
    \item When $\gamma \in (1, \infty)$ and $\tilalph$ approaches $\big( 1 - \frac{1}{\gamma} \big)^+$, we have that $D^* \approx \frac{2 \rho \rc}{1 - \lambda} \sqrt{\frac{d}{\gamma}} \cdot \Big( \kappa \big( 1 - \frac{1}{\gamma} \big) \Big)^{\frac{3}{2}} \cdot \hat{\alpha}^{-1}$, where $\hat{\alpha} = \tilalph - \big( 1 - \frac{1}{\gamma} \big)$.
\end{itemize}

\begin{remark}
In fact, the states contraction property (Definition~\ref{def: contraction}) implicitly relates to network connectivity, where connectivity assures the states' contraction property for non-faulty cases. In such cases, network connectivity conditions suffice to achieve both convergence and consensus, as evidenced in \cite{yuan2016convergence} and \cite{nedic2017achieving}. However, in Byzantine scenarios, the states' contraction depends not only on connectivity but also on filter properties specific to each resilient algorithm. Our work introduces the states contraction property (Definition~\ref{def: contraction}) to capture the desired characteristics of resilient algorithm filters, while the mixing dynamics property (Definition~\ref{def: mixing}) aims to capture network connectivity (associated with the equivalent dynamics of regular agents). However, Theorem~\ref{thm: consensus} presents a set of assumptions that lead to both the convergence and consensus results (even though the convergence result does not require the repeatedly jointly rooted assumption).
\end{remark}

\subsection{Implications for Existing Resilient Distributed Optimization Algorithms}
\label{subsec: algorithms}

We now describe the implication of our above results (for our general framework) for the specific existing algorithms discussed in subsection~\ref{subsec: existing algs}: SDMMFD \cite{kuwaran2020byzantine_ACC, kuwaran2024scalable}, SDFD \cite{kuwaran2024scalable}, CWTM \cite{sundaram2018distributed, su2015byzantine, su2020byzantine, fu2021resilient, zhao2019resilient, fang2022bridge}, and RVO \cite{park2017fault, abbas2022resilient}.
In particular, we show that the algorithms satisfy the states contraction (Definition~\ref{def: contraction}) and the mixing dynamics (Definition~\ref{def: mixing}) properties with different quantities which are determined in the following theorem.

Before stating the theorem, recall that $d \in \Z_+$ is the number of dimensions of the optimization variable $\x$ in \eqref{prob: regular node}. For SDMMFD and SDFD, let $\y[\infty] \in \R^d$ be the point such that $\lim_{k \to \infty} \y_i [k] = \y[\infty]$ for all $v_i \in \vr$. Such a point exists due to \cite[Proposition~1]{kuwaran2024scalable}.

Additionally, recall that $F \in \Z_+$ is the parameter in the $F$-local model (Assumption~\ref{asm: robust}). Since the step in RVO depends on a specific implemented algorithm, we assume that there exists a function $p: \Z_+ \times \Z_+ \to \Z_+$ such that if the graph $\G$ is $p(d, F)$-robust then the step in RVO returns a non-empty set of states.

\begin{theorem}  \label{thm: contraction}
Suppose Assumptions~\ref{asm: convex}-\ref{asm: weight matrices} hold, $\alpha_k = \alpha$ for all $k \in \N$, and $\rc$ and $\beta$ are defined in \eqref{def: x_c distance} and \eqref{def: beta}, respectively. Let $\{ 0[k] \} = \{ 0 \}_{k \in \N}$.
\begin{itemize}
    \item If $\G$ is $((2 d + 1)F + 1)$-robust then there exists $c_1, c_2 \in \R_{\geq 0}$ such that the SDMMFD from \cite{kuwaran2020byzantine_ACC, kuwaran2024scalable} satisfies the $(\y [\infty], 1, \{ 2c_1 e^{-c_2k} \})$-states contraction property and there exists $\{ \W^{(\ell)} [k] \}_{k \in \N, \ell \in [d]} \subset \S^{| \vr |}$ such that the algorithm satisfies the $(\{ \W^{(\ell)} [k] \}, G)$-mixing dynamics property with $G = \rc \tilL \Big( 1 + \frac{\sqrt{\alpha \tilL}}{1 - \beta} \Big)$.
    
    \item If $\G$ is $(2 F + 1)$-robust then there exists $c_1, c_2 \in \R_{\geq 0}$ such that the SDFD from \cite{kuwaran2024scalable} satisfies the $(\y [\infty], 1, \{ 2c_1 e^{-c_2k} \})$-states contraction property.
    
    \item If $\G$ is $(2 F + 1)$-robust then the CWTM from \cite{sundaram2018distributed, su2015byzantine, su2020byzantine, fu2021resilient, zhao2019resilient, fang2022bridge} satisfies the $(\c^*, d, \{ 0[k] \})$-states contraction property and there exists $\{ \W^{(\ell)} [k] \}_{k \in \N, \ell \in [d]}$ $\subset \S^{| \vr |}$ such that the algorithm satisfies the $(\{ \W^{(\ell)} [k] \}, G)$-mixing dynamics property with $G = \rc \tilL \Big( 1 + \frac{\sqrt{\alpha d \tilL}}{1 - \beta \sqrt{d}} \Big)$.
    
    \item If $\G$ is $p(d, F)$-robust then the RVO from \cite{park2017fault, abbas2022resilient} satisfies the $(\c^*, 1, \{ 0[k] \})$-states contraction property and there exists $\{ \W^{(\ell)} [k] \}_{k \in \N, \ell \in [d]} \subset \S^{| \vr |}$ such that the algorithm satisfies the $(\{ \W^{(\ell)} [k] \}, G)$-mixing dynamics property with $G = \rc \tilL \Big( 1 + \frac{\sqrt{\alpha \tilL}}{1 - \beta} \Big)$.
\end{itemize}
\end{theorem}

We provide a summary of the proof for Theorem~\ref{thm: contraction} here, with the full proof available in Appendix \ref{subsec: proof of algorithms}. To establish the states contraction property for each algorithm, we proceed as follows:
\begin{itemize}
    \item For SDMMFD and SDFD, we combine the results from \cite[Proposition~1]{kuwaran2024scalable} and \cite[Lemma~2]{kuwaran2024scalable}.
    \item For CWTM, we apply the results of \cite[Proposition~5.1]{sundaram2018distributed}.
    \item For RVO, we manipulate the equation \eqref{eqn: convex comb} to derive the required result.
\end{itemize}

To demonstrate the mixing dynamics property for each algorithm (specifically for SDMMFD, CWTM, and RVO), we proceed as follows: 
\begin{itemize}
    \item For both SDMMFD and CWTM, we rewrite the dynamics of the regular agents as in \eqref{def: perturbed mixing}, utilizing \cite[Theorem~6.1]{sundaram2018distributed}.
    \item For RVO, we similarly rewrite the dynamics of the regular agents using \eqref{eqn: convex comb}.
\end{itemize}

Next, we employ \cite[Lemma~2.3]{sundaram2018distributed}, given the robustness conditions, to establish repeated jointly rootedness for the corresponding graph sequences. Finally, we determine the constant $G$ in Definition~\ref{def: mixing} for each case by substituting the corresponding contraction factor $\gamma$ into \eqref{eqn: limsup grad bound} (from Lemma~\ref{lem: grad bound}).

Now, we consider bounding the distance $\rc$, as defined in \eqref{def: x_c distance}. It is worth noting that the constant $\rc$ appears in two important quantities: the convergence radius $R^*$ and approximate consensus diameter $D^*$ defined in \eqref{eqn: conv thm limit} and \eqref{eqn: cor lim approx consensus}, respectively. In fact, $\rc$ can be upper bounded by a quantity depending on the diameter of the minimizers of the regular agents' functions $r^*$ defined in subsection~\ref{subsec: asm}. The formal statement is provided below while the proof is provided in Appendix~\ref{subsec: r_c bound}.

\begin{lemma}  \label{lem: radius-c}
Suppose Assumption~\ref{asm: convex} holds and for the initialization step of SDMMFD and SDFD, there exists $\epsilon^* \in \R_{\geq 0}$ such that $\| \hat{\x}_i^* - \x_i^* \|_\infty \leq \epsilon^*$ for all $v_i \in \vr$.
\begin{itemize}
    \item For SDMMFD and SDFD, we have $\rc \leq \sqrt{d} (r^* + \epsilon^*) + r^*$.
    \item For CWTM and RVO, we have $\rc \leq r^*$.
\end{itemize}
\end{lemma} 

Applying the lemma to \eqref{eqn: conv thm limit}, we can conclude that the convergence radius $R^*$ is $\O (\sqrt{d} r^*)$ for SDMMFD and SDFD, and $\O (r^*)$ for CWTM and RVO. Even though \cite{sundaram2018distributed, kuwaran2024scalable} consider different set of assumptions, they also obtain the linear dependency on $r^*$. To achieve a convergence radius of $o (r^*)$, additional assumptions are required, as evident from \cite{yang2019byrdie, fang2022bridge, gupta2021byzantine}. We conjecture that $\O (r^*)$ may be an optimal characteristic, applicable to both convex cases (under bounded gradient assumptions) and strongly convex cases (under Lipschitz gradient assumptions). This assertion arises from the inherent ambiguity in determining the minimizer of the collective sum of all local functions, as discussed in \cite{kuwaran2018location, kuwaran2020set, kuwaran2023minimizer, zamani2024set}. Still, the question regarding a tight lower bound on the convergence radius for the general case (whether how it depends on $r^*$, $\tilmu$, and $\tilL$) remains an open problem.

\begin{remark}  \label{rem: another-algorithms}
It is worth noting that the algorithms proposed in \cite{ben2015robust} and \cite{peng2021byzantine} do not satisfy the states contraction property (Definition~\ref{def: contraction}). However, in fact, they satisfy inequality \eqref{eqn: contraction def} with the perturbation term being bounded by a constant, and thus it is not difficult to use our techniques to show that they geometrically converge to a region with the contraction center $\xc$ but the region has the radius greater than $R^*$ given in \eqref{eqn: conv thm limit}. On the other hand, the algorithms in \cite{guo2021byzantine, elkordy2022basil} do not satisfy the contraction property and may require other techniques to establish convergence (if possible).
\end{remark}

Having proved the states contraction and mixing dynamics properties of the algorithms from \cite{sundaram2018distributed, su2015byzantine, su2020byzantine, fu2021resilient, zhao2019resilient, fang2022bridge, kuwaran2020byzantine_ACC, kuwaran2024scalable, park2017fault, abbas2022resilient}, from Theorem~\ref{thm: main-convergence}, we can deduce that under certain conditions on the graph robustness and step-size $\alpha_k$, the states of the regular agents geometrically converge to the convergence region with $\xc$ and $\gamma$ determined by Theorem~\ref{thm: contraction}. On the other hand, from Theorem~\ref{thm: consensus}, we can deduce that the states of the regular agents geometrically converge together at least until the diameter reaches the approximate consensus diameter $D^*$ (where the parameters $\rc$ and $\gamma$ depend on the implemented algorithm). 

To the best of our knowledge, our work is the first to show the geometric convergence results and characterize the convergence region for the resilient algorithms mentioned above. Thus, our framework, defined properties, and proof techniques provide a general approach for analyzing the convergence region and rate for a wide class of resilient optimization algorithms.
\section{Numerical Experiments}
\label{sec: experiment}

We now present numerical experiments to illustrate the behavior of the algorithms discussed in subsection~\ref{subsec: existing algs}. Using synthetic quadratic functions, we investigate the geometric convergence properties of these algorithms and analyze how step-size affects both convergence rates and final outcomes.\footnote{Our code is available at \url{https://github.com/kkuwaran/resilient-distributed-optimization}.}

In particular, we focus on quadratic functions with two independent variables as the local cost functions, primarily due to the computational bottleneck posed by RVO. The network comprises 40 agents and is modeled as an 11-robust graph. We implement the $F$-local adversary model, setting $F = 2$. Each Byzantine agent transmits a random vector to its regular neighbors, designed to closely resemble other received vectors, increasing the likelihood of bypassing the filter applied by regular agents. For all algorithms (SDMMFD, SDFD, CWTM, and RVO), a constant step-size of either $\alpha = 0.02$ or $\alpha = 0.04$ is used.

Keeping the network structure, local functions, and Byzantine agent identity consistent, we perform four independent runs of the experiment for each algorithm to account for the stochastic nature of adversary behavior and variability in the initialization of state and auxiliary vectors. The results, reported as the mean and standard deviation of key metrics, are averaged over all runs.

\begin{figure}
\centering
\subfloat[The Euclidean distance from the average of the regular agents' states to the true minimizer $\| \Bar{\x} - \x^* \|$ for each algorithm.]
{\includegraphics[width=.45\textwidth]{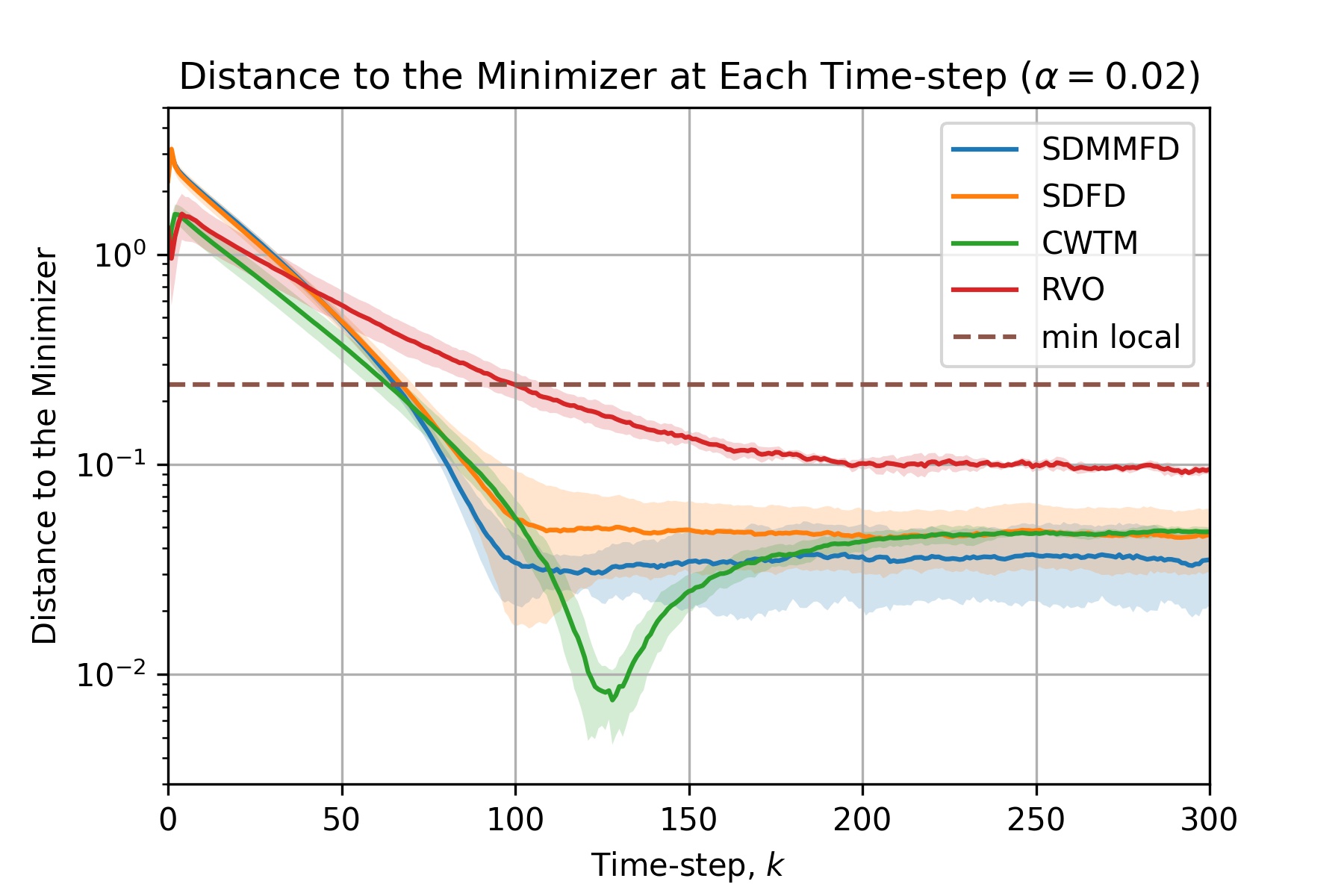}
\includegraphics[width=.45\textwidth]{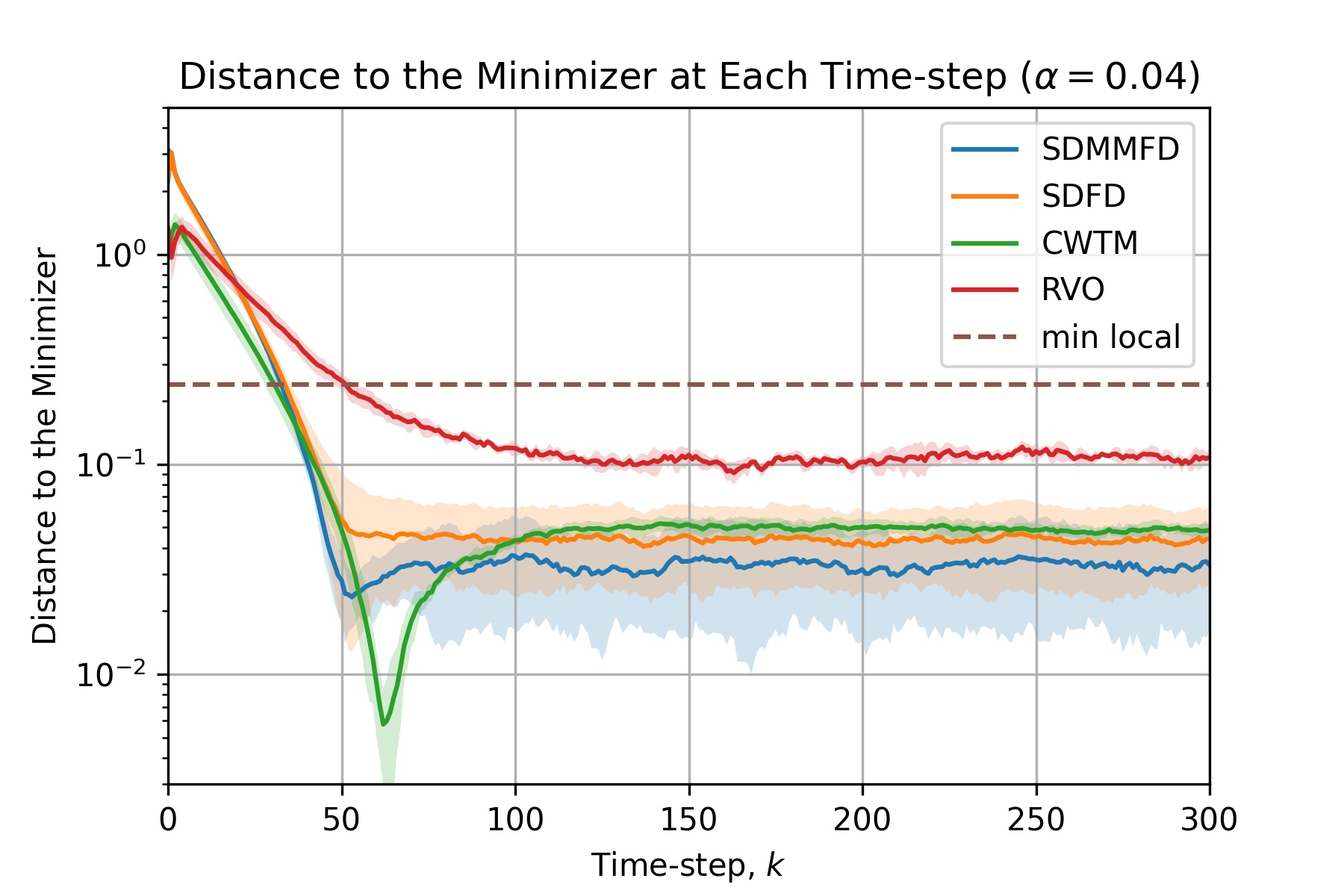} 
\label{fig: distance}}\;\;
\subfloat[The optimality gap evaluated at the average of the regular agents' states $f (\Bar{\x}) - f^*$ for each algorithm.]
{\includegraphics[width=.45\textwidth]{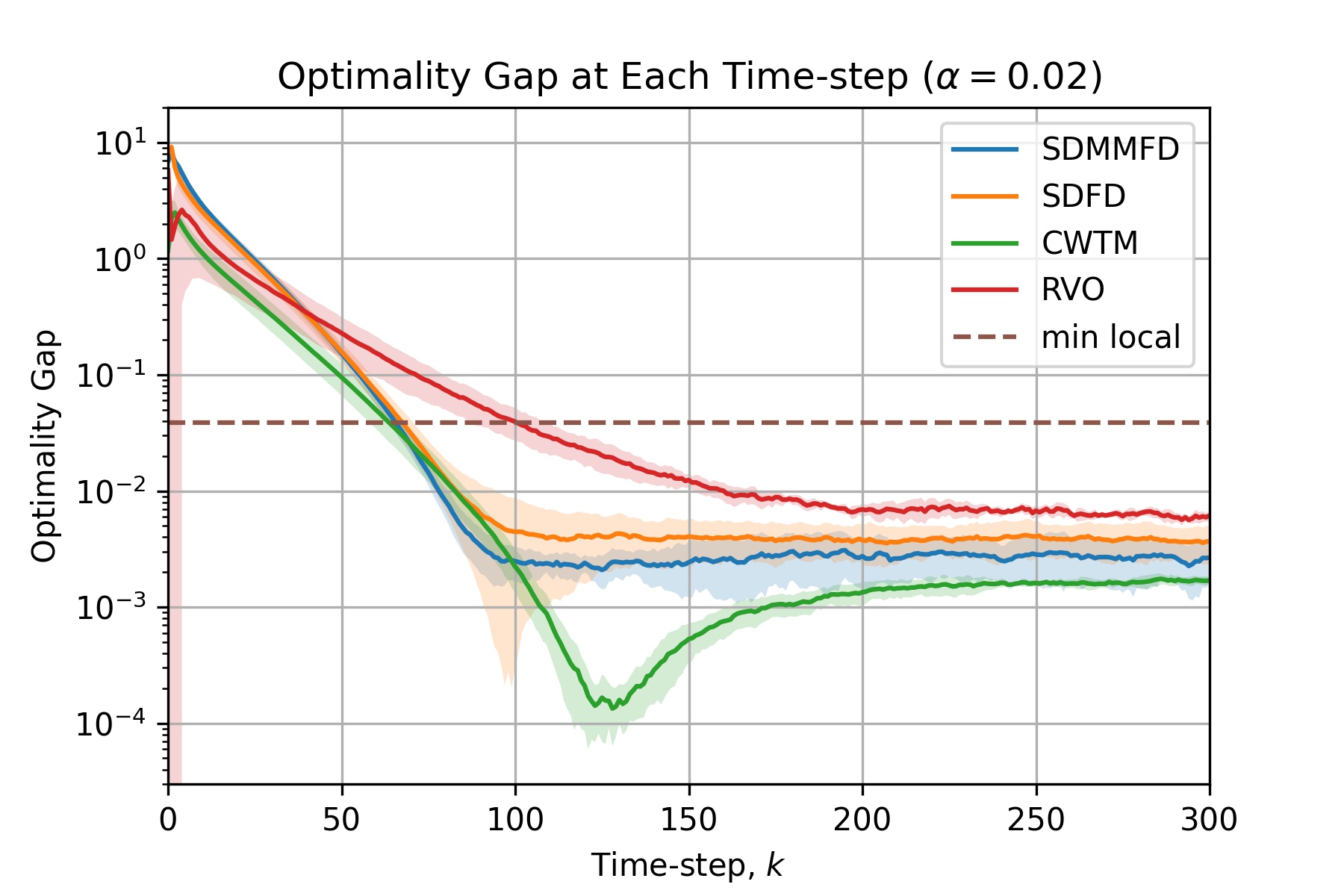}
\includegraphics[width=.45\textwidth]{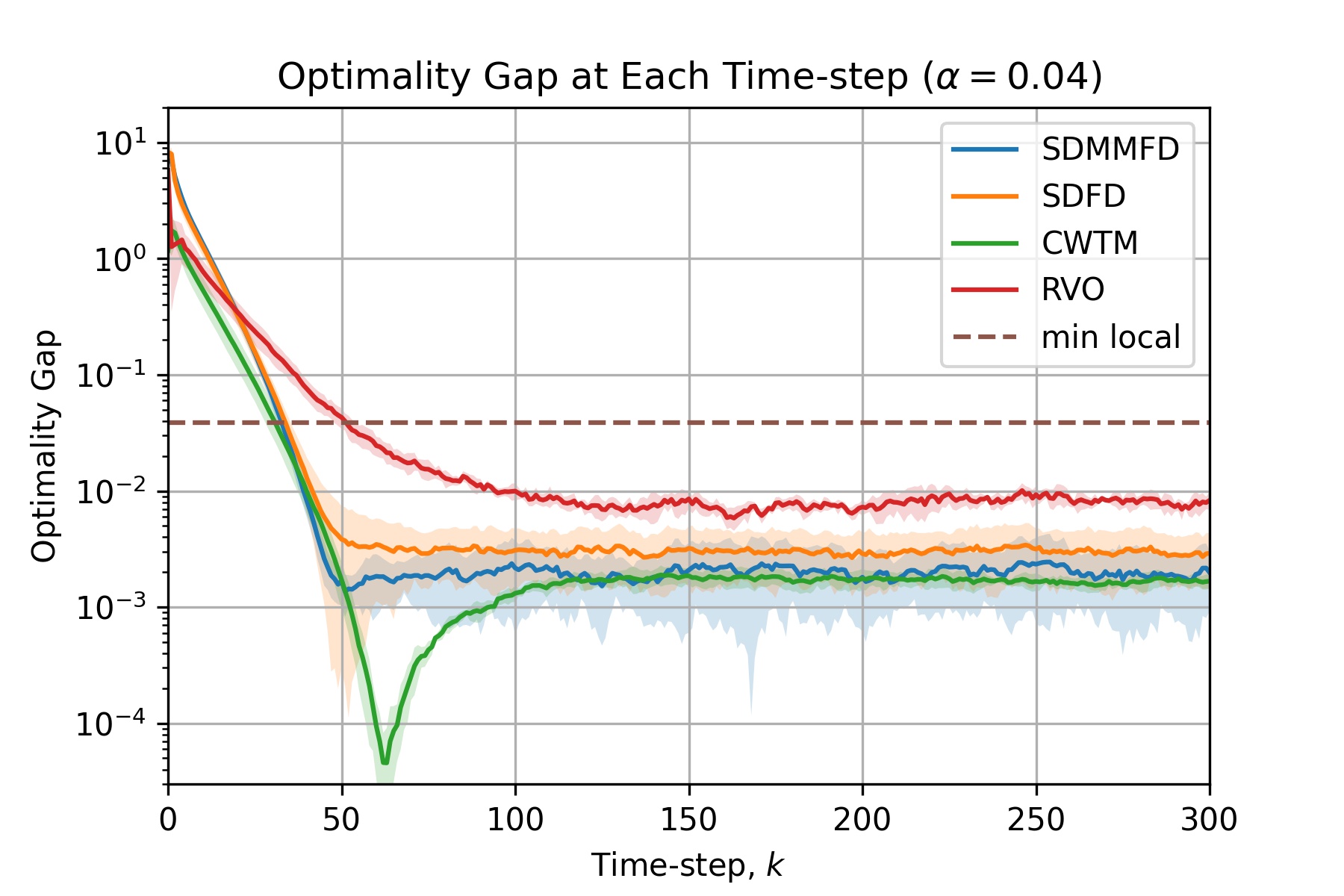}
\label{fig: optimality}}\;\;
\subfloat[The maximum Euclidean distance between two regular agents' states (regular states' diameter) $\max_{v_i, v_j \in \vr} \| \x_i - \x_j \|$ for each algorithm.]
{\includegraphics[width=.45\textwidth]{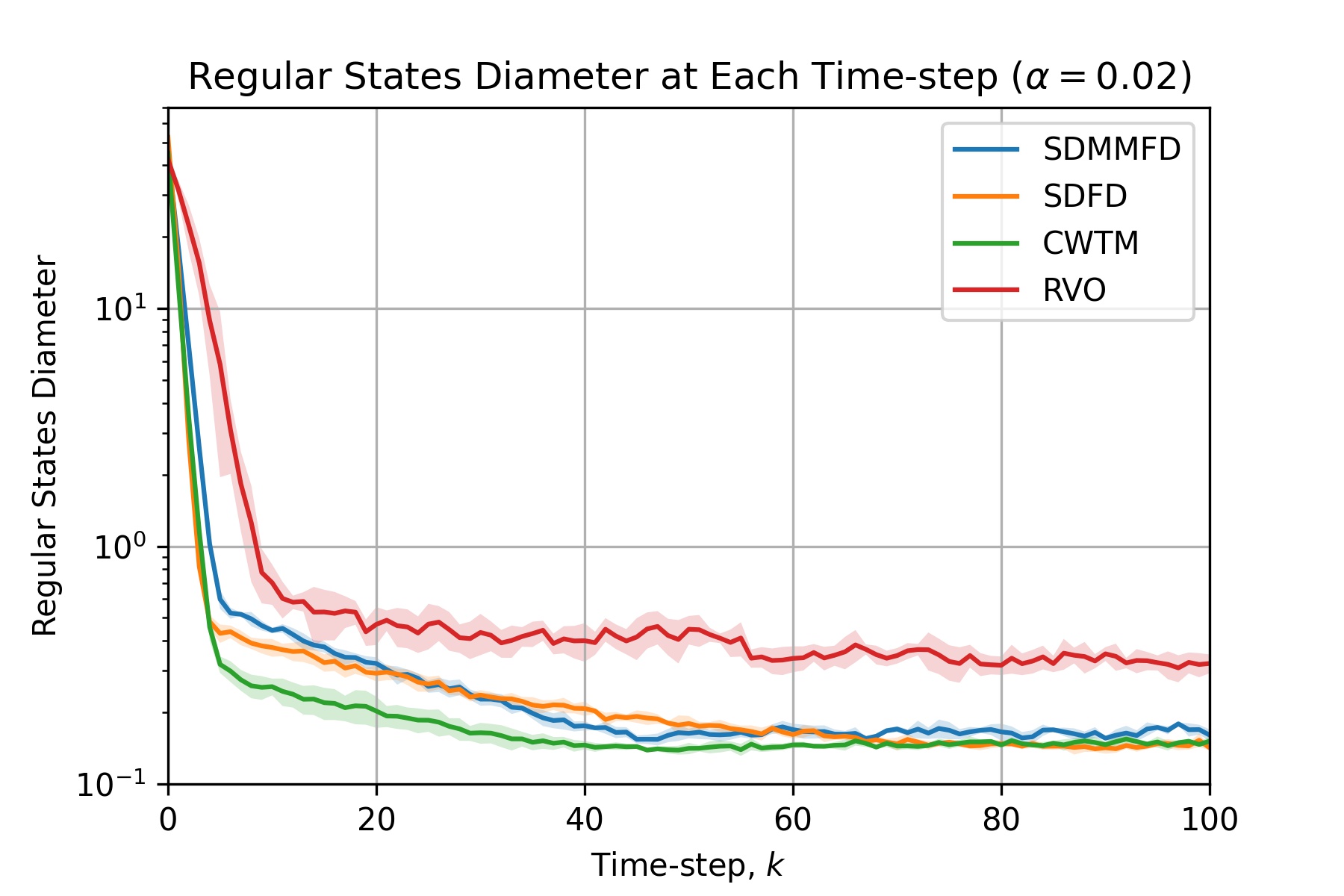}
\includegraphics[width=.45\textwidth]{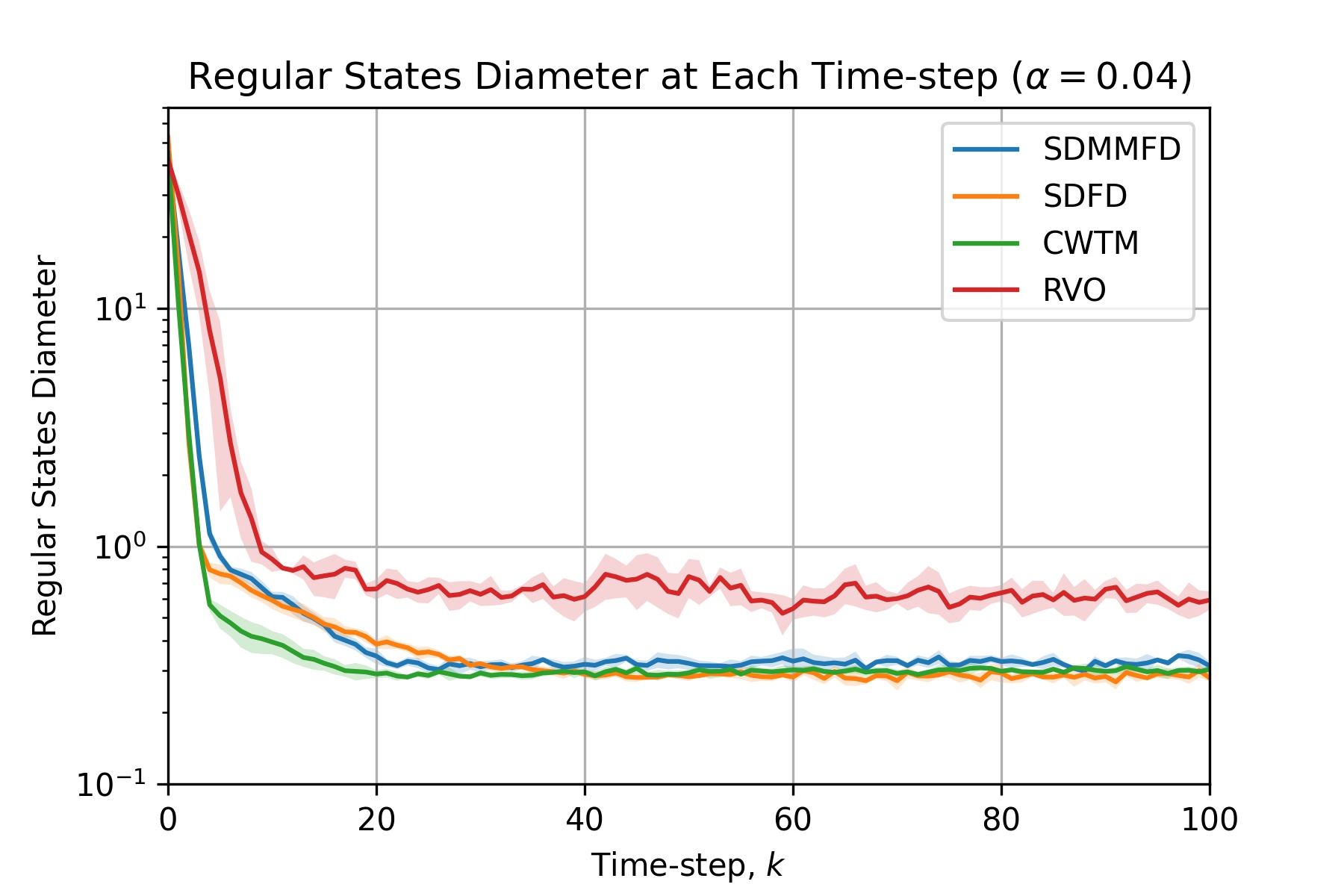}
\label{fig: diameter}}
\caption{The plots show the results obtained from SDMMFD (blue), SDFD (orange), CWTM (green), and RVO (red) for given constant step-sizes $\alpha = 0.02$ (left), and $\alpha = 0.04$ (right).}
\label{fig: exp_quadratic}
\end{figure}

Let $\bar{\x} [k] = \frac{1}{| \vr |} \sum_{v_i \in \vr} \x_i [k]$ be the average of the states over regular agents at time-step $k$. 
In Figure~\ref{fig: distance}, each solid curve represents the Euclidean distance from the average state to the true minimizer, i.e., $\| \bar{\x} [k] - \x^* \|$, while the dotted line labeled min\_local denotes the minimum Euclidean distance from the minimizers of the local functions to the true minimizer across all regular agents, i.e., $\min_{v_i \in \vr} \| \x_i^* - \x^* \|$. 
In Figure~\ref{fig: optimality}, each solid curve corresponds to the optimality gap computed at the average state, i.e., $f ( \bar{\x} [k] ) - f^*$, and the dotted line labeled min\_local indicates the minimum optimality gap computed at the local function minimizers among all regular agents, i.e., $\min_{v_i \in \vr} f( \x_i^* ) - f^*$. 
In Figure~\ref{fig: diameter}, each solid curve represents the maximum Euclidean distance between the states of any two regular agents (regular agents' diameter), i.e., $\max_{v_i, v_j \in \vr} \| \x_i[k] - \x_j[k] \|$. 
In Figures~\ref{fig: distance} to \ref{fig: diameter}, the solid curves are the means over all experiment rounds, and the shaded regions represent a $\pm 1$ standard deviation from the means.

As we can see from Figure~\ref{fig: distance}, for both cases, $\alpha = 0.02$ and $\alpha = 0.04$, the distances to the true minimizer drop at geometric rates in the early time-steps. However, in the later time-steps, these distances remain relatively stable, albeit with slight oscillations of small magnitudes. This behavior corresponds to the regular nodes' states entering the convergence region. Furthermore, it is notable that the convergence rates of all algorithms decrease, indicating faster convergence as the constant step-size $\alpha$ transitions from $0.02$ to $0.04$, which aligns with the predictions from Theorem~\ref{thm: main-convergence}. Still, for each algorithm, the distances to the minimizer at later time-steps are similar for both constant step-sizes. Analogous patterns can be observed in the optimality gaps, mirroring the trends seen in the distance to the minimizer, as illustrated in Figure~\ref{fig: optimality}.

In Figure~\ref{fig: diameter}, representing the diameters of the regular states, it is important to note the shorter time horizon for the time-step $k$. The diameters initially drop at a geometric rate during the early time-steps, followed by relatively stable values. A slightly faster convergence can be observed for higher (constant) step-sizes, with the values ceasing to drop at around $k = 60$ for $\alpha = 0.02$ and $k = 40$ for $\alpha = 0.04$. Additionally, the diameter at later time-steps is evidently higher for larger step-sizes compared to smaller ones. These results align well with the discussion in subsection~\ref{subsec: consensus} regarding approximate consensus, particularly Theorem~\ref{thm: consensus} and Figure~\ref{fig: analysis_cons-diam}.

It is noteworthy that, even though Theorem~\ref{thm: contraction} implies a contraction factor of $\gamma = d$ for CWTM (with $d = 2$ in this case), and $\gamma = 1$ for SDMMFD, SDFD, and RVO, for most of the time-steps in this experiment, the states contraction property \eqref{eqn: contraction def} for all algorithms holds with $\gamma$ approximately less than $1$. The insensitivity of distances to the minimizer at later time-steps with respect to the step-size $\alpha$ is attributed to the flat region observed in Figure~\ref{fig: conv-radius}, corresponding to $\gamma < 1$.  In conclusion, the observed behaviors of all considered algorithms concerning the step-size are comprehensively explained by the theories presented in section~\ref{sec: assumption result}.

From Figure~\ref{fig: distance} and \ref{fig: optimality}, we observe that although some algorithms exhibit worse distances to the minimizer compared to the local minimizers (i.e., comparing the solid curves to the dotted line), all algorithms are likely to achieve better optimality gaps than the local minimizers. Notably, all local minimizers fall within the convergence region $\B (\xc, R^*)$. Additionally, when comparing the algorithms, SDFD shows higher uncertainty regarding distances to the minimizer and optimality gaps due to its vulnerability to hostile state vectors. Interestingly, RVO might have a higher effective contraction factor $\gamma$ for this experiment, resulting in slower convergence and worse final values compared to the others across all metrics, as shown in Figure~\ref{fig: exp_quadratic}.
\section{Conclusions} 
\label{sec: conclusion}

In this work, we considered the (peer-to-peer) distributed optimization problem in the presence of Byzantine agents. We introduced a general resilient (peer-to-peer) DGD framework called R{\scriptsize{ED}}G{\scriptsize{RAF}} which includes some state-of-the-art resilient algorithms such as SDMMFD \cite{kuwaran2020byzantine_ACC, kuwaran2024scalable}, SDFD \cite{kuwaran2024scalable}, CWTM \cite{sundaram2018distributed, su2015byzantine, su2020byzantine, fu2021resilient, zhao2019resilient, fang2022bridge}, and RVO \cite{park2017fault, abbas2022resilient} as special cases. We analyzed the convergence of algorithms captured by our framework, assuming they satisfy a certain states contraction property. In particular, we derived a geometric rate of convergence of all regular agents to the convergence region under the strong convexity of the local functions and constant step-size regime. As we have shown, the convergence region, in fact, contains the true minimizer (the minimizer of the sum of the regular agents' functions). In addition, given a mixing dynamics property, we also derived a geometric rate of convergence of all regular agents to approximate consensus with a certain diameter under similar conditions. Considering each resilient algorithm, we analyzed the states contraction and mixing dynamics properties which, in turn, dictate the convergence rates, the size of the convergence region, and the approximate consensus diameter.

Future work includes developing resilient algorithms satisfying both the states contraction and mixing dynamics properties which give fast rates of convergence as well as a small convergence region and small approximate consensus diameter, identifying other properties for resilient algorithms to achieve good performance, analyzing the convergence property of other existing algorithms from the literature, considering non-convex functions with certain properties, and establishing a tight lower bound for the convergence region.

\bibliographystyle{unsrtnat}
\bibliography{reference}

\begin{thebibliography}{44}
\providecommand{\natexlab}[1]{#1}
\providecommand{\url}[1]{\texttt{#1}}
\expandafter\ifx\csname urlstyle\endcsname\relax
  \providecommand{\doi}[1]{doi: #1}\else
  \providecommand{\doi}{doi: \begingroup \urlstyle{rm}\Url}\fi

\bibitem[Nedic and Ozdaglar(2009)]{nedic2009distributed}
Angelia Nedic and Asuman Ozdaglar.
\newblock Distributed subgradient methods for multi-agent optimization.
\newblock \emph{IEEE Trans. Automat. Control}, 54\penalty0 (1):\penalty0 48--61, 2009.
\newblock URL \url{https://dx.doi.org/10.1109/TAC.2008.2009515}.

\bibitem[Duchi et~al.(2012)Duchi, Agarwal, and Wainwright]{duchi2011dual}
John~C. Duchi, Alekh Agarwal, and Martin~J. Wainwright.
\newblock Dual averaging for distributed optimization: Convergence analysis and network scaling.
\newblock \emph{IEEE Trans. Automat. Control}, 57\penalty0 (3):\penalty0 592--606, 2012.
\newblock URL \url{https://dx.doi.org/10.1109/TAC.2011.2161027}.

\bibitem[Shi et~al.(2014)Shi, Ling, Yuan, Wu, and Yin]{shi2014linear}
Wei Shi, Qing Ling, Kun Yuan, Gang Wu, and Wotao Yin.
\newblock On the linear convergence of the {ADMM} in decentralized consensus optimization.
\newblock \emph{IEEE Trans. Signal Process.}, 62\penalty0 (7):\penalty0 1750--1761, 2014.
\newblock URL \url{https://dx.doi.org/10.1109/TSP.2014.2304432}.

\bibitem[Nedi\'{c} et~al.(2017)Nedi\'{c}, Olshevsky, and Shi]{nedic2017achieving}
Angelia Nedi\'{c}, Alex Olshevsky, and Wei Shi.
\newblock Achieving geometric convergence for distributed optimization over time-varying graphs.
\newblock \emph{SIAM J. Optim.}, 27\penalty0 (4):\penalty0 2597--2633, 2017.
\newblock URL \url{https://dx.doi.org/10.1137/16M1084316}.

\bibitem[Pu et~al.(2021)Pu, Shi, Xu, and Nedić]{pu2020push}
Shi Pu, Wei Shi, Jinming Xu, and Angelia Nedić.
\newblock Push–pull gradient methods for distributed optimization in networks.
\newblock \emph{IEEE Trans. on Automat. Control}, 66\penalty0 (1):\penalty0 1--16, 2021.
\newblock URL \url{https://dx.doi.org/10.1109/TAC.2020.2972824}.

\bibitem[Nedić et~al.(2018)Nedić, Olshevsky, and Rabbat]{nedic2018network}
Angelia Nedić, Alex Olshevsky, and Michael~G. Rabbat.
\newblock Network topology and communication-computation tradeoffs in decentralized optimization.
\newblock \emph{Proc. IEEE}, 106\penalty0 (5):\penalty0 953--976, 2018.
\newblock URL \url{https://dx.doi.org/10.1109/JPROC.2018.2817461}.

\bibitem[Yang et~al.(2019)Yang, Yi, Wu, Yuan, Wu, Meng, Hong, Wang, Lin, and Johansson]{yang2019survey}
Tao Yang, Xinlei Yi, Junfeng Wu, Ye~Yuan, Di~Wu, Ziyang Meng, Yiguang Hong, Hong Wang, Zongli Lin, and Karl~H. Johansson.
\newblock A survey of distributed optimization.
\newblock \emph{Annu. Rev. Control}, 47:\penalty0 278--305, 2019.
\newblock ISSN 1367-5788.
\newblock URL \url{https://dx.doi.org/10.1016/j.arcontrol.2019.05.006}.

\bibitem[Xin et~al.(2020)Xin, Pu, Nedić, and Khan]{xin2020general}
Ran Xin, Shi Pu, Angelia Nedić, and Usman~A. Khan.
\newblock A general framework for decentralized optimization with first-order methods.
\newblock \emph{Proc. IEEE}, 108\penalty0 (11):\penalty0 1869--1889, 2020.
\newblock URL \url{https://dx.doi.org/10.1109/JPROC.2020.3024266}.

\bibitem[Sundaram and Gharesifard(2019)]{sundaram2018distributed}
Shreyas Sundaram and Bahman Gharesifard.
\newblock Distributed optimization under adversarial nodes.
\newblock \emph{IEEE Trans. Automat. Control}, 64\penalty0 (3):\penalty0 1063--1076, 2019.
\newblock URL \url{https://dx.doi.org/10.1109/TAC.2018.2836919}.

\bibitem[Su and Vaidya(2015)]{su2015byzantine}
Lili Su and Nitin Vaidya.
\newblock Byzantine {M}ulti-agent {O}ptimization: Part {I}.
\newblock \emph{preprint}, 2015.
\newblock URL \url{https://arxiv.org/abs/1506.04681}.

\bibitem[Fang et~al.(2022)Fang, Yang, and Bajwa]{fang2022bridge}
Cheng Fang, Zhixiong Yang, and Waheed~U. Bajwa.
\newblock {BRIDGE}: Byzantine-resilient decentralized gradient descent.
\newblock \emph{IEEE Trans. Signal and Inform. Process. Netw.}, 8:\penalty0 610--626, 2022.
\newblock URL \url{https://dx.doi.org/10.1109/TSIPN.2022.3188456}.

\bibitem[Yang et~al.(2020)Yang, Gang, and Bajwa]{yang2020adversary}
Zhixiong Yang, Arpita Gang, and Waheed~U. Bajwa.
\newblock Adversary-resilient distributed and decentralized statistical inference and machine learning: An overview of recent advances under the {B}yzantine threat model.
\newblock \emph{IEEE Signal Process. Mag.}, 37\penalty0 (3):\penalty0 146--159, 2020.
\newblock URL \url{https://dx.doi.org/10.1109/MSP.2020.2973345}.

\bibitem[Su and Vaidya(2021)]{su2020byzantine}
Lili Su and Nitin~H. Vaidya.
\newblock Byzantine-resilient multiagent optimization.
\newblock \emph{IEEE Trans. Automat. Control}, 66\penalty0 (5):\penalty0 2227--2233, 2021.
\newblock URL \url{https://dx.doi.org/10.1109/TAC.2020.3008139}.

\bibitem[Fu et~al.(2021)Fu, Ma, Qin, and Kang]{fu2021resilient}
Weiming Fu, Qichao Ma, Jiahu Qin, and Yu~Kang.
\newblock Resilient consensus-based distributed optimization under deception attacks.
\newblock \emph{Internat. J. Robust Nonlinear Control}, 31\penalty0 (6):\penalty0 1803--1816, 2021.
\newblock URL \url{https://dx.doi.org/10.1002/rnc.5026}.

\bibitem[Zhao et~al.(2020)Zhao, He, and Wang]{zhao2019resilient}
Chengcheng Zhao, Jianping He, and Qing-Guo Wang.
\newblock Resilient distributed optimization algorithm against adversarial attacks.
\newblock \emph{IEEE Trans. Automat. Control}, 65\penalty0 (10):\penalty0 4308--4315, 2020.
\newblock URL \url{https://dx.doi.org/10.1109/TAC.2019.2954363}.

\bibitem[Yang and Bajwa(2019)]{yang2019byrdie}
Zhixiong Yang and Waheed~U. Bajwa.
\newblock {ByRDiE}: Byzantine-resilient distributed coordinate descent for decentralized learning.
\newblock \emph{IEEE Trans. Signal Inform. Process. Netw.}, 5\penalty0 (4):\penalty0 611--627, 2019.
\newblock URL \url{https://dx.doi.org/10.1109/TSIPN.2019.2928176}.

\bibitem[Ravi and Scaglione(2019)]{ravi2019detection}
Nikhil Ravi and Anna Scaglione.
\newblock Detection and isolation of adversaries in decentralized optimization for non-strongly convex objectives.
\newblock \emph{IFAC-PapersOnLine}, 52\penalty0 (20):\penalty0 381--386, 2019.
\newblock ISSN 2405-8963.
\newblock URL \url{https://dx.doi.org/10.1016/j.ifacol.2019.12.185}.

\bibitem[Nedić and Olshevsky(2015)]{nedic2014distributed}
Angelia Nedić and Alex Olshevsky.
\newblock Distributed optimization over time-varying directed graphs.
\newblock \emph{IEEE Trans. Automat. Control}, 60\penalty0 (3):\penalty0 601--615, 2015.
\newblock URL \url{https://dx.doi.org/10.1109/TAC.2014.2364096}.

\bibitem[Kuwaranancharoen et~al.(2020)Kuwaranancharoen, Xin, and Sundaram]{kuwaran2020byzantine_ACC}
Kananart Kuwaranancharoen, Lei Xin, and Shreyas Sundaram.
\newblock Byzantine-resilient distributed optimization of multi-dimensional functions.
\newblock In \emph{2020 American Control Conference (ACC)}, pages 4399--4404. IEEE, Piscataway, NJ, 2020.
\newblock URL \url{https://dx.doi.org/10.23919/ACC45564.2020.9147396}.

\bibitem[Kuwaranancharoen et~al.(2024)Kuwaranancharoen, Xin, and Sundaram]{kuwaran2024scalable}
Kananart Kuwaranancharoen, Lei Xin, and Shreyas Sundaram.
\newblock Scalable distributed optimization of multi-dimensional functions despite {B}yzantine adversaries.
\newblock \emph{IEEE Trans. Signal Inform. Process. Netw.}, 10:\penalty0 360--375, 2024.
\newblock URL \url{https://dx.doi.org/10.1109/TSIPN.2024.3379844}.

\bibitem[Gupta et~al.(2021)Gupta, Doan, and Vaidya]{gupta2021byzantine}
Nirupam Gupta, Thinh~T. Doan, and Nitin~H. Vaidya.
\newblock Byzantine fault-tolerance in decentralized optimization under 2f-redundancy.
\newblock In \emph{2021 American Control Conference (ACC)}, pages 3632--3637, Piscataway, NJ, 2021. IEEE.
\newblock URL \url{https://dx.doi.org/10.23919/ACC50511.2021.9483067}.

\bibitem[Peng et~al.(2021)Peng, Li, and Ling]{peng2021byzantine}
Jie Peng, Weiyu Li, and Qing Ling.
\newblock Byzantine-robust decentralized stochastic optimization over static and time-varying networks.
\newblock \emph{Signal Process.}, 183:\penalty0 108020, 2021.
\newblock ISSN 0165-1684.
\newblock URL \url{https://dx.doi.org/10.1016/j.sigpro.2021.108020}.

\bibitem[Ben-Ameur et~al.(2016)Ben-Ameur, Bianchi, and Jakubowicz]{ben2015robust}
Walid Ben-Ameur, Pascal Bianchi, and Jérémie Jakubowicz.
\newblock Robust distributed consensus using total variation.
\newblock \emph{IEEE Trans. Automat. Control}, 61\penalty0 (6):\penalty0 1550--1564, 2016.
\newblock URL \url{https://dx.doi.org/10.1109/TAC.2015.2471755}.

\bibitem[Guo et~al.(2022)Guo, Zhang, Yu, Xie, Ma, Xiang, and Liu]{guo2021byzantine}
Shangwei Guo, Tianwei Zhang, Han Yu, Xiaofei Xie, Lei Ma, Tao Xiang, and Yang Liu.
\newblock Byzantine-resilient decentralized stochastic gradient descent.
\newblock \emph{IEEE Trans. on Circuits. Syst. Video Technol.}, 32\penalty0 (6):\penalty0 4096--4106, 2022.
\newblock URL \url{https://dx.doi.org/10.1109/TCSVT.2021.3116976}.

\bibitem[Elkordy et~al.(2022)Elkordy, Prakash, and Avestimehr]{elkordy2022basil}
Ahmed~Roushdy Elkordy, Saurav Prakash, and Salman Avestimehr.
\newblock Basil: A fast and {B}yzantine-resilient approach for decentralized training.
\newblock \emph{IEEE J. Sel. Area. Comm.}, 40\penalty0 (9):\penalty0 2694--2716, 2022.
\newblock URL \url{https://dx.doi.org/10.1109/JSAC.2022.3191347}.

\bibitem[Wu et~al.(2023)Wu, Chen, and Ling]{wu2023byzantine}
Zhaoxian Wu, Tianyi Chen, and Qing Ling.
\newblock Byzantine-resilient decentralized stochastic optimization with robust aggregation rules.
\newblock \emph{IEEE Trans. Signal Process.}, 71:\penalty0 3179--3195, 2023.
\newblock URL \url{https://dx.doi.org/10.1109/TSP.2023.3300629}.

\bibitem[He et~al.(2023)He, Karimireddy, and Jaggi]{he2023byzantinerobust}
Lie He, Sai~Praneeth Karimireddy, and Martin Jaggi.
\newblock Byzantine-{R}obust {D}ecentralized {L}earning via {C}lipped{G}ossip.
\newblock \emph{preprint}, 2023.
\newblock URL \url{https://arxiv.org/abs/2202.01545}.

\bibitem[Park and Hutchinson(2017)]{park2017fault}
Hyongju Park and Seth~A. Hutchinson.
\newblock Fault-tolerant rendezvous of multirobot systems.
\newblock \emph{IEEE Trans. Robot.}, 33\penalty0 (3):\penalty0 565--582, 2017.
\newblock URL \url{https://dx.doi.org/10.1109/TRO.2017.2658604}.

\bibitem[Abbas et~al.(2022)Abbas, Shabbir, Li, and Koutsoukos]{abbas2022resilient}
Waseem Abbas, Mudassir Shabbir, Jiani Li, and Xenofon Koutsoukos.
\newblock Resilient distributed vector consensus using centerpoint.
\newblock \emph{Automatica J. IFAC}, 136:\penalty0 110046, 2022.
\newblock ISSN 0005-1098.
\newblock URL \url{https://dx.doi.org/10.1016/j.automatica.2021.110046}.

\bibitem[LeBlanc et~al.(2013)LeBlanc, Zhang, Koutsoukos, and Sundaram]{leblanc2013resilient}
Heath~J. LeBlanc, Haotian Zhang, Xenofon Koutsoukos, and Shreyas Sundaram.
\newblock Resilient asymptotic consensus in robust networks.
\newblock \emph{IEEE J. Sel. Area. Comm.}, 31\penalty0 (4):\penalty0 766--781, 2013.
\newblock URL \url{https://dx.doi.org/10.1109/JSAC.2013.130413}.

\bibitem[Cao et~al.(2008)Cao, Morse, and Anderson]{cao2008reaching}
Ming Cao, A.~Stephen Morse, and Brian D.~O. Anderson.
\newblock Reaching a consensus in a dynamically changing environment: A graphical approach.
\newblock \emph{SIAM J. Control Optim.}, 47\penalty0 (2):\penalty0 575--600, 2008.
\newblock URL \url{https://dx.doi.org/10.1137/060657005}.

\bibitem[Lynch(1996)]{lynch1996distributed}
Nancy~A. Lynch.
\newblock \emph{Distributed Algorithms}.
\newblock Morgan Kaufmann Publishers Inc., San Francisco, CA, USA, 1996.
\newblock ISBN 9780080504704.
\newblock URL \url{https://dl.acm.org/doi/book/10.5555/2821576}.

\bibitem[Yin et~al.(2018)Yin, Chen, Kannan, and Bartlett]{yin2018byzantine}
Dong Yin, Yudong Chen, Ramchandran Kannan, and Peter Bartlett.
\newblock {B}yzantine-robust distributed learning: Towards optimal statistical rates.
\newblock In \emph{Proc. Mach. Learn. Res. (PMLR)}, volume~80, pages 5650--5659, 2018.
\newblock URL \url{https://proceedings.mlr.press/v80/yin18a.html}.

\bibitem[Farhadkhani et~al.(2022)Farhadkhani, Guerraoui, Gupta, Pinot, and Stephan]{farhadkhani2022byzantine}
Sadegh Farhadkhani, Rachid Guerraoui, Nirupam Gupta, Rafael Pinot, and John Stephan.
\newblock {B}yzantine machine learning made easy by resilient averaging of momentums.
\newblock In \emph{Proc. Mach. Learn. Res. (PMLR)}, volume 162, pages 6246--6283, 2022.
\newblock URL \url{https://proceedings.mlr.press/v162/farhadkhani22a.html}.

\bibitem[Yuan et~al.(2016)Yuan, Ling, and Yin]{yuan2016convergence}
Kun Yuan, Qing Ling, and Wotao Yin.
\newblock On the convergence of decentralized gradient descent.
\newblock \emph{SIAM J. Optim.}, 26\penalty0 (3):\penalty0 1835--1854, 2016.
\newblock URL \url{https://dx.doi.org/10.1137/130943170}.

\bibitem[Hendrikx et~al.(2019)Hendrikx, Bach, and Massoulie]{hendrikx2019accelerated}
Hadrien Hendrikx, Francis Bach, and Laurent Massoulie.
\newblock Accelerated decentralized optimization with local updates for smooth and strongly convex objectives.
\newblock In \emph{Proc. Mach. Learn. Res. (PMLR)}, volume~89, pages 897--906, 2019.
\newblock URL \url{https://proceedings.mlr.press/v89/hendrikx19a.html}.

\bibitem[Kovalev et~al.(2021)Kovalev, Koloskova, Jaggi, Richtarik, and Stich]{kovalev2021linearly}
Dmitry Kovalev, Anastasia Koloskova, Martin Jaggi, Peter Richtarik, and Sebastian Stich.
\newblock A linearly convergent algorithm for decentralized optimization: Sending less bits for free!
\newblock In \emph{Proc. Mach. Learn. Res. (PMLR)}, volume 130, pages 4087--4095, 2021.
\newblock URL \url{https://proceedings.mlr.press/v130/kovalev21a.html}.

\bibitem[Gutman and Pena(2021)]{gutman2021condition}
David~H Gutman and Javier~F Pena.
\newblock The condition number of a function relative to a set.
\newblock \emph{Math. Program.}, 188:\penalty0 255--294, 2021.
\newblock URL \url{https://dx.doi.org/10.1007/s10107-020-01510-4}.

\bibitem[Kuwaranancharoen and Sundaram(2018)]{kuwaran2018location}
Kananart Kuwaranancharoen and Shreyas Sundaram.
\newblock On the location of the minimizer of the sum of two strongly convex functions.
\newblock In \emph{2018 IEEE Conference on Decision and Control (CDC)}, pages 1769--1774. IEEE, Piscataway, NJ, 2018.
\newblock URL \url{https://dx.doi.org/10.1109/CDC.2018.8619735}.

\bibitem[Kuwaranancharoen and Sundaram(2020)]{kuwaran2020set}
Kananart Kuwaranancharoen and Shreyas Sundaram.
\newblock On the set of possible minimizers of a sum of known and unknown functions.
\newblock In \emph{2020 American Control Conference (ACC)}, pages 106--111. IEEE, Piscataway, NJ, 2020.
\newblock URL \url{https://dx.doi.org/10.23919/ACC45564.2020.9147407}.

\bibitem[Kuwaranancharoen and Sundaram(to appear)]{kuwaran2023minimizer}
Kananart Kuwaranancharoen and Shreyas Sundaram.
\newblock The minimizer of the sum of two strongly convex functions.
\newblock \emph{Optimization}, to appear.
\newblock URL \url{https://dx.doi.org/10.1080/02331934.2024.2402923}.

\bibitem[Zamani et~al.(2024)Zamani, Glineur, and Hendrickx]{zamani2024set}
Moslem Zamani, Francois Glineur, and Julien~M. Hendrickx.
\newblock On the set of possible minimizers of a sum of convex functions.
\newblock \emph{IEEE Control Syst. Lett.}, 8:\penalty0 1871--1876, 2024.
\newblock URL \url{https://dx.doi.org/10.1109/LCSYS.2024.3414378}.

\bibitem[Zhou(2018)]{zhou2018fenchel}
Xingyu Zhou.
\newblock On the {F}enchel {D}uality between {S}trong {C}onvexity and {L}ipschitz {C}ontinuous {G}radient.
\newblock \emph{preprint}, 2018.
\newblock URL \url{https://arxiv.org/abs/1803.06573}.

\bibitem[Castano et~al.(2016)Castano, Paksoy, and Zhang]{castano2016angles}
Diego Castano, Vehbi~E. Paksoy, and Fuzhen Zhang.
\newblock Angles, triangle inequalities, correlation matrices and metric-preserving and subadditive functions.
\newblock \emph{Linear Algebra Appl.}, 491:\penalty0 15--29, 2016.
\newblock ISSN 0024-3795.
\newblock URL \url{https://dx.doi.org/10.1016/j.laa.2014.10.011}.

\end{thebibliography}

\newpage
\appendix
\section{Additional Lemmas}

\subsection{Graph Robustness}

We now state a lemma regarding transmissibility of information after dropping some edges from the graph (from Lemma~2.3 in \cite{sundaram2018distributed}).
\begin{lemma}  \label{lem: rooted subgraph2}
Suppose $r \in \Z_+$ and $\G$ is $r$-robust. Let $\G'$ be a graph obtained by removing $r - 1$ or fewer incoming edges from each node in $\G$. Then $\G'$ is rooted.
\end{lemma}

This means that if we have enough redundancy in the network, information from at least one node can still flow to the other nodes in the network even after each regular node discards some neighboring states. 

\subsection{Series of Products}

\begin{lemma} \label{lem: limsup liminf sum of prod}
Suppose $\{ b[k] \}_{k \in \N} \subset \R$ is a sequence such that $\lim_{k \to \infty} \sum_{s=0}^k b[s]$ $= b$.
\begin{itemize}
    \item If $\{ a[k] \}_{k \in \N} \subset \R$ is a sequence such that $\limsup_k a[k] = a^*$, then it holds that
    \begin{equation}
        \limsup_k \sum_{s=0}^k a[s] b[k-s] \leq a^* b.
        \label{eqn: limsup sum of prod}
    \end{equation}
    \item If $\{ a[k] \}_{k \in \N} \subset \R$ is a sequence such that $\liminf_k a[k] = a_*$, then it holds that
    \begin{equation}
        \liminf_k \sum_{s=0}^k a[s] b[k-s] \geq a_* b.
        \label{eqn: liminf sum of prod}
    \end{equation}
\end{itemize}
\end{lemma}

\begin{proof}
Consider the first part of the lemma. Since $\limsup_k a[k] = a^*$, we have that for a given $\epsilon \in \R_{>0}$, there exists $k'(\epsilon) \in \N$ such that $a[k] \leq a^* + \epsilon$ for all $k \geq k' (\epsilon)$. Suppose $\Tilde{k} \in \N$ and $\Tilde{k} \geq k' (\epsilon)$, and we can write
\begin{align}
    \sum_{s=0}^{\Tilde{k}} a[s] b[\Tilde{k}-s]
    &= \sum_{s=0}^{k'(\epsilon) - 1} a[s] b[\Tilde{k}-s] + \sum_{s=k'(\epsilon)}^{\Tilde{k}} a[s] b[\Tilde{k}-s] \nonumber \\
    &\leq \sum_{s=0}^{k'(\epsilon) - 1} a[s] b[\Tilde{k}-s] + (a^* + \epsilon) \sum_{s=0}^{\Tilde{k} - k'(\epsilon)} b[s].
    \label{eqn: series bound 2}
\end{align}
Since $\lim_{k \to \infty} \sum_{s=0}^k b[s] = b$, we have that $\lim_{k \to \infty} b[k] = 0$. Taking $\limsup_{\Tilde{k}}$ to both sides of \eqref{eqn: series bound 2}, we obtain that
\begin{align*}
    \limsup_{\Tilde{k}} \sum_{s=0}^{\Tilde{k}} a[s] b[\Tilde{k}-s] 
    &\leq \sum_{s=0}^{k'(\epsilon) - 1} a[s] \lim_{\Tilde{k} \to \infty} b[\Tilde{k}-s]  
    + (a^* + \epsilon) \lim_{\Tilde{k} \to \infty} \sum_{s=0}^{\Tilde{k} - k'(\epsilon)} b[s] \\
    &= (a^* + \epsilon) b.
\end{align*}
Since $\epsilon \in \R_{>0}$ can be chosen to be arbitrary small, we obtain \eqref{eqn: limsup sum of prod}. 

For the second part of the lemma, since $\liminf_k a[k] = a_*$, we have that for a given $\epsilon \in \R_{>0}$, there exists $k'(\epsilon) \in \N$ such that $a[k] \geq a_* + \epsilon$ for all $k \geq k' (\epsilon)$. By using this fact and taking the steps same as the proof above, we obtain \eqref{eqn: liminf sum of prod}.
\end{proof}

\begin{corollary} \label{cor: lim sum of prod}
Suppose $\{ a[k] \}_{k \in \N} \subset \R$ is a sequence such that $\lim_{k \to \infty} a[k]$ $= 0$ and $\{ b[k] \}_{k \in \N} \subset \R$ is a sequence such that $\lim_{k \to \infty} \sum_{s=0}^k b[s]$ is finite. Then, it holds that
\begin{equation}
    \lim_{k \to \infty} \sum_{s=0}^k a[s] b[k-s] = 0. 
    \label{eqn: lim sum of prod}
\end{equation}
\end{corollary}
\begin{proof}
Since $\lim_{k \to \infty} a[k] = 0$, we can write
\begin{equation*}
    \limsup_k a[k] = \liminf_k a[k] = 0.
\end{equation*}
Using Lemma~\ref{lem: limsup liminf sum of prod}, we have that
\begin{equation*}
    0 \leq \liminf_k \sum_{s=0}^k a[s] b[k-s] \leq \limsup_k \sum_{s=0}^k a[s] b[k-s] \leq 0.
\end{equation*}
The inequalities above implies that
\begin{equation*}
    \liminf_k \sum_{s=0}^k a[s] b[k-s] = \limsup_k \sum_{s=0}^k a[s] b[k-s] = 0,
\end{equation*}
and the result \eqref{eqn: lim sum of prod} follows.
\end{proof}

\subsection{Function Analysis}
\label{subsec: func analysis}

\begin{lemma} \label{lem: R* decreasing}
Given a constant $\gamma \in \R_{\geq 0}$, suppose $h: \left( \max \big\{ 1 - \frac{1}{\gamma}, 0 \big\}, 1 \right] \to \R_{\geq 0}$ such that 
\begin{equation*}
    h(s) = \frac{\sqrt{s}}{1 - \sqrt{\gamma} \cdot \sqrt{1 - s}}. 
\end{equation*}
Then, the following statements hold.
\begin{itemize}
    \item If $\gamma \in [1, \infty)$, then $h$ is a strictly decreasing function.
    \item If $\gamma \in [0, 1)$, then $h$ is strictly increasing on the interval $(0, 1 - \gamma]$ and strictly decreasing on the interval $(1 - \gamma, 1]$.
\end{itemize}
\end{lemma}

\begin{proof}
Compute the derivative of $h$ with respect to $s$ yields
\begin{align*}
    h'(s) 
    &= \Big( \frac{1}{1 - \sqrt{\gamma} \sqrt{1 - s}} \Big)^2 
    \Big( \frac{1 - \sqrt{\gamma} \sqrt{1 - s}}{2 \sqrt{s}} - \frac{\sqrt{\gamma} \sqrt{s}}{2 \sqrt{1 - s}} \Big) \\
    &= \Big( \frac{1}{1 - \sqrt{\gamma} \sqrt{1 - s}} \Big)^2 
    \Big( \frac{\sqrt{1 - s} - \sqrt{\gamma}}{2 \sqrt{s} \sqrt{1-s}} \Big).
\end{align*}
From the expression $h'(s)$ above, we need to consider only the sign of $\sqrt{1 - s} - \sqrt{\gamma}$. First, consider the case that $\gamma \geq 1$. We have $s \in \big( 1 - \frac{1}{\gamma}, 1 \big]$ which implies that $\sqrt{1 - s} - \sqrt{\gamma} < 0$. Next, consider the case that $\gamma \in [0, 1)$. We have that $s \in (0, 1 - \gamma]$ implies that $\sqrt{1 - s} - \sqrt{\gamma} \geq 0$ with equality only if $s = 1 - \gamma$, and $s \in ( 1 - \gamma, 1]$ implies that $\sqrt{1 - s} - \sqrt{\gamma} < 0$.
\end{proof}
\section{Proof of Convergence Results in Subsection~\ref{subsec: convergence}}

\subsection{Convex Functions}

From \cite{zhou2018fenchel}, an equivalent definition of a $\mu$-strongly convex differentiable function $f$ is as follows: for all $\x_1$, $\x_2 \in \R^d$,
\begin{equation}
    \langle \nabla f ( \x_1 ) - \nabla f ( \x_2 ), \; \x_1 - \x_2 \rangle \geq \mu \| \x_1 - \x_2 \|^2.
    \label{def: strongly convex 2}
\end{equation}

We will use the following useful result from \cite{zhou2018fenchel} regarding the convexity and Lipschitz gradient of a function.
\begin{lemma}
If $f$ is convex and has $L$-Lipschitz gradient then for all $\x_1$, $\x_2 \in \R^d$,
\begin{equation}
    f (\x_1)
    \geq f (\x_2) + \langle \nabla f (\x_2), \x_1 - \x_2 \rangle
    + \frac{1}{2 L} \| \nabla f (\x_1) - \nabla f (\x_2) \|^2
    \label{def: lipschitz gradient 1}
\end{equation}
and
\begin{equation}
    f ( \x_1 ) 
    \leq f ( \x_2 ) + \langle \nabla f ( \x_2 ), \x_1 - \x_2 \rangle + \frac{L}{2} \| \x_1 - \x_2 \|^2. 
    \label{def: lipschitz gradient 2}
\end{equation}
\end{lemma} 

\subsection{The Reduction Property Implication}

We first introduce the following lemma which is useful for deriving the convergence result (Proposition~\ref{prop: convergence}).
\begin{lemma}  \label{lem: reduction}
Suppose Assumption~\ref{asm: convex} holds. If an algorithm $A$ in R{\scriptsize{ED}}G{\scriptsize{RAF}} satisfies the reduction property of Type-I or Type-II, then $\beta \sqrt{\gamma} < 1$.
\end{lemma}
\begin{proof}
In the first case, where $\gamma \in [0, 1)$ and $\alpha_k = \alpha \in \big( 0, \frac{1}{L} \big]$, we have $\beta \sqrt{\gamma} = \sqrt{\gamma} \cdot \sqrt{1 - \alpha \tilmu} < 1$. 
In the second case, since $\gamma \in \Big[ 1, \frac{1}{1 - \frac{\tilmu}{\tilL}} \Big)$, we have that $0 \leq \frac{1}{\tilmu} \big( 1 - \frac{1}{\gamma} \big) < \frac{1}{\tilL}$
which indicates that setting the step-size $\alpha_k = \alpha \in \Big( \frac{1}{\tilmu} \big( 1 - \frac{1}{\gamma} \big), \frac{1}{\tilL} \Big]$ is valid.
Since $\alpha > \frac{1}{\tilmu} \big( 1 - \frac{1}{\gamma} \big)$, we also obtain that $\beta \sqrt{\gamma} = \sqrt{\gamma} \cdot \sqrt{1 - \alpha \tilmu} < 1$. For both cases, we have that $\beta \sqrt{\gamma} < 1$.
\end{proof}

\subsection{Proof of Proposition~\ref{prop: convergence}}
\label{subsec: convergence proof}

We refactor Proposition~\ref{prop: convergence} into Lemmas~\ref{appx-lem: convergence} and \ref{appx-lem: conv rate} where Lemma~\ref{appx-lem: convergence} mainly captures the final convergence radius and Lemma~\ref{appx-lem: conv rate} captures the convergence rate.

\begin{lemma}  \label{appx-lem: convergence}
Suppose Assumption~\ref{asm: convex} holds.
If algorithm $A$ in R{\scriptsize{ED}}G{\scriptsize{RAF}} satisfies the $( \xc, \gamma, \{ c[k] \} )$-states contraction property (for some $\xc \in \R^d$, $\gamma \in \R_{\geq 0}$, and $\{ c[k] \}_{k \in \N} \subset \R$) and $\alpha_k = \alpha \in \big( 0, \frac{1}{\tilL} \big]$, then for all $k \in \N$ and $v_i \in \vr$,
\begin{equation}
    \| \x_i [k] - \xc \|
    \leq ( \beta \sqrt{\gamma} )^k \max_{v_s \in \vr} 
    \| \x_s [0] - \xc \|  
    + \beta \sum_{s=0}^{k-1} (\beta \sqrt{\gamma})^s c[k-s-1]
    + r_c \sqrt{\alpha \tilL} \sum_{s=0}^{k-1} (\beta \sqrt{\gamma})^s,
    \label{eqn: conv thm}
\end{equation}
where $r_c$ and $\beta$ are defined in \eqref{def: x_c distance} and \eqref{def: beta}, respectively. Furthermore, if $A$ satisfies the reduction property of Type-I or Type-II, then it holds that
\begin{equation}
    \limsup_k \| \x_i [k] - \xc \| \leq \frac{r_c \sqrt{\alpha \tilL}}{1 - \beta \sqrt{\gamma}}
    \quad \text{for all} \quad v_i \in \vr.
    \label{appx eqn: conv thm limit}
\end{equation}
\end{lemma}

\begin{proof}
Consider a regular agent $v_i \in \vr$.
Since $\x_i [k+1] = \Tilde{\x}_i [k] - \alpha_k \g_i [k]$ from \eqref{eqn: gradient step}, we can write
\begin{equation}
    \| \x_i [k+1] - \xc \|^2 
    = \| \Tilde{\x}_i [k] - \xc \|^2  
    - 2 \langle \Tilde{\x}_i [k] - \xc, \; \alpha_k \g_i [k] \rangle 
    + \alpha_k^2 \| \g_i [k] \|^2.
    \label{eqn: expansion}
\end{equation}
Since $f_i$ is $\mu_i$-strongly convex (from Assumption~\ref{asm: convex}), from \eqref{def: strongly convex 1} we have that $- \langle \Tilde{\x}_i [k]$ $- \xc, \g_i [k] \rangle \leq (f_i (\xc) - f_i (\Tilde{\x}_i [k])) - \frac{\mu_i}{2} \| \Tilde{\x}_i [k] - \xc \|^2$.
Substituting this inequality into \eqref{eqn: expansion} we get
\begin{equation}
    \| \x_i [k+1] - \xc \|^2 
    \leq ( 1 - \alpha_k \mu_i ) \| \Tilde{\x}_i [k] - \xc \|^2 
    + \alpha_k^2 \| \g_i [k] \|^2 
    + 2 \alpha_k ( f_i (\xc) - f_i (\Tilde{\x}_i [k]) ).
    \label{eqn: strongly convex sub}
\end{equation}
Since $f_i$ has an $L_i$-Lipschitz gradient (from Assumption~\ref{asm: convex}), from \eqref{def: lipschitz gradient 1} we have that $\| \g_i [k] \|^2 \leq 2 L_i ( f_i (\Tilde{\x}_i [k]) - f_i (\x_i^*) )$.
Substituting this inequality into \eqref{eqn: strongly convex sub} yields
\begin{equation}
    \| \x_i [k+1] - \xc \|^2 
    \leq ( 1 - \alpha_k \mu_i ) \| \Tilde{\x}_i [k] - \xc \|^2 
    - 2 \alpha_k (1 - \alpha_k L_i) f_i (\Tilde{\x}_i [k])  
    - 2 \alpha_k^2 L_i f_i (\x_i^*) + 2 \alpha_k f_i (\xc).
    \label{eqn: lipschitz gradient sub}
\end{equation}
Since $f_i$ has an $L_i$-Lipschitz gradient (from Assumption~\ref{asm: convex}), from \eqref{def: lipschitz gradient 2} we have that $f_i (\xc) \leq f_i (\x_i^*) + \frac{L_i}{2} \| \xc - \x_i^* \|^2$.
Substituting this inequality into \eqref{eqn: lipschitz gradient sub}, we obtain 
\begin{equation*}
    \| \x_i [k+1] - \xc \|^2 
    \leq ( 1 - \alpha_k \mu_i ) \| \Tilde{\x}_i [k] - \xc \|^2 
    - 2 \alpha_k (1 - \alpha_k L_i) (f_i ( \Tilde{\x}_i [k] ) - f_i (\x_i^*) ) 
    + \alpha_k L_i \| \xc - \x_i^* \|^2.
\end{equation*}
Since $\alpha_k = \alpha \in \big( 0, \frac{1}{\tilL} \big]$, $L_i \leq \tilL$, and $\mu_i \geq \tilmu$, the above inequality implies that
\begin{equation}
    \| \x_i [k+1] - \xc \|^2 
    \leq ( 1 - \alpha \tilmu ) \| \Tilde{\x}_i [k] - \xc \|^2
    + \alpha \tilL \| \xc - \x_i^* \|^2.
    \label{eqn: simplified}
\end{equation}
Since the algorithm $A$ satisfies the $( \xc, \gamma, \{ c[k] \} )$-states contraction property given in \eqref{eqn: contraction def}, \eqref{eqn: simplified} becomes
\begin{equation*}
    \| \x_i [k+1] - \xc \|^2 
    \leq ( 1 - \alpha \tilmu ) 
    \Big( \sqrt{\gamma} \max_{v_j \in \vr} \| \x_j [k] - \xc \| 
    + c[k] \Big)^2 
    + \alpha \tilL \max_{v_j \in \vr} \| \xc - \x_j^* \|^2,
\end{equation*}
which implies that
\begin{equation*}
    \| \x_i [k+1] - \xc \|
    \leq \sqrt{\gamma} \cdot \sqrt{1 - \alpha \tilmu} \;
    \max_{v_j \in \vr} \| \x_j [k] - \xc \|  
    + \sqrt{1 - \alpha \tilmu} \; c[k]
    + \sqrt{\alpha \tilL} \max_{v_j \in \vr} \| \xc - \x_j^* \|.
\end{equation*}
Recall the definition of $r_c$ and $\beta$ from \eqref{def: x_c distance} and \eqref{def: beta}, respectively. Since the above inequality holds for all $v_i \in \vr$, taking the maximum over $v_i \in \vr$ yields
\begin{equation*}
    \max_{v_i \in \vr} \| \x_i [k+1] - \xc \|
    \leq \beta \sqrt{\gamma}  \max_{v_i \in \vr} 
    \| \x_i [k] - \xc \| 
    +  \beta c[k] + r_c \sqrt{\alpha \tilL}.
\end{equation*}
Unfolding the recursive inequality above, we obtain that
\begin{equation*}
    \max_{v_i \in \vr} \| \x_i [k] - \xc \|
    \leq ( \beta \sqrt{\gamma} )^k \max_{v_i \in \vr} 
    \| \x_i [0] - \xc \| 
    + \beta \sum_{s=0}^{k-1} (\beta \sqrt{\gamma})^s c[k-s-1]
    + r_c \sqrt{\alpha \tilL} \sum_{s=0}^{k-1} (\beta \sqrt{\gamma})^s,
\end{equation*}
which completes the first part of the proof.

Consider the second part of the lemma. Since the algorithm $A$ satisfies the reduction property of Type-I or Type-II, from Lemma~\ref{lem: reduction} we have that $\beta \sqrt{\gamma} < 1$. Considering the right-hand side (RHS) of \eqref{eqn: conv thm}, since $\lim_{k \to \infty} \sum_{s=0}^k (\beta \sqrt{\gamma})^s$ is finite, using Corollary~\ref{cor: lim sum of prod}, we have
\begin{equation*}
    \lim_{k \to \infty} \bigg[
    r_c \sqrt{\alpha \tilL} \sum_{s=0}^{k-1} (\beta \sqrt{\gamma})^s
    + \beta \sum_{s=0}^{k-1} (\beta \sqrt{\gamma})^s c[k-s-1]  
    + ( \beta \sqrt{\gamma} )^k \max_{v_s \in \vr} \| \x_s [0] - \xc \| \bigg] 
    = \frac{r_c \sqrt{\alpha \tilL}}{1 - \beta \sqrt{\gamma}}.
\end{equation*}
The result \eqref{appx eqn: conv thm limit} follows from taking $\limsup_k$ on both sides of \eqref{eqn: conv thm} and then applying the above equation.
\end{proof}

\begin{lemma}  \label{appx-lem: conv rate}
Suppose Assumption~\ref{asm: convex} holds and an algorithm $A$ in R{\scriptsize{ED}}G{\scriptsize{RAF}} satisfies the reduction property of Type-I or Type-II. If the perturbation sequence $c[k] = \O (\xi^k)$ with $\xi \in (0, 1) \setminus \{ \beta \sqrt{\gamma} \}$, then
\begin{equation}
    \| \x_i [k] - \xc \| \leq R^* + \O \big( ( \max \{ \beta \sqrt{\gamma}, \xi \} )^k \big)
    \quad \text{for all} \quad v_i \in \vr.
    \label{appx eqn: convergence rate}
\end{equation}
\end{lemma}

\begin{proof}
Consider the second term on the RHS of \eqref{eqn: conv thm}. Since $c[k] = \O (\xi^k)$, we obtain that
\begin{equation*}
    \beta \sum_{s=0}^{k-1} (\beta \sqrt{\gamma})^s c[k-s-1]
    = \O \bigg( \xi^{k-1} \sum_{s=0} \Big( \frac{\beta \sqrt{\gamma}}{\xi} \Big)^s \bigg) 
    = \O \bigg( \frac{ (\beta \sqrt{\gamma})^k - \xi^k }{ \beta \sqrt{\gamma} - \xi } \bigg) 
    = \O \big( ( \max \{ \beta \sqrt{\gamma}, \xi \} )^k \big).
\end{equation*}
Using the above equation and the fact that $r_c \sqrt{\alpha \tilL} \cdot \sum_{s=0}^{k-1} (\beta \sqrt{\gamma})^s \leq R^*$, we have that \eqref{eqn: conv thm} implies \eqref{appx eqn: convergence rate}.
\end{proof}

\subsection{Convergence Region Containment of the True Minimizer}
\label{subsec: true minimizer proof}

\begin{lemma}  \label{appx-lem: true minimizer}
Suppose Assumption~\ref{asm: convex} holds, and let $\x^*$ be the minimizer of the function \eqref{prob: regular node}. For a given vector $\xc \in \R^d$, if $\rc = \max_{v_i \in \vr} \| \xc - \x_i^* \|$, then $\x^* \in \B \big( \xc, \frac{\tilL}{\tilmu} \rc \big)$.
Furthermore, if $\gamma \in \Big[ 1, \frac{1}{1 - \frac{\tilmu}{\tilL}} \Big)$ and $\alpha \in \Big( \frac{1}{\tilmu} \big( 1 - \frac{1}{\gamma} \big), \frac{1}{\tilL} \Big]$, then $\x^* \in \B ( \xc, R^* )$.
\end{lemma}

\begin{proof}
Suppose $\x$ is a point in $\R^d$ such that $\| \x - \xc \| > \frac{\tilL}{\tilmu} r_c$. In order to conclude that $\x^* \in \B \Big( \xc, \frac{\tilL}{\tilmu} r_c \Big)$, we will show that $\sum_{v_i \in \vr} \nabla f_i (\x) \neq \boldsymbol{0}$.

In the first step, we will show that $\cos \angle ( \nabla f_i (\x), \x - \xc ) > 0$ for all $v_i \in \vr$.
For a regular agent $v_i \in \vr$, consider the angle between the vectors $\x - \xc$ and $\x - \x_i^*$.
Suppose $r_c > 0$; otherwise, we have that $\angle ( \x - \xc, \x - \x_i^* ) = 0$. Using Lemma~A.1 from the arXiv version of \cite{kuwaran2024scalable}, we can bound the angle as follows:
\begin{equation}
    \angle ( \x - \xc, \x - \x_i^* )
    \leq \max_{\x_0 \in \B (\xc, r_c)} \angle ( \x - \xc, \x - \x_0 )  
    = \arcsin \Big( \frac{r_c}{\| \x - \xc \|} \Big)
    < \arcsin \Big( \frac{\tilmu}{\tilL} \Big).
    \label{eqn: angle bound 1}
\end{equation}
On the other hand, for a regular agent $v_i \in \vr$, since $f_i$ is $\mu_i$-strongly convex, from \eqref{def: strongly convex 2} we have that $\langle \nabla f_i (\x), \x - \x_i^* \rangle \geq \mu_i \| \x - \x_i^* \|^2$
which is equivalent to 
\begin{equation}
    \| \nabla f_i (\x) \| \cos \angle ( \nabla f_i (\x), \x - \x_i^* )
    \geq \mu_i \| \x - \x_i^* \|.
    \label{eqn: str cvx angle}
\end{equation}
Since $f_i$ has an $L_i$-Lipschitz gradient, from \eqref{def: lipschitz gradient 0} we have that $\| \nabla f_i (\x) \| \leq L_i \| \x - \x_i^* \|$.
Substitute this inequality into \eqref{eqn: str cvx angle} to obtain $\cos \angle ( \nabla f_i (\x), \x - \x_i^* ) \geq \frac{\mu_i}{L_i} \geq \frac{\tilmu}{\tilL}$
which implies that
\begin{equation}
    \angle ( \nabla f_i (\x), \x - \x_i^* )
    \leq \arccos \Big( \frac{\tilmu}{\tilL} \Big).
    \label{eqn: angle bound 2}
\end{equation}
Using \eqref{eqn: angle bound 1} and \eqref{eqn: angle bound 2}, we can bound the angle between the vectors $\nabla f_i (\x)$ and $\x - \xc$ as follows:
\begin{equation*}
    \angle ( \nabla f_i (\x), \x - \xc ) 
    \leq \angle ( \nabla f_i (\x), \x - \x_i^* ) + \angle ( \x - \xc, \x - \x_i^* )  
    < \arccos \Big( \frac{\tilmu}{\tilL} \Big) + \arcsin \Big( \frac{\tilmu}{\tilL} \Big)
    = \frac{\pi}{2},
\end{equation*}
where the first inequality is obtained from Corollary~12 in \cite{castano2016angles}. This means that $\cos \angle ( \nabla f_i (\x), \x - \xc ) > 0$ as desired.

In the second step, we will show that $\| \nabla f_i (\x) \| > 0$ for all $v_i \in \vr$.
For a regular agent $v_i \in \vr$, consider the lower bound of the gradient's norm $\| \nabla f_i (\x) \|$ in \eqref{eqn: str cvx angle} which implies that
\begin{equation*}
    \| \nabla f_i (\x) \| 
    \geq \mu_i \| \x - \x_i^* \|
    \geq \mu_i \big( \| \x - \xc \| - \| \x_i^* - \xc \| \big).
\end{equation*}
Since $\| \x - \xc \| > \frac{\tilL}{\tilmu} r_c$, the above inequality becomes
\begin{equation*}
    \| \nabla f_i (\x) \| 
    > \Big( \frac{\tilL}{\tilmu} 
    \max_{v_j \in \vr} \| \x_j^* - \xc  \| 
    - \| \x_i^* - \xc \| \Big)
    \geq 0,
\end{equation*}
where the second inequality is obtained by using $\tilL \geq \tilmu$.

In the last step, we will show that $\sum_{v_i \in \vr} \nabla f_i (\x) \neq \boldsymbol{0}$. 
Consider the following inner product
\begin{equation*}
    \Big\langle \sum_{v_i \in \vr} \nabla f_i (\x), \x - \xc \Big\rangle 
    = \| \x - \xc \| 
    \sum_{v_i \in \vr}
    \| \nabla f_i (\x) \| \cos \angle ( \nabla f_i (\x), \x - \xc ).
\end{equation*}
Since $\| \nabla f_i (\x) \| > 0$ and $\cos \angle ( \nabla f_i (\x), \x - \xc ) > 0$ for all $v_i \in \vr$, and $\| \x - \xc \| > 0$, we have that $\big\langle \sum_{v_i \in \vr} \nabla f_i (\x), \x - \xc \big\rangle > 0$.
This implies that $\sum_{v_i \in \vr} \nabla f_i (\x) \neq \boldsymbol{0}$ which completes the first part of the proof.

For the second part of the lemma, in order to conclude that $\x^* \in \B ( \xc, R^* )$, we will show that $\frac{\tilL}{\tilmu} r_c \leq R^*$, where $R^*$ is defined in \eqref{eqn: conv thm limit}. Since $\gamma \geq 1$, we have that 
\begin{equation*}
    R^* 
    \geq \frac{r_c \sqrt{\alpha \tilL}}{1 - \sqrt{1 - \alpha \tilmu}}.
\end{equation*}
Multiplying $1 + \sqrt{1 - \alpha \tilmu}$ to both the numerator and denominator of the RHS of the above inequality, we obtain that
\begin{equation*}
    R^* 
    \geq r_c \sqrt{\frac{\tilL}{\tilmu}} 
    \bigg( \frac{1}{\sqrt{\alpha \tilmu}} + \sqrt{ \frac{1}{\alpha \tilmu} - 1 } \bigg).
\end{equation*}
Since $\alpha \leq \frac{1}{\tilL}$ implies that $\frac{1}{\alpha \tilmu} \geq \frac{\tilL}{\tilmu}$, we can bound $R^*$ as follows:
\begin{equation*}
    R^* \geq r_c \sqrt{\frac{\tilL}{\tilmu}} \bigg( \sqrt{\frac{\tilL}{\tilmu}} + \sqrt{\frac{\tilL}{\tilmu} - 1} \bigg)
    = r_c \bigg( \frac{\tilL}{\tilmu} + \sqrt{ \frac{\tilL}{\tilmu} \Big( \frac{\tilL}{\tilmu} - 1 \Big) } \bigg).
\end{equation*}
Since $\tilL \geq \tilmu$, we obtain that $R^* \geq \frac{\tilL}{\tilmu} r_c$ which completes the second part of the proof.
\end{proof}
\section{Proof of Consensus Results in Subsection~\ref{subsec: consensus}}

\subsection{Bound on Gradients}
\label{subsec: grad bound}

As we have claimed in the main text, the states contraction property (Definition~\ref{def: contraction}) implies a bound on the gradient $\| \g_i [k] \|_\infty$. The following lemma formally illustrates this fact.

\begin{lemma}  \label{lem: grad bound}
Suppose Assumption~\ref{asm: convex} holds.
If an algorithm $A$ in R{\scriptsize{ED}}G{\scriptsize{RAF}} satisfies the $( \xc, \gamma, \{ c[k] \} )$-states contraction property (for some $\xc \in \R^d$, $\gamma \in \R_{\geq 0}$, and $\{ c[k] \}_{k \in \N} \subset \R$) and $\alpha_k = \alpha \in \big( 0, \frac{1}{\tilL} \big]$, then for all $k \in \N$ and $v_i \in \vr$,
\begin{multline}
    \| \g_i [k] \|_\infty
    \leq \tilL \sqrt{\gamma} \bigg[ ( \beta \sqrt{\gamma} )^k \max_{v_s \in \vr} \| \x_s [0] - \xc \|  \\
    + \beta \sum_{s=0}^{k-1} (\beta \sqrt{\gamma})^s c[k-s-1] 
    + \rc \sqrt{\alpha \tilL} \sum_{s=0}^{k-1} (\beta \sqrt{\gamma})^s \bigg] 
    + \tilL c[k] + \rc \tilL,
    \label{eqn: grad bound}
\end{multline}
where $\rc$ and $\beta$ are defined in \eqref{def: x_c distance} and \eqref{def: beta}, respectively. Furthermore, if $A$ satisfies the reduction property of Type-I or Type-II, then it holds that
\begin{equation}
    \limsup_k \; \| \g_i [k] \|_\infty 
    \leq \rc \tilL \bigg( 1 + \frac{\sqrt{\alpha \gamma \tilL}}{1 - \beta \sqrt{\gamma}} \bigg)
    \quad \text{for all} \quad v_i \in \vr.
    \label{eqn: limsup grad bound}
\end{equation}
\end{lemma}

\begin{proof}
Consider a time-step $k \in \N$, and a regular agent $v_i \in \vr$. We can write 
\begin{equation*}
    \| \Tilde{\x}_i [k] - \x_i^* \| \leq \| \Tilde{\x}_i [k] - \xc \| + \| \x_i^* - \xc \|.
\end{equation*}
Since the algorithm $A$ satisfies the $( \xc, \gamma, \{ c[k] \} )$-states contraction property, applying \eqref{eqn: contraction def} to the above inequality yields
\begin{equation}
    \| \Tilde{\x}_i [k] - \x_i^* \|
    \leq \sqrt{\gamma} \max_{v_j \in \vr} \| \x_j [k] - \xc \| + c[k] + \rc,
    \label{eqn: base distance}
\end{equation}
where $\rc$ is defined in \eqref{def: x_c distance}. Since $\g_i [k] = \nabla f_i ( \Tilde{\x}_i [k] )$, using Assumption~\ref{asm: convex} we can write
\begin{equation*}
    \| \g_i [k] \|
    = \| \nabla f_i ( \Tilde{\x}_i [k] ) - \nabla f_i ( \x_i^* ) \|
    \leq \Tilde{L} \| \Tilde{\x}_i [k] - \x_i^* \|.
\end{equation*}
Substituting \eqref{eqn: base distance} into the above inequality and using the fact that $\| \g_i [k] \|_\infty \leq \| \g_i [k] \|$, we obtain that
\begin{equation}
    \| \g_i [k] \|_\infty \leq \tilL \sqrt{\gamma} \max_{v_j \in \vr} \| \x_j [k] - \xc \| + \tilL c[k] + \rc \tilL.
    \label{eqn: individual grad bound}
\end{equation}
Substituting \eqref{eqn: conv thm} from Lemma~\ref{appx-lem: convergence} into the above inequality yields \eqref{eqn: grad bound}.

To show the second part of the lemma, taking $\limsup_k$ to both sides of \eqref{eqn: individual grad bound}, we have that
\begin{equation*}
    \limsup_k \| \g_i [k] \|_\infty 
    \leq \tilL \sqrt{\gamma} \; \limsup_k \max_{v_j \in \vr} \| \x_j [k] - \xc \| 
    + \tilL \lim_{k \to \infty} c[k] + \rc \tilL.
\end{equation*}
Using \eqref{eqn: conv thm limit} from Proposition~\ref{prop: convergence} and $\lim_{k \to \infty} c[k] = 0$ yields the result \eqref{eqn: limsup grad bound}.
\end{proof}

\subsection{Proof of Proposition~\ref{prop: consensus}}
\label{subsec: proof of consensus proposition}

Here, we present a more general version of Proposition~\ref{prop: consensus} which will be used to prove Theorem~\ref{thm: consensus}.

\begin{proposition}  \label{appx-prop: consensus}
If an algorithm $A$ in R{\scriptsize{ED}}G{\scriptsize{RAF}} satisfies the $( \{ \W^{(\ell)} [k] \}, G )$- mixing dynamics property (for some $\{ \W^{(\ell)}[k] \}_{k \in \N, \; \ell \in [d]} \subset \mathbb{S}^{| \vr |}$ and $G \in \R_{\geq 0}$) and $\alpha_k = \alpha$ for all $k \in \N$, then there exist $\rho \in \R_{\geq 0}$ and $\lambda \in (0, 1)$ such that for all $k \in \N$ and $v_i, v_j \in \vr$, it holds that
\begin{equation}
    \| \x_i [k] - \x_j [k] \| 
    \leq \rho \sqrt{d} \bigg( \lambda^k \max_{v_r \in \vr} \| \x_r[0] \|_\infty  
    + \alpha \sum_{s=0}^{k-1} \lambda^{k-s-1} \max_{v_r \in \vr} \| \g_r [s] \|_{\infty} \bigg).
    \label{eqn: approx consensus}
\end{equation}
Furthermore, there exist $\rho \in \R_{\geq 0}$ and $\lambda \in (0, 1)$ such that 
\begin{equation}
    \limsup_k \| \x_i [k] - \x_j [k] \| \leq \frac{\alpha \rho G \sqrt{d}}{1 - \lambda}
    \quad \text{for all} \quad v_i, v_j \in \vr.
    \label{appx eqn: lim approx consensus}
\end{equation}
\end{proposition}

\begin{proof}
Consider a time-step $k \in \Z_+$ and a dimension $\ell \in [d]$. Since the algorithm $A$ satisfies the $( \{ \W^{(\ell)} [k] \}, G )$-mixing dynamics property in \eqref{def: perturbed mixing}, we have that
\begin{equation}
    \x^{(\ell)} [k]
    = \W^{(\ell)} [k-1] \x^{(\ell)} [k-1] - \alpha_{k-1} \g^{(\ell)} [k-1].
    \label{eqn: mixing dynamics}
\end{equation}
For $s, t \in \N$, let
\begin{equation*}
    \Phil [t,s] = \begin{cases}
    \W^{(\ell)} [t] \W^{(\ell)} [t-1] \cdots \W^{(\ell)} [s] &\text{if} \;\; t \geq s, \\
    \I &\text{if} \;\; t < s.
\end{cases}
\end{equation*}
We can expand \eqref{eqn: mixing dynamics} as follows:
\begin{equation}
    \x^{(\ell)} [k] = \Phil [k-1, 0] \x^{(\ell)} [0] 
    - \sum_{s=0}^{k-1} \alpha_s \Phil [k-1, s+1] \g^{(\ell)} [s].
    \label{eqn: mixing dynamics expanded}
\end{equation}
Let $\boldsymbol{q}^{(\ell)} (s) \in \R^{| \vr |}$ be such that 
$\lim_{t \to \infty} \Phil [t, s] = \boldsymbol{1} \boldsymbol{q}^{(\ell) T} [s]$, 
and let $\Bar{\x}^{(\ell)} [k] = \boldsymbol{1} \boldsymbol{q}^{(\ell) T} [k] \x^{(\ell)} [k]$. We can write
\begin{equation*}
    \| \x^{(\ell)} [k] - \Bar{\x}^{(\ell)} [k] \|_\infty
    = \Big\| \big( \I - \boldsymbol{1} \boldsymbol{q}^{(\ell)T} [k] \big) \x^{(\ell)} [k] \Big \|_\infty.
\end{equation*}
Applying \eqref{eqn: mixing dynamics expanded} to the above equation, we obtain that
\begin{multline}
    \| \x^{(\ell)} [k] - \Bar{\x}^{(\ell)} [k] \|_\infty 
    \leq \big\| \Phil [k-1, 0] 
    - \boldsymbol{1} \boldsymbol{q}^{(\ell)T} [0] \big\|_\infty
    \| \x^{(\ell)} [0] \|_\infty  \\
    + \sum_{s=0}^{k-1} \Big( \alpha_s \big\| \Phil [k-1, s+1]
    - \boldsymbol{1} \boldsymbol{q}^{(\ell)T} [s+1] \big\|_\infty 
    \| \g^{(\ell)} [s] \|_\infty \Big).
    \label{eqn: expanded diff}
\end{multline}
From Proposition~1 in \cite{cao2008reaching}, we have that there exist constants $\rho' \in \R_{\geq 0}$ and $\lambda \in (0, 1)$ such that for all $k > s \geq 0$,
\begin{equation*}
    \Big\| \Phil [k-1, s] - \boldsymbol{1} \boldsymbol{q}^{(\ell)T} [s] \Big\|_\infty
    \leq \rho' \lambda^{k - s}.
\end{equation*}
Thus, applying the above inequality, \eqref{eqn: expanded diff} can be bounded as
\begin{equation}
    \| \x^{(\ell)} [k] - \Bar{\x}^{(\ell)} [k] \|_{\infty}
    \leq \rho' \lambda^k \| \x^{(\ell)} [0] \|_{\infty}  
    + \sum_{s=0}^{k-1} \alpha_s \rho' \lambda^{k-s-1} \| \g^{(\ell)} [s] \|_{\infty}.
    \label{eqn: diff infty}
\end{equation}
Since $\alpha_s = \alpha$ for all $s \in \N$, and for $s \in \N$, $\ell \in [d]$, and $\z^{(\ell)} [s] = \x^{(\ell)} [s]$ or $\g^{(\ell)} [s]$,
\begin{equation*}
    \| \z^{(\ell)} [s] \|_{\infty} 
    \leq \max_{\ell \in [d]} \max_{v_i \in \vr} | z_i^{(\ell)} [s] |
    = \max_{v_i \in \vr} \| \z_i[s] \|_\infty,
\end{equation*}
the inequality \eqref{eqn: diff infty} becomes 
\begin{equation}
    \| \x^{(\ell)} [k] - \Bar{\x}^{(\ell)} [k] \|_{\infty}
    \leq \rho' \lambda^k \max_{v_i \in \vr} \| \x_i[0] \|_\infty 
    + \alpha \rho' \sum_{s=0}^{k-1} \lambda^{k-s-1} \max_{v_i \in \vr} \| \g_i [s] \|_{\infty}.
    \label{eqn: diff infty 2}
\end{equation}
Let $\Bar{x}^{(\ell)} [k] = \boldsymbol{q}^{(\ell)T} [k] \x^{(\ell)} [k]$. For $v_i \in \vr$, we can write
\begin{equation*}
    \| \x_i [k] - \Bar{\x} [k] \|
    = \sqrt{\sum_{\ell \in [d]} | x_i^{(\ell)} [k] - \bar{x}^{(\ell)} [k] |^2} 
    \leq \sqrt{\sum_{\ell \in [d]} \| \x^{(\ell)} [k] - \bar{\x}^{(\ell)} [k] \|^2_\infty}.
\end{equation*}
Using the above inequality, we have that for all $v_i, v_j \in \vr$,
\begin{equation*}
    \| \x_i [k] - \x_j [k] \| 
    \leq \| \x_i [k] - \Bar{\x} [k] \| + \| \x_j [k] - \Bar{\x} [k] \|  
    \leq 2 \sqrt{\sum_{\ell \in [d]} \| \x^{(\ell)} [k] - \bar{\x}^{(\ell)} [k] \|^2_\infty}.
\end{equation*}
Substituting \eqref{eqn: diff infty 2} into the above inequality and letting $\rho = 2 \rho'$, we obtain the result \eqref{eqn: approx consensus}. Taking $\limsup_k$ to both sides of \eqref{eqn: approx consensus}, we have
\begin{equation*}
    \limsup_k \| \x_i [k] - \x_j [k] \|  
    \leq \alpha \rho \sqrt{d} \; \limsup_k \sum_{s=0}^{k-1} \lambda^{k-s-1} \max_{v_r \in \vr} \| \g_r [s] \|_{\infty}.
\end{equation*}
Since for all $v_r \in \vr$, we have $\limsup_k \| \g_r [k] \|_{\infty} \leq G$ from Definition~\ref{def: mixing} and $\lim_{k \to \infty} \sum_{s=0}^k \lambda^k = \frac{1}{1 - \lambda}$, using Lemma~\ref{lem: limsup liminf sum of prod}, the above inequality becomes \eqref{appx eqn: lim approx consensus}.
\end{proof}

\subsection{Proof of Theorem~\ref{thm: consensus}}
\label{subsec: consensus theorem}

\begin{proof}[Proof of Theorem~\ref{thm: consensus}]
From the inequality \eqref{eqn: limsup grad bound} in Lemma~\ref{lem: grad bound}, we have that the algorithm $A$ satisfies the $( \{ \W^{(\ell)} [k] \}, G )$-mixing dynamics property with
\begin{equation*}
    G = \rc \tilL \bigg( 1 + \frac{\sqrt{\alpha \gamma \tilL}}{1 - \beta \sqrt{\gamma}} \bigg).
\end{equation*}  
Substituting $G$ into \eqref{eqn: lim approx consensus} in Proposition~\ref{prop: consensus} yields the result \eqref{eqn: cor lim approx consensus}.

To show the second part of the theorem, consider the expression in the square bracket of \eqref{eqn: grad bound}. Since $c[k] = \O(\xi^k)$ and $\xi \in (0, 1) \setminus \{ \beta \sqrt{\gamma} \}$, we have that
\begin{equation*}
    ( \beta \sqrt{\gamma} )^k \max_{v_s \in \vr} \| \x_s [0] - \xc \| + \beta \sum_{s=0}^{k-1} (\beta \sqrt{\gamma})^s c[k-s-1] 
    + \rc \sqrt{\alpha \tilL} \sum_{s=0}^{k-1} (\beta \sqrt{\gamma})^s 
    \leq R^* + \O \big( ( \max \{ \beta \sqrt{\gamma}, \xi \} )^k \big),
\end{equation*}
where $R^*$ is defined in \eqref{eqn: conv thm limit}. Substituting the above inequality into \eqref{eqn: grad bound}, we obtain that for all $v_i \in \vr$,
\begin{align*}
    \| \g_i [k] \|_\infty 
    &\leq \tilL \sqrt{\gamma} \Big[ R^* + \O \big( ( \max \{ \beta \sqrt{\gamma}, \xi \} )^k \big) \Big] 
    + \O(\xi^k) + \rc \tilL  \\
    &= R^* \tilL \sqrt{\gamma} + \rc \tilL + \O \big( ( \max \{ \beta \sqrt{\gamma}, \xi \} )^k \big).
\end{align*}
Substituting the above inequality into \eqref{eqn: approx consensus} in Proposition~\ref{appx-prop: consensus}, we have that there exist $\rho \in \R_{\geq 0}$ and $\lambda \in (0, 1)$ such that for all $v_i, v_j \in \vr$,
\begin{equation}
    \| \x_i [k] - \x_j [k] \| 
    \leq \rho \sqrt{d} \bigg[ \O (\lambda^k) + \frac{\alpha \tilL}{1 - \lambda} (\rc + R^* \sqrt{\gamma})  
    + \alpha \sum_{s=0}^{k-1} \lambda^{k-s-1} \O \big( ( \max \{ \beta \sqrt{\gamma}, \xi \} )^s \big) \bigg].
    \label{eqn: consensus order temp}
\end{equation}
If $\lambda = \max \{ \beta \sqrt{\gamma}, \xi \}$, we can replace $\lambda$ in \eqref{eqn: consensus order temp} with $\lambda' = \lambda + \epsilon$, where $\epsilon \in (0, 1 - \lambda)$. Thus, without loss of generality, there exist $\rho \in \R_{\geq 0}$ and $\lambda \in (0, 1) \setminus \{ \max \{ \beta \sqrt{\gamma}, \xi \} \}$ such that for all $v_i, v_j \in \vr$, the inequality \eqref{eqn: consensus order temp} holds. 
Consider the last term of \eqref{eqn: consensus order temp}. Since $\lambda \neq \max \{ \beta \sqrt{\gamma}, \xi \}$, we have that
\begin{equation*}
    \sum_{s=0}^{k-1} \lambda^{k-s-1} \O \big( ( \max \{ \beta \sqrt{\gamma}, \xi \} )^s \big) 
    = \O \big( ( \max \{ \beta \sqrt{\gamma}, \xi, \lambda \} )^k \big).
\end{equation*}
Substituting the above equality into \eqref{eqn: consensus order temp} yields the result \eqref{eqn: cor approx consensus}.
\end{proof}
\section{Proof of Algorithms Results in Subsection~\ref{subsec: algorithms}}

\subsection{Proof of Theorem~\ref{thm: contraction}}
\label{subsec: proof of algorithms}

Before proving Theorem~\ref{thm: contraction}, we consider a property of SDMMFD and SDFD \cite{kuwaran2024scalable}. Specifically, for SDMMFD and SDFD, since the dynamics of the estimated auxiliary points $\{ \y_i[k] \}_{\vr} \subset \R^d$ are independent of the dynamics of the estimated solutions $\{ \x_i[k] \}_{\vr} \subset \R^d$, we begin by presenting the convergence results of the estimated auxiliary points $\{ \y_i[k] \}_{\vr}$ from Proposition~1 in \cite{kuwaran2024scalable}.

\begin{lemma} \label{lem: aux convergence}
Suppose the set of estimated auxiliary points $\{ \y_i[k] \}_{\vr}$ follow SDMMFD or SDFD \cite{kuwaran2024scalable}.
Suppose Assumption~\ref{asm: robust} hold, the graph $\mathcal{G}$ is $(2F + 1)$-robust, and the weights $w^{(\ell)}_{ij} [k]$ satisfy Assumption~\ref{asm: weight matrices}.
Then, there exists $c_1 \in \R_{>0}$, $c_2 \in \R_{\geq 0}$, and $\y[\infty] \in \R^d$ with 
\begin{equation}
    y^{(\ell)}[\infty] 
    \in \Big[ \min_{v_i \in \vr} y_i^{(\ell)}[k],
    \max_{v_i \in \vr} y_i^{(\ell)}[k] \Big]
    \label{eqn: asymptotic aux}
\end{equation}
for all $k \in \N$ and $\ell \in [d]$ such that for all $v_i \in \vr$, we have
\begin{equation}
    \| \y_i[k] - \y[\infty] \| < c_1 e^{- c_2 k}.
    \label{eqn: aux exp conv}
\end{equation}
\end{lemma} 
Essentially, the lemma above shows that the estimated auxiliary points $\{ \y_i[k] \}_{\vr}$ converge exponentially fast to a single point called $\y[\infty] \in \R^d$. 

We now provide a proof of Theorem~\ref{thm: contraction}

\begin{proof}[Proof of Theorem~\ref{thm: contraction}]
We first show that each algorithm satisfies the states contraction property with some particular quantities.

Consider a regular agent $v_i \in \vr$ following SDMMFD or SDFD. From Lemma~2 in \cite{kuwaran2024scalable}, we have that 
\begin{equation}
    \| \Tilde{\x}_i [k] - \y_i [k] \| 
    \leq \max_{v_j \in \vr} \| \x_j [k] - \y_i [k] \|.
    \label{eqn: lem2 prev paper}
\end{equation}
From Lemma~\ref{lem: aux convergence}, the limit point $\y [\infty] \in \R^d$ exists, and we can write
\begin{equation*}
    \| \x_j [k] - \y_i [k] \|
    \leq \| \x_j [k] - \y [\infty] \| + \| \y [\infty] - \y_i [k] \|.
\end{equation*}
Substitute the above inequality into \eqref{eqn: lem2 prev paper} to get
\begin{equation}
    \| \Tilde{\x}_i [k] - \y_i [k] \| 
    \leq \max_{v_j \in \vr} \| \x_j [k] 
    - \y [\infty] \| + \| \y [\infty] - \y_i [k] \|. 
    \label{eqn: local z bound}
\end{equation}
Thus, we can write
\begin{align*}
    \| \Tilde{\x}_i [k] - \y [\infty] \| 
    &\leq \| \Tilde{\x}_i [k] - \y_i [k] \| 
    + \| \y_i [k] - \y [\infty] \|  \\
    &\leq \max_{v_j \in \vr} \| \x_j [k] 
    - \y [\infty] \| + 2 \| \y_i [k] - \y [\infty] \|,
\end{align*}
where the last inequality comes from substituting \eqref{eqn: local z bound}. Let $c_1 \in \R_{>0}$ and $c_2 \in \R_{\geq 0}$ be the constants given in Lemma~\ref{lem: aux convergence}. By directly applying \eqref{eqn: aux exp conv} to $\| \y_i [k] - \y [\infty] \|$ in the above inequality, we conclude that the algorithms satisfy $(\y [\infty], 1, \{ 2c_1 e^{-c_2k} \})$-states contraction property.

Consider a regular agent $v_i \in \vr$ following CWTM. Since $v_i$ has at least $2F + 1$ in-neighbors, from Proposition~5.1 in \cite{sundaram2018distributed}, we have
\begin{equation*}
    \Tilde{x}^{(\ell)}_i [k]
    = \sum_{v_j \in (\mathcal{N}^{\text{in}}_i \cap \vr) \cup \{ v_i \}}
    \Tilde{w}^{(\ell)}_{ij} [k]  x^{(\ell)}_j [k],
\end{equation*}
where $\sum_{v_j \in (\mathcal{N}^{\text{in}}_i \cap \vr) \cup \{ v_i \}} \Tilde{w}^{(\ell)}_{ij} [k] = 1$. Thus, we can write
\begin{equation*}
    | \Tilde{x}^{(\ell)}_i [k] - c^{* (\ell)} | 
    \leq \sum_{v_j \in (\mathcal{N}^{\text{in}}_i \cap \vr) \cup \{ v_i \}}
    \Tilde{w}^{(\ell)}_{ij} [k] \; | x^{(\ell)}_j [k] - c^{* (\ell)} |  
    \leq \max_{v_j \in \vr}
    | x^{(\ell)}_j [k] - c^{* (\ell)} |.
\end{equation*}
Using the above inequality, we obtain that 
\begin{equation*}
    \| \Tilde{\x}_i [k] - \c^* \|^2
    = \sum_{\ell \in [d]} | \Tilde{x}^{(\ell)}_i [k] - c^{* (\ell)} |^2 
    \leq \sum_{\ell \in [d]} \max_{v_j \in \vr} 
    | x^{(\ell)}_j [k] - c^{* (\ell)} |^2 
    \leq \sum_{\ell \in [d]} \max_{v_j \in \vr} 
    \| \x_j [k] - \c^* \|^2.
\end{equation*}
Thus, we have that $\| \Tilde{\x}_i [k] - \c^* \| \leq \sqrt{d} \; \max_{v_j \in \vr} \| \x_j [k] - \c^* \|$ which corresponds to the $(\c^*, d, \{ 0[k] \})$-states contraction property in Definition~\ref{def: contraction}.

Finally, consider a regular agent $v_i \in \vr$ following RVO. From \eqref{eqn: convex comb}, we can write
\begin{equation*}
    \Tilde{\x}_i [k] - \c^*
    = \sum_{v_j \in ( \mathcal{N}_i^{\text{in}} \cap \vr ) \cup \{ v_i \}} 
    w_{ij} [k] \; ( \x_j [k] - \c^* ).
\end{equation*}
Since $\sum_{v_j \in ( \mathcal{N}_i^{\text{in}} \cap \vr ) \cup \{ v_i \}} w_{ij} [k] = 1$, we have
\begin{equation*}
    \| \Tilde{\x}_i [k] - \c^* \|
    \leq \sum_{v_j \in ( \mathcal{N}_i^{\text{in}} \cap \vr ) \cup \{ v_i \}} 
    w_{ij} [k] \; \| \x_j [k] - \c^* \|   
    \leq \max_{v_j \in \vr} \| \x_j [k] - \c^* \|,
\end{equation*}
which corresponds to the $(\c^*, 1, \{ 0[k] \})$-states contraction property in Definition~\ref{def: contraction}.

Next, we show that SDMMFD, CWTM and RVO satisfy the mixing dynamics property with some particular quantities. To this end, we need to show that there exists $\{ \W^{(\ell)} [k] \}_{k \in \N, \ell \in [d]} \subset \S^{| \vr |}$ such that the state dynamics of each algorithm can be written as \eqref{def: perturbed mixing}, the sequences of graphs $\{ \mathbb{G}(\W^{(\ell)} [k]) \}_{k \in \N}$ are repeatedly jointly rooted for all $\ell \in [d]$, and there exists $G \in \R_{\geq 0}$ such that
\begin{equation*}
    \limsup_k \| \g_i [k] \|_\infty \leq G 
    \quad \text{for all} \quad v_i \in \vr.
\end{equation*}

For SDMMFD, since in the \texttt{dist\_filter} step, each regular agent removes at most $F$ states and in the \texttt{full\_mm\_filter} step, each regular agent removes at most $2dF$ states, from Proposition~5.1 in \cite{sundaram2018distributed}, for all $k \in \N$, $\ell \in [d]$ and $v_i \in \vr$, the dynamics can be rewritten as
\begin{equation*}
    x^{(\ell)}_i [k+1] 
    = \sum_{v_j \in (\mathcal{N}^{\text{in}}_i \cap \vr) \cup \{v_i\}} 
    \Tilde{w}^{(\ell)}_{ij} [k] \; x^{(\ell)}_j [k] - \alpha_k \; g^{(\ell)}_i [k],
\end{equation*}
where $\Tilde{w}^{(\ell)}_{ii} [k] + \sum_{v_j \in \mathcal{N}^{\text{in}}_i \cap \vr} \Tilde{w}^{(\ell)}_{x,ij} [k] = 1$, and $\Tilde{w}^{(\ell)}_{ii} [k] > \omega$ and at least $| \mathcal{N}^{\text{in}}_i | - (2d + 1)F$ of the other weights are lower bounded by $\frac{\omega}{2}$ (where $\omega$ is defined in Assumption~\ref{asm: weight matrices}). For $k \in \N$, $\ell \in [d]$ and $v_i, v_j \in \vr$, let
\begin{equation*}
    ( \widetilde{\W}^{(\ell)} [k] )_{ij} = \begin{cases}
        \Tilde{w}^{(\ell)}_{ij} [k]  \quad &\text{if} \;\; v_j \in (\mathcal{N}^{\text{in}}_i \cap \vr) \cup \{v_i\},  \\
        0 &\text{otherwise},
    \end{cases}
\end{equation*}
and we can write the dynamics of all regular agents as
\begin{equation}
    \x^{(\ell)} [k+1] = \widetilde{\W}^{(\ell)} [k] \x^{(\ell)} [k] - \alpha_k \g^{(\ell)} [k].
    \label{eqn: regular dynamics2}
\end{equation}
Following the proof of Theorem~6.1 in \cite{sundaram2018distributed}, we have that each regular agent $v_i \in \vr$ removes at most $(2d + 1)F$ incoming edges (including all incoming edges from Byzantine agents). Since the graph $\mathcal{G}$ is $((2 d + 1)F + 1)$-robust, applying Lemma~\ref{lem: rooted subgraph2}, we can conclude that the subgraph $\mathbb{G} (\widetilde{\W}^{(\ell)} [k])$ is rooted for all $k \in \N$ and $\ell \in [d]$ (which implies that the sequence $\big\{ \mathbb{G} (\widetilde{\W}^{(\ell)} [k]) \big\}_{k \in \N}$ is repeatedly jointly rooted for all $\ell \in [d]$).
Since SDMMFD satisfies the states contraction property with $\gamma = 1$, substituting $\gamma = 1$ into the inequality \eqref{eqn: limsup grad bound} in Lemma~\ref{lem: grad bound}, we obtain the mixing dynamics result. 

For CWTM, by noting that the graph $\mathcal{G}$ is $(2 F + 1)$-robust and in the \texttt{cw\_mm\_filter} step, each regular agent removes at most $2F$ states, we can use the same steps as in SDMMFD case to show that there exists $\{ \widetilde{\W}^{(\ell)} [k] \}_{k \in \N, \ell \in [d]}$ such that the dynamics of all regular agents can be written as \eqref{eqn: regular dynamics2} and $\mathbb{G} (\widetilde{\W}^{(\ell)} [k])$ is rooted for all $k \in \N$ and $\ell \in [d]$. Since CWTM satisfies the states contraction property with $\gamma = d$, substituting $\gamma = d$ into the inequality \eqref{eqn: limsup grad bound} in Lemma~\ref{lem: grad bound}, we obtain the mixing dynamics result.

For RVO, from the \texttt{safe\_point} step, for all $k \in \N$ and $v_i \in \vr$, we can write
\begin{equation*}
    \x_i [k+1] = \sum_{v_j \in (\mathcal{N}^{\text{in}}_i \cap \vr) \cup \{v_i\}}
    w_{ij} [k] \x_j[k] - \alpha_k \g_i [k],
\end{equation*}
where $w_{ij} [k]$ is defined as in \eqref{eqn: convex comb}. For $k \in \N$ and $v_i, v_j \in \vr$, let
\begin{equation*}
    ( \W [k] )_{ij} = \begin{cases}
        w_{ij} [k]  \quad &\text{if} \;\; v_j \in (\mathcal{N}^{\text{in}}_i \cap \vr) \cup \{v_i\},  \\
        0 &\text{otherwise},
    \end{cases}
\end{equation*}
and we can write the dynamics of all regular agents as \eqref{def: perturbed mixing} using the above $\W [k]$ for all $\ell \in [d]$. Since $p(d, F) > F$ from the definition of $p$ in Subsection~\ref{subsec: algorithms}, the subgraph $\mathbb{G} (\W [k])$ is rooted for all $k \in \N$ (by Lemma~\ref{lem: rooted subgraph2}). Since RVO satisfies the states contraction property with $\gamma = 1$, substituting $\gamma = 1$ into the inequality \eqref{eqn: limsup grad bound} in Lemma~\ref{lem: grad bound}, we obtain the mixing dynamics result.
\end{proof}

\subsection{Proof of Lemma~\ref{lem: radius-c}}
\label{subsec: r_c bound}

\begin{proof}[Proof of Lemma~\ref{lem: radius-c}]
First, consider the case where the regular agents follow SDMMFD or SDFD \cite{kuwaran2020byzantine_ACC, kuwaran2024scalable}.
From Lemma~\ref{lem: aux convergence}, we have that for all $\ell \in [d]$,
\begin{equation}
    y^{(\ell)}[\infty] 
    \in \Big[ \min_{v_i \in \vr} y_i^{(\ell)}[0],
    \max_{v_i \in \vr} y_i^{(\ell)}[0] \Big].
    \label{eqn: y infty interval}
\end{equation}
Since in the initialization step, we set $\y_i [0] = \hat{\x}_i^*$ for all $v_i \in \vr$, we can rewrite \eqref{eqn: y infty interval} as
\begin{equation*}
    y^{(\ell)}[\infty] 
    \in \Big[ \min_{v_i \in \vr} \hat{x}_i^{*(\ell)},
    \max_{v_i \in \vr} \hat{x}_i^{*(\ell)} \Big].
\end{equation*}
Using the above expression, we can write
\begin{equation}
    y^{(\ell)}[\infty] - c^{*(\ell)}
    \leq \max_{v_i \in \vr} \hat{x}_i^{*(\ell)} - c^{*(\ell)}.
    \label{eqn: x-y relationship}
\end{equation}
Let $v_{i'} \in \vr$ be an agent such that $\hat{x}_{i'}^{*(\ell)} = \max_{v_i \in \vr} \hat{x}_i^{*(\ell)}$. We can rewrite inequality \eqref{eqn: x-y relationship} as
\begin{equation*}
    y^{(\ell)}[\infty] - c^{*(\ell)}
    \leq \big( \hat{x}_{i'}^{*(\ell)} - x_{i'}^{*(\ell)} \big) 
    + \big( x_{i'}^{*(\ell)} - c^{*(\ell)} \big).
\end{equation*}
Since $\| \hat{\x}_{i'}^* - \x_{i'}^* \|_\infty \leq \epsilon^*$ and $\| \x_{i'}^* - \c^* \| \leq r^*$ from the definition of $\c^*$ and $r^*$, the above inequality becomes
\begin{equation*}
    | y^{(\ell)}[\infty] - c^{*(\ell)} |
    \leq \big| \hat{x}_{i'}^{*(\ell)} - x_{i'}^{*(\ell)} \big| 
    + \big| x_{i'}^{*(\ell)} - c^{*(\ell)} \big|
    \leq \epsilon^* + r^*.
\end{equation*}
Applying the above inequality, we have that 
\begin{equation}
    \| \y[\infty] - \c^* \|^2
    = \sum_{\ell \in [d]} | y^{(\ell)}[\infty] - c^{*(\ell)} |^2
    \leq d ( r^* + \epsilon^* )^2.
    \label{eqn: sq-distance y-c}
\end{equation}
Consider a regular agent $v_i \in \vr$. Using inequality \eqref{eqn: sq-distance y-c} and the definition of $\c^*$ and $r^*$ (defined in Subsection~\ref{subsec: asm}), we obtain that
\begin{equation*}
    \| \x_i^* - \y [\infty] \| 
    \leq \| \x_i^* - \c^* \| + \| \c^* - \y [\infty] \|
    \leq \sqrt{d} (r^* + \epsilon^*) + r^*.
\end{equation*}
The result follows from noting that $\xc = \y[\infty]$ for SDMMFD and SDFD.

Now, consider the case where the regular agents follow CWTM \cite{sundaram2018distributed, su2015byzantine, su2020byzantine, fu2021resilient, zhao2019resilient, fang2022bridge} or RVO \cite{park2017fault, abbas2022resilient}. Since in this case, $\xc = \c^*$, the result directly follows from the definition of $\c^*$ and $r^*$.
\end{proof}

\end{document}